\documentclass[12pt]{article}

\usepackage{latexsym}
\usepackage{a4wide}
\usepackage{amscd}
\usepackage{graphics}
\usepackage{amsmath}
\usepackage{amssymb}
\usepackage{esint}
\usepackage{mathrsfs}
\usepackage{amsthm}
\usepackage{bbm}
\usepackage{color}
\usepackage{accents}
\usepackage{enumerate}
\usepackage{mathtools}
\usepackage{wasysym}
\input xy
\xyoption{all}

\newcommand{\N}{\mathbb{N}}                     
\newcommand{\Z}{\mathbb{Z}}                     
\newcommand{\Q}{\mathbb{Q}}                     
\newcommand{\R}{\mathbb{R}}                     
\newcommand{\C}{\mathbb{C}}                     
\newcommand{\T}{\mathbb{T}}                     

\newtheorem{mainthm}{\sc Theorem}           
\newtheorem{maincor}{\sc Corollary}           
\newtheorem{mainrem}{\sc Remark}           
\newtheorem{mainex}{\sc Example}           
\newtheorem{mainprop}{\sc Proposition}           
\newtheorem{thm}{\sc Theorem}[section]            
\newtheorem*{thm*}{\sc Theorem}               
\newtheorem*{cor*}{\sc Corollary}        
\newtheorem{lem}[thm]{\sc Lemma}            
\newtheorem{prop}[thm]{\sc Proposition}     
\newtheorem{rem}[thm]{\sc Remark}           
\newtheorem{claim}[thm]{\sc Claim}

\title{Symplectic capacities of domains close to the ball and Banach--Mazur geodesics in the space of contact forms}
\date{}
\author{Alberto Abbondandolo\footnote{Fakult\"at f\"ur Mathematik, Ruhr-Universit\"at Bochum, Universit\"atsstra{\ss}e 150, 44801 Bochum, Germany, alberto.abbondandolo@rub.de}, Gabriele Benedetti\footnote{Department of Mathematics, Vrije Universiteit Amsterdam, De Boelelaan 1111, 1081 HV Amsterdam, g.benedetti@vu.nl}, and Oliver Edtmair\footnote{Department of Mathematics, University of California at Berkeley, Berkeley, CA, 94720, USA, oliver\_edtmair@berkeley.edu}}

\begin{document}

\maketitle

\begin{abstract}
We prove that all normalized symplectic capacities coincide on smooth domains in $\C^n$ which are $C^2$-close to the Euclidean ball, whereas this fails for some smooth domains which are just $C^1$-close to the ball. We also prove that all symplectic capacities whose value on ellipsoids agrees with that of the $n$-th Ekeland--Hofer capacity coincide in a $C^2$-neighborhood of the Euclidean ball of $\C^n$. These results are deduced from a general theorem about contact forms which are $C^2$-close to Zoll ones, saying that these contact forms can be pulled back to suitable ``quasi-invariant'' contact forms. We relate all this to the question of the existence of minimizing geodesics in the space of contact forms equipped with a Banach--Mazur pseudo-metric. Using some new spectral invariants for contact forms, we prove the existence of minimizing geodesics from a Zoll contact form to any contact form which is $C^2$-close to it.
This paper also contains an appendix in which we review the construction of exotic ellipsoids by the Anosov--Katok conjugation method, as these are related to the above mentioned pseudo-metric.
\end{abstract}

\section*{Introduction}

\paragraph{Normalized symplectic capacities.} Let 
\[
\omega_0 = \sum_{j=1}^n dx_j\wedge dy_j
\]
be the standard symplectic form on $\C^n$, endowed with coordinates $z_j=x_j+ i y_j$, for $j=1,\dots,n$. In this paper, by a symplectic capacity on $(\C^n,\omega_0)$ we mean a function
\[
c: \{ \mbox{open subsets of } \C^n \} \rightarrow [0,+\infty]
\]
satisfying the following conditions:
\begin{itemize}
\item (monotonicity) if there exists a symplectomorphism $\phi: \C^n \rightarrow \C^n$ such that $\phi(A)\subset A'$, then $c(A)\leq c(A')$;
\item (conformality) $c(rA) = r^2 c(A)$ for every $r>0$;
\item (non-triviality)  $c(B)>0$ and $c(Z)<\infty$.
\end{itemize}
Here, $B$ denotes the unit Euclidean ball in $\C^n$ and $Z$ the cylinder $\{z\in \C^n \mid |z_1|<1\}$. The symplectic capacity $c$ is said to be normalized if the non-triviality condition is upgraded by the following requirement:
\begin{itemize}
\item (normalization)  $c(B) = c(Z) = \pi$.
\end{itemize}
Symplectic capacities were introduced by Ekeland and Hofer in \cite{eh89} building on Gromov's work \cite{gro85}. What we here call a symplectic capacity on $(\C^n,\omega_0)$ is called by some authors a {\em relative} or {\em extrinsic} symplectic capacity, whereas a symplectic capacity should be {\em intrinsic}, meaning that the monotonicity axiom should hold more generally for all symplectic embeddings of $A$ into $A'$, see \cite[Chapter 12]{ms17}. With this terminology, any symplectic capacity is a relative symplectic capacity, but there are interesting relative symplectic capacities, such as Hofer's displacement energy from \cite{hof90b}, which are not symplectic capacities. In this paper, we wish to consider also these more general functions and hence work with the above weaker definition, but, to avoid to burden the terminology, we omit the word ``relative''. 

Gromov's non-squeezing theorem implies that the functions
\[
\begin{split}
\underline{c}(A) &\coloneqq \sup \{ \pi r^2 \mid \exists \mbox{ symplectomorphism $\phi:\C^n \rightarrow \C^n$ such that } \phi(rB) \subset A\}, \\
\overline{c}(A) &\coloneqq \inf \{ \pi r^2 \mid \exists \mbox{ symplectomorphism $\phi:\C^n \rightarrow \C^n$ such that } \phi(A) \subset rZ \},
\end{split}
\] 
are normalized symplectic capacities. They are known as {\em ball capacity} (or Gromov width) and {\em cylindrical capacity}, respectively. From the above axioms, one easily sees that they are the smallest and largest normalized capacities: any normalized symplectic capacity $c$ satisfies the inequalities
\begin{equation}
\label{BcZ}
\underline{c} \leq c \leq \overline{c}.
\end{equation}
The ball and cylindrical capacities are difficult to compute, as this would involve understanding symplectic embeddings, which is precisely what one would use a symplectic capacity for. Over the last decades, several more computable symplectic capacities based on periodic Hamiltonian orbits, pseudo-holomorphic curves, Lagrangian submanifolds, or a combination of these ingredients have been  constructed. See the survey \cite{chls07} and references therein. Many of these symplectic capacities are {\em spectral}, meaning that the symplectic capacity of a bounded domain $A$ with smooth boundary of restricted contact type coincides with the action of some closed characteristic on $\partial A$ (these hypersurfaces always admit closed characteristics, as first proven in \cite{vit87b}).

It is important to observe that the normalization axiom does not determine $c$ uniquely. Indeed, different normalized symplectic capacities may have different values on domains belonging to the set
\[
\begin{split}
\mathcal{A} \coloneqq \{ & \mbox{bounded open neighborhoods of the origin in $\C^n$ whose boundary is smooth} \\ & \mbox{and transverse to the radial direction} \}.
\end{split}
\]
The first examples are due to Hermann, who in \cite{her98} constructed Reinhardt domains in $\mathcal{A}$ with arbitrarily small volume, and hence arbitrarily small ball capacity, but large displacement energy, and hence large cylindrical capacity. Examples of domains in $\mathcal{A}$ whose cylindrical capacity is strictly larger than the displacement energy are constructed in \cite{ham02}. Examples of domains  whose ball capacity does not coincide with any spectral capacity are given by the domains in $\mathcal{A}$
with fixed volume and arbitrarily large systole which are constructed in \cite{abhs18} and \cite{sag21}. Here, the systole of $A\in \mathcal{A}$ is the positive number
\[
\mathrm{sys}(A)\coloneqq \min \{ \mbox{action of closed characteristics on } \partial A\}.
\]

A long standing open question is whether all normalized symplectic capacities coincide on bounded convex domains, see e.g.~\cite[Conjecture 1.9]{her98} and \cite[Section 14.9, Problem 53]{ms17}. A positive answer to this question would imply that for any normalized symplectic capacity and any bounded convex domain $C$ the inequality
\begin{equation}
\label{viterbo}
c(C)^n \leq n! \, \mathrm{vol}(C)
\end{equation}
holds, as the above inequality trivially holds for the ball capacity. Inequality \eqref{viterbo} was conjectured by Viterbo in \cite{vit00} and, for this reason, the coincidence of all normalized symplectic capacities on bounded convex domains is in the recent literature referred to as the strong Viterbo conjecture. Note that the equality holds in \eqref{viterbo} when $C$ is symplectomorphic to a ball, and Viterbo actually conjectured that these are the only bounded convex domains for which the equality holds (this part of Viterbo's conjecture would not follow from the coincidence of all normalized capacities on bounded convex domains). See \cite{ost14} and \cite{ghr22} for more information about these conjectures and their consequences in convex geometry.

Evidence for a positive answer to the strong Viterbo conjecture is given by the fact that for many normalized symplectic capacities $c$ it has been shown that
\begin{equation}
\label{c=sys}
c(C) = \mathrm{sys}(C)
\end{equation}
whenever $C$ is a bounded convex domain with smooth boundary. Identity \eqref{c=sys} has been proven  in \cite{vit89} for the Ekeland--Hofer capacity from \cite{eh89}, for the Hofer--Zehnder capacity in \cite{hz90},  and in \cite{ak22,iri22} for the capacity from symplectic homology introduced in \cite{fhw94}, which by Hermann's work \cite{her04} coincides also with Viterbo's generating functions capacity from \cite{vit92}.

Thanks to \eqref{BcZ}, proving the equivalence of all normalized capacities on bounded convex domains amounts to prove that $\underline{c} = \overline{c}$ for these domains, for instance by showing that both these capacities satisfy \eqref{c=sys}. For an arbitrary  bounded convex domain $C$, it is not known whether \eqref{c=sys} holds for $\underline{c}$ or for $\overline{c}$.

In this paper, we restrict our analysis to smooth convex domains that are $C^k$-close to the ball $B$. More precisely, we observe that the domains in $\mathcal{A}$ are precisely the sets of the form
\[
A_f \coloneqq \{ rz \mid z\in S^{2n-1}, \; 0\leq r < f(z) \},
\]
where $f$ is an arbitrary smooth positive function on the unit sphere $S^{2n-1}$ of $\C^n$, and  define the $C^k$-distance between domains in $\mathcal{A}$ as the $C^k$-distance between the corresponding smooth functions on $S^{2n-1}$. With the notation above, $B=A_1$.

Our first main result is the following theorem.

\begin{mainthm}
\label{1stcap}
There exists a $C^2$-neighborhood $\mathcal{B}$ of $B$ in $\mathcal{A}$ such that for every $A$ in $\mathcal{B}$ the following facts hold:
\begin{enumerate}[(i)]
\item There is a symplectomorphism $\phi: \C^n \rightarrow \C^n$ mapping $A$ into the cylinder  whose systole coincides with the systole of $A$. In particular, $\overline{c}(A)=\mathrm{sys}(A)$.
\item There is a symplectomorphism $\phi: \C^n \rightarrow \C^n$ mapping the ball whose systole coincides with the systole of $A$ into $A$.  In particular, $\underline{c}(A)=\mathrm{sys}(A)$.

\end{enumerate}
Consequently, all normalized symplectic capacities coincide on $\mathcal{B}$.
\end{mainthm}

Statement (i) is an improvement of Proposition 4 from \cite{ab23}, where the analogous result is proven for domains that are $C^3$-close to the ball. The proof of (i) (as that of \cite[Proposition 4]{ab23}) is elementary. It uses generating functions together with ideas from \cite{edt22} for constructing symplectic embeddings starting with suitable Hamiltonian diffeomorpisms.  In \cite{edt22}, one of us showed that in dimension four the equality $\overline{c} = \mathrm{sys}$ holds also under non-perturbative assumptions: indeed, this equality holds whenever the boundary of the smooth bounded convex domain $C\subset \C^2$ admits a closed characteristic of minimal action which is unknotted as a closed curve on $\partial C\cong S^3$ and has self-linking number -1. The latter condition holds for domains that are $C^2$-close to $B$, but it may actually hold for any $A$ in $\mathcal{A}$ which is convex, see the open question concluding the introduction in \cite{hwz98}.

Statement (ii) is more subtle. In dimension four, \cite{edt22} proves the equality $\underline{c}=\mathrm{sys}$ on a $C^3$-neighborhood of the ball using global surfaces of section for the characteristic foliation of $\partial A$. The novelty here is that we shall prove this in all dimensions, while at the same time improving the closeness assumption from $C^3$ to $C^2$. Note that domains that are sufficiently $C^2$-close to the ball are convex, whereas any $C^1$-neighborhood of $B$ contains non-convex domains. This is consistent with the following statement which implies that the $C^2$-closeness assumption in Theorem \ref{1stcap} (ii) is optimal.

\begin{mainprop}
\label{C^1counter}
There exists a smooth family of domains $\{A_{\lambda}\}_{\lambda\in (0,\lambda_0]}$ in $\mathcal{A}$ which $C^1$-converges to the ball $B$ for $\lambda\rightarrow 0$, and such that
\begin{equation}
\label{counter}
\underline{c}(A_{\lambda})  <  \mathrm{sys}(A_{\lambda})\leq \overline{c}(A_{\lambda}) \qquad \forall \lambda\in (0,\lambda_0].
\end{equation}
\end{mainprop}

Here, the second inequality holds for any domain in $\mathcal{A}$, just because there exist normalized capacities which are spectral. So the new statement is the first strict inequality.

In the proof of Theorem \ref{1stcap} (ii), we exploit the $S^1$-symmetry of the ball. Recall that if the domain $A\in \mathcal{A}$ is convex and invariant under the $S^1$-action on $\C^n$ given by multiplication by unitary complex numbers, then the equivalence of all normalized symplectic capacities for $A$ follows from a short argument of Ostrover: by the $S^1$-invariance, the largest ball centered at the origin and contained in $A$ touches $\partial A$ along a closed characteristic, proving that $\underline{c}(A)\geq \mathrm{sys}(A)$. By convexity, $A$ is contained in the cylinder over the disk spanned by this closed characteristic, proving that $\overline{c}(A)\leq \mathrm{sys}(A)$, and the equivalence of all normalized symplectic capacities for $A$ follows from \eqref{BcZ}. Under a stronger symmetry assumption, namely the invariance under the $n$-torus action given by multiplication of each coordinate in $\C^n$ by unitary complex numbers, the equivalence of all normalized symplectic capacities holds also by replacing convexity by a weaker monotonicity property (see \cite{ghr22} for the 4-dimensional case and \cite{cgh23} for the general case). 

We will prove Theorem \ref{1stcap} (ii) by showing that any $A\in \mathcal{A}$ which is $C^2$-close to $B$ is symplectomorphic to a domain $A'$ which, albeit not being $S^1$-invariant, is such that the largest ball centered at the origin and contained in $A'$ touches $\partial A'$ along a closed characteristic. This fact will be deduced from a general theorem on ``quasi-invariant'' contact forms, which is described further down in this introduction. The same theorem implies that the smallest ball centered at the origin and containing $A'$ touches $\partial A'$ along a closed characteristic. The latter fact has a consequence for higher symplectic capacities, which we now describe. 

\paragraph{Higher symplectic capacities.} In \cite{eh90}, Ekeland and Hofer used the $S^1$-invariance of the Hamiltonian action functional in order to define a sequence of symplectic capacities, only the first of which is normalized (it is the already mentioned Ekeland--Hofer capacity from \cite{eh89}). In \cite{gh18}, Gutt and Hutchings  introduced a sequence of symplectic capacities using $S^1$-equivariant symplectic homology, the first of which is the already mentioned normalized symplectic capacity from \cite{fhw94}. Conjecturally, these two sequences of symplectic capacities coincide. 

Let $\widehat{c}_k$ be either the $k$-th Ekeland--Hofer capacity or Gutt--Hutchings capacity. The considerations below hold for both of them, as these symplectic capacities do coincide on ellipsoids. Indeed, any ellipsoid in $\C^n$ is linearly symplectic equivalent to an ellipsoid of the form
\[
E(a_1,\dots,a_n) \coloneqq \Bigl\{ (z_1,\dots,z_n)\in \C^n\ \Big|\ {\textstyle \sum_{j=1}^n} {\textstyle \frac{\pi}{a_j} |z_j|^2} < 1 \Bigr\}, 
\]
where the $a_j$ are positive numbers, and 
\[
\begin{split}
\widehat{c}_k(E(a_1,\dots,a_n)) = & \mbox{ $k$-th element in the sequence obtained by ordering the list} \\ & \bigl(ha_j\bigr)_{h\in \N, j=1,\dots,n}
\mbox{ in increasing order, allowing repetitions,}
\end{split}
\]
as proven in \cite[Proposition 4]{eh90} and \cite[Example 1.8]{gh18}.
Building on this, we shall say that a symplectic capacity $c$ is {\em $k$-normalized} if it satisfies the following condition:

\begin{itemize}
\item ($k$-th normalization) $c(E)=\widehat{c}_k(E)$ for every ellipsoid $E$.
\end{itemize}

Normalized symplectic capacities, as defined at the beginning of this introduction, are 1-normalized. The $k$-th Ekeland--Hofer and Gutt--Hutchings symplectic capacities are $k$-normalized. The functions
\[
\begin{split}
\underline{c}_k(A) &\coloneqq \sup\{ \widehat{c}_k(E) \mid E \mbox{ ellipsoid, } \exists \phi:\C^n \rightarrow \C^n \mbox{ symplectomorphisms with } \phi(E) \subset A  \},
\\
\overline{c}_k(A) &\coloneqq \inf\{ \widehat{c}_k(E) \mid E \mbox{ ellipsoid, } \exists \phi:\C^n \rightarrow \C^n \mbox{ symplectomorphisms with } \phi(A) \subset E \},
\end{split} 
\]
are $k$-normalized capacities, and any $k$-normalized capacity $c$ satisfies
\[
\underline{c}_k \leq c \leq \overline{c}_k.
\]
The coincidence of all $k$-normalized capacities for the bounded open set $A\subset \C^n$ is equivalent to the equality $\underline{c}_k(A)=\overline{c}_k(A)$ and hence to the fact that $A$ contains the symplectic image of an ellipsoid $E$ and can be symplectically embedded into an ellipsoid $E'$ such that the difference $\widehat{c}_k(E') - \widehat{c}_k(E)$ is arbitrarily small. Our next result concerns the case $k=n$, where $2n$ is the dimension of the symplectic vector space we are considering.

\begin{mainthm}
\label{nthcap}
There exists a $C^2$-neighborhood $\mathcal{B}$ of $B$ in $\mathcal{A}$ such that for any $A$ in $\mathcal{B}$ and any $n$-normalized symplectic capacity $c$ the following facts hold:
\begin{enumerate}[(i)]
\item There exists an ellipsoid $E$ with $c(E)=c(A)$ and a symplectomorphism $\phi: \C^n \rightarrow \C^n$ such that $\phi(E) \subset A$.
\item There exists a symplectomorphism $\psi: \C^n \rightarrow \C^n$ mapping $A$ into the ball $B'$ such that $c(B') = c(A)$.
\end{enumerate}
In particular, all $n$-normalized symplectic capacities coincide on $\mathcal{B}$.
\end{mainthm}

It is natural to ask whether all $k$-normalized capacities coincide on bounded convex domains of $\C^n$. As we shall see, the answer is negative for most pairs $(n,k)$, but besides the pairs $(n,1)$ (corresponding to the strong Viterbo conjecture) there are other pairs for which we do not know the answer. The next result shows that all 2-normalized capacities of polydiscs in $\C^2$ coincide and that this is not true anymore for $k\geq 3$.

\begin{mainprop}
\label{polydiscs}
Let $P(a,b)$ be the following four-dimensional polydisk:
\[
P(a,b) \coloneqq \{ (z_1,z_2) \in \C^2 \mid {\textstyle \frac{\pi}{a}} |z_1|^2 < 1, \; {\textstyle \frac{\pi}{b}} |z_2|^2 < 1 \},
\]
where $a,b>0$. Then:
\begin{enumerate}[(i)]
\item All 2-normalized symplectic capacities coincide on $P(a,b)$: if $c$ is such a capacity, then
\[
c(P(a,b)) = 2 \min\{ a,b\}.
\]
\item For every $k\in \N$ we have
\[
\underline{c}_k(P(1,1))= \widehat{c}_k(E(1,2)) \leq k = \overline{c}_k(P(1,1)),
\]
and the inequality is strict if and only if $k\geq 3$.
\end{enumerate}
\end{mainprop}

More generally, Gutt and Ramos have very recently proved that all 2-normalized capacities coincide on convex toric domains in $\C^2$, see \cite{gr23}. They also showed that when $k\geq \max\{n,3\}$ the following strict inequality holds 
\[
\underline{c}_k(P(1,\dots,1)) < \widehat{c}_k(P(1,\dots,1)) = k,
\]
where 
\[
P(1,\dots,1)  \coloneqq \{ z \in \C^n \mid \pi |z_j|^2 < 1 \; \forall j=1,\dots,n \}
\]
is the equilateral polydisc in $\C^n$. Therefore, for these pairs $(n,k)$ the $k$-normalized capacities do not coincide on bounded convex domains. In particular, the open $C^2$-neighborhood $\mathcal{B}$ of Theorem \ref{nthcap} does not contain all (smooth bounded) convex domains.

These results suggest that the case $k=n=2$ may be somehow special. As far as we know, it is indeed possible that all 2-normalized capacities coincide on bounded convex domains of $\C^2$. Notice that this would imply the following inequality
\begin{equation}
\label{viterbo-2}
\widehat{c}_2(A)^2 \leq 4 \, \mathrm{vol}(A)
\end{equation}
since the ellipsoid $E(1,2)$ maximizes $\widehat{c}_2$ among all ellipsoids with the same volume. 
Some evidence for the validity of \eqref{viterbo-2} comes from a recent paper of Baracco, Bernardi, Lange and Mazzucchelli: in \cite{bblm23}, the authors studied the Ekeland--Hofer capacities on smooth starshaped domains in $\C^2$ and for each $k$ identified the local maximizers of the ratio
\[
\frac{\widehat{c}_k(A)^2}{\mathrm{vol}(A)} 
\]
with respect to the $C^3$-topology on $\mathcal{A}$. In particular, their analysis shows that the inequality \eqref{viterbo-2} holds for every $A$ in a $C^3$-neighborhood of $E(1,2)$, with equality if and only if $A$ is symplectomorphic to a rescaled copy of the ellipsoid $E(1,2)$.

\medskip 

The proofs of Theorems \ref{1stcap} and \ref{nthcap} are based on a general result about contact forms which are close to Zoll ones, which might have independent interest and is described in the next part of this introduction.

\paragraph{Quasi invariant contact forms.}  Let $M$ be a $(2n-1)$-dimensional closed manifold, $n\geq 1$, and $\alpha$ a contact form on $M$, that is, a smooth 1-form such that $\alpha\wedge d\alpha^{n-1}$ is nowhere vanishing. The corresponding Reeb vector field $R_{\alpha}$ is defined by the identities
\[
\imath_{R_{\alpha}} d\alpha =0, \qquad \imath_{R_{\alpha}} \alpha=1.
\]
The flow of $R_{\alpha}$ is called the Reeb flow of $\alpha$. The set of periods of the closed orbits of the Reeb flow of $\alpha$ is closed in $\R$ and we denote by $\mathrm{sys}(\alpha)>0$ its minimum (setting $\mathrm{sys}(\alpha)\coloneqq+\infty$ if $R_{\alpha}$ has no closed orbits, a possibility which never occurs if the Weinstein conjecture holds true). If $A$ belongs to the set $\mathcal{A}$ defined above, then the standard primitive 
\begin{equation}
\label{lambda_0}
\lambda_0\coloneqq \frac{1}{2} \sum_{j=1}^n ( x_j \, dy_j - y_j\, dx_j)
\end{equation}
of $\omega_0$ restricts to a contact form on $\partial A$, the Reeb vector field of $\lambda_0|_{\partial A}$ spans the characteristic distribution of $\partial A$, and we have the identity
\[
\mathrm{sys}(A) = \mathrm{sys}(\lambda_0|_{\partial A}),
\]
which justifies the choice of a common notation. 

The contact form $\alpha$ is said to be Zoll if all the orbits of its Reeb flow are closed and have the same minimal period, that is, if the Reeb flow of $\alpha$ defines a free $S^1$-action on $M$. We shall deduce Theorems \ref{1stcap} and \ref{nthcap} from the following theorem. 

\begin{mainthm}
\label{qi}
Let $\alpha_0$ be a Zoll contact form on the closed manifold $M$. Then for every $\epsilon>0$ there is a $\delta>0$ such that, for any contact form $\alpha$ on $M$ with $\|\alpha-\alpha_0\|_{C^2} < \delta$, there exists a diffeomorphism $\varphi: M \rightarrow M$ with the property that
\begin{equation}
\label{tesi}
\varphi^* \alpha = T e^h \alpha_0,
\end{equation}
where:
\begin{enumerate}[(a)]
\item $T$ is a smooth positive function on $M$ invariant under the free $S^1$-action defined by the Reeb flow of $\alpha_0$;
\item $h$ is a smooth real function on $M$ such that $h$ and $dh$ vanish on the critical set of $T$;
\item $\displaystyle{\min_M T e^h = \min_M T}$ and $\displaystyle{\max_M T e^h = \max_M T}$;
\item any closed orbit  of $R_{\alpha}$ is either {\em long}, meaning that its minimal period is larger than $\frac{1}{\epsilon}$, or {\em short}, meaning that its minimal period belongs to the interval $(\mathrm{sys}(\alpha_0) - \epsilon, \mathrm{sys}(\alpha_0) + \epsilon)$;  
\item the short closed orbits of $R_{\alpha}$ are precisely the images via $\varphi$ of those orbits $\gamma$ of the free $S^1$-action defined by the Reeb flow of $\alpha_0$ consisting of critical points of $T$, and the minimal period of such an orbit is $\mathrm{sys}(\alpha_0) T(\gamma)$.
\end{enumerate}
Moreover, the map $\alpha\mapsto (\varphi,T,h)$ is smooth and maps $\alpha_0$ to $(\mathrm{id},1,0)$.
\end{mainthm}

The smoothness mentioned at the end of the statement is meant in the diffeological sense: if $\{\alpha_t\}_{t\in \R^k}$ is a smooth family of contact forms having $C^2$-distance less than $\delta$ from $\alpha_0$, then the diffeomorphism $\varphi_t$ and the functions $T_t$ and $h_t$ which are given by the above theorem depend smoothly on $t\in \R^k$. This immediately implies that if $\ker \alpha=\ker \alpha_0$ then the diffeomorphism $\varphi$ belongs to the identity component in the contactomorphism group of $(M,\ker \alpha_0)$.

By statement (e), the $S^1$-invariant function $T$ in the above theorem gives us a finite dimensional variational principle for detecting the short closed Reeb orbits of $\alpha$ and, together with (c), implies the identity
\[
\mathrm{sys}(\alpha) =  \mathrm{sys}(\alpha_0) \min_M T = \mathrm{sys}(\alpha_0) \min_M T e^h.
\]
An immediate consequence of Theorem \ref{qi} is then the inequality
\begin{equation}
\label{sysine}
\frac{\mathrm{sys}(\alpha)^n}{\mathrm{vol} (M, \alpha\wedge d\alpha^{n-1})} = \mathrm{sys}(\alpha_0)^n \frac{\min_M T^n e^{nh}}{\int_M T^n e^{nh} \,  \alpha_0\wedge d\alpha_0^{n-1} } \leq \frac{\mathrm{sys}(\alpha_0)^n}{ \mathrm{vol}(M,\alpha_0\wedge d\alpha_0^{n-1})},
\end{equation}
which gives us the following corollary.

\begin{maincor}
\label{systineq}
Zoll contact forms on the closed manifold $M$ are $C^2$-local maximizers of the systolic ratio
\[
\rho_{\mathrm{sys}}(\alpha) \coloneqq \frac{\mathrm{sys}(\alpha)^n}{\mathrm{vol} (M, \alpha\wedge d\alpha^{n-1})},
\]
and if $\alpha$ is $C^2$-close to the Zoll contact form $\alpha_0$ and satisfies $\rho_{\mathrm{sys}}(\alpha)=\rho_{\mathrm{sys}}(\alpha_0)$, then there is a diffeomorphism $\varphi:M \rightarrow M$ such that $\varphi^* \alpha = c\, \alpha_0$ for some constant $c>0$.
\end{maincor}

The latter assertion follows from the fact that if the equality holds in \eqref{sysine} then the function $T$ is constant and hence by Theorem \ref{qi} (b) and \eqref{tesi} the diffeomorphism $\varphi$ has the required property. Under a stronger closeness assumption, this systolic inequality is proven in \cite{ab23} using arguments which will also be used in the proof of Theorem \ref{qi}. 

A modification of the proof of Proposition \ref{C^1counter} shows that the $C^2$-closeness assumption is sharp in Corollary \ref{systineq}.

\begin{mainprop}
\label{C^1counter2}
For every Zoll contact form $\alpha_0$ there exists a smooth family of contact forms $\{\alpha_{\lambda}\}_{\lambda\in (0,\lambda_0]}$ which $C^1$-converges to $\alpha_0$ for $\lambda\to0$, and satisfies 
\[
\rho_{\mathrm{sys}}(\alpha_{\lambda}) > \rho_{\mathrm{sys}}(\alpha_0) \qquad \forall \lambda\in (0,\lambda_0].
\] 
\end{mainprop}

In this paper, our interest in Theorem \ref{qi} is that it allows us to construct symplectic embedding of balls and into balls, thus proving Theorems \ref{1stcap} and \ref{nthcap}. 

Contact forms of the form $T e^h \alpha_0$ with $\alpha_0$ Zoll and $T$, $h$ satisfying the conditions (a)-(e) of Theorem \ref{qi} might be called ``quasi-invariant'', the ``quasi'' referring to the presence of the non-$S^1$-invariant function $e^h$. Note the similarity with the notion of quasi-autonomous Hamiltonian: a compactly supported time-dependent Hamiltonian $H=H(t,x)$ on a symplectic manifold $(M,\omega)$ is said to be quasi-autonomous if the sets of maximizers and minimizers of the functions $H(t,\cdot)$ do not depend on $t$. As shown in \cite{bp94} for $(\C^n,\omega_0)$ and in \cite{lmd95} for closed symplectic manifolds, any compactly supported Hamiltonian diffeomorphism which is $C^1$-close to the identity is generated by a quasi-autonomous Hamiltonian. 
The proof of this fact is based on the Weinstein tubular neighborhood theorem and on a version of the Hamilton--Jacobi equation. 

Theorem \ref{qi} seems much more subtle. Our proof is based on a normal form for contact forms which are closed to Zoll ones (proved by two of us in \cite{ab23}), and uses averaging techniques and Moser's homotopy argument.

Paths in the Hamiltonian group which are generated by quasi-autonomous Hamiltonians are minimizing geodesics in the Hofer metric, see again \cite{bp94} and \cite{lmd95}. Results of the same flavour hold also for a Banach--Mazur pseudo-metric on the space of contact forms on a given contact manifold. In the next subsection, we discuss these results.

\paragraph{A contact Banach--Mazur pseudo-metric.} 
Let $\xi$ be a co-oriented contact structure on the closed manifold $M$. We denote by $\mathcal{F}(M,\xi)$ the set of contact forms $\alpha$ defining $\xi$, meaning that $\ker \alpha=\xi$ and $\alpha$ is positive on tangent vectors which are positively transverse to $\xi$. We denote by $\mathrm{Cont}_0(M,\xi)$ the identity component of the contactomorphism group of $(M,\xi)$, i.e.\ the set of all diffeomorphisms of $M$ which are smoothly isotopic to the identity by a path of diffeomorphisms whose differential maps $\xi$ to itself. The following function
\begin{equation}
\label{bmpm}
d(\alpha,\beta) \coloneqq \inf \{ \max f - \min f \mid \exists \varphi\in \mathrm{Cont}_0(M,\xi) \mbox{ such that } \varphi^* \beta = e^f \alpha \}
\end{equation}
is readily seen to be a pseudo-metric on $\mathcal{F}(M,\xi)$. The group $G\coloneqq \R\times  \mathrm{Cont}_0(M,\xi)$ acts on $\mathcal{F}(M,\xi)$ by
\[
(s,\varphi)\cdot \alpha \mapsto e^s \varphi^* \alpha,
\]
and the pseudo-metric $d$ is clearly invariant under this action, meaning that 
\[
d(e^s \alpha, e^t \beta) = d(\alpha,\beta), \quad \forall s,t\in \R, \qquad d(\varphi^* \alpha,\psi^* \beta) = d(\alpha,\beta), \quad \forall \varphi,\psi\in \mathrm{Cont}_0(M,\xi),
\]
for every $\alpha,\beta\in \mathcal{F}(M,\xi)$. In particular, the pseudo-distance of any two elements of the same $G$-orbit is zero. More precisely, $d$ measures the minimal distance between the $G$-orbits of two contact forms in $\mathcal{F}(M,\xi)$ in the $C^0$-topology and can be therefore considered as a contact analogue of the Banach--Mazur pseudo-metric on the set of convex bodies in $\R^n$ (there, the action is given by the group of affine transformations).

\begin{mainrem}
{\rm Banach--Mazur pseudo-metrics on contact forms were introduced by Rosen and Zhang in \cite{rz21}, where they considered the slightly different pseudo-metric
\[
d'(\alpha,\beta) \coloneqq \inf \{ \max |f| \mid \exists \varphi\in \mathrm{Cont}_0(M,\xi) \mbox{ such that } \varphi^* \beta = e^f \alpha \}.
\]
This pseudo-metric is invariant only under the action of $\mathrm{Cont}_0(M,\xi)$ and satisfies
\[
d'(e^s\alpha,e^t \alpha) = |t-s|, \qquad d'(e^s \alpha,e^t \beta) \leq  |t-s| + d'(\alpha,\beta), \qquad \forall s,t\in \R.
\]
Here, we prefer to work with the pseudo-metric $d$ which is invariant also under rescalings, but all the results below hold for $d'$ as well. A different pseudo-metric on $\mathcal{F}(M,\xi)$ is studied in Melistas' PhD thesis \cite{mel21}. Banach--Mazur pseudo-metrics on contact forms are strictly related to symplectically invariant Banach--Mazur pseudo-metrics on domains, whose definition was suggested by Ostrover and Polterovich and which are studied in \cite{prsz20,sz21,ush22}.}
\end{mainrem}

By invariance, $d$ descends to the quotient of $\mathcal{F}(M,\xi)$ by $G$. This quotient pseudo-metric need not be a genuine metric either. Indeed, by construction $d(\alpha,\beta)$ vanishes if and only if $\beta$ belongs to the $C^0$-closure of the $G$-orbit of $\alpha$. Therefore, the element $G\cdot \alpha$ in the orbit space $\mathcal{F}(M,\xi)/G$ has positive pseudo-distance from all the other elements if and only if $G\cdot \alpha$ is $C^0$-closed in $\mathcal{F}(M,\xi)$.  The next examples show that some $G$-orbits are indeed $C^0$-closed, but other ones are not even $C^{\infty}$-closed.

In both examples, we consider the standard tight contact structure $\xi_{\mathrm{st}}$ on $S^3$ which is defined by the restriction of the one-form $\lambda_0$ to $S^3\subset \C^2$. For every pair of positive numbers $a,b$ we denote by $\varepsilon_{a,b}\in \mathcal{F}(S^3,\xi_{\mathrm{st}})$ the contact form which is given by the pull-back of the restriction of $\lambda_0$ to the boundary of the ellipsoid $E(a,b)$ by the radial projection $S^3 \rightarrow \partial E(a,b)$. 

\begin{mainex}
\label{rational}
{\rm If $\frac{a}{b}$ is rational then the $G$-orbit of $\varepsilon_{a,b}$ is $C^0$-closed in $\mathcal{F}(S^3,\xi_{\mathrm{st}})$. Here is a proof based on Hutchings' embedded contact homology (ECH), for which we refer to \cite{hut14}. Let $(s_j,\varphi_j)$ be a sequence in $G$ such that $e^{s_j} \varphi_j^* \varepsilon_{a,b}$ $C^0$-converges to some contact form $\alpha$ in $\mathcal{F}(S^3,\xi_{\mathrm{st}})$. Then the total volume of this sequence of contact forms converges to the total volume of $\alpha$. This implies that the sequence $s_j$ converges to a real number $s$, and hence $\varphi_j^* \varepsilon_{a,b}$ converges to the contact form $\beta\coloneqq e^{-s} \alpha$. Denote by $\sigma_k: \mathcal{F}(S^3,\xi_{\mathrm{st}})\rightarrow \R$ the ECH-spectral invariant which is associated with the generator of degree $2k$ of the embedded contact homology of $(S^3,\xi_{\mathrm{st}})$. Since the ECH-spectral invariants are invariant under contactomorphisms and $C^0$-continuous, we have $\sigma_k(\beta)=\sigma_k(\varepsilon_{a,b})$ for every $k$. As shown by Cristofaro-Gardiner and Mazzucchelli in \cite[Lemma 3.1]{cgm20}, any contact form $\beta$ in $\mathcal{F}(S^3,\xi_{\mathrm{st}})$ having the same ECH-spectral invariants of $\varepsilon_{a,b}$ with $\frac{a}{b}$ rational is equivalent to $\varepsilon_{a,b}$: there exists $\psi\in \mathrm{Cont}_0(S^3,\xi_{\mathrm{st}})$ such that $\beta=\psi^* \varepsilon_{a,b}$ (see also \cite{mr23}). Therefore, $\alpha = e^s \psi^* \varepsilon_{a,b}$ belongs to the $G$-orbit of $\varepsilon_{a,b}$, which is then $C^0$-closed.
}
\end{mainex}

\begin{mainex} 
\label{irrational}
{\rm If $\frac{a}{b}$ is irrational and not Diophantine then the $G$-orbit of $\varepsilon_{a,b}$ is not $C^{\infty}$-closed in $\mathcal{F}(S^3,\xi_{\mathrm{st}})$. Indeed, by the Anosov--Katok conjugation method from \cite{ak70} and \cite{kat73} one can construct a sequence $(\varphi_j)$ in $\mathrm{Cont}_0(S^3,\xi_{\mathrm{st}})$ such that $(\varphi_j^* \varepsilon_{a,b})$ $C^{\infty}$-converges to a contact form $\alpha\in \mathcal{F} (S^3,\xi_{\mathrm{st}})$ whose Reeb flow has a dense orbit. See Theorem \ref{katok} and Remark \ref{more} in the appendix below. The contact form $\alpha$ cannot belong to the $G$-orbit of $\varepsilon_{a,b}$ because the Reeb flow of the latter contact form does not have dense orbits.}

{\rm Note however that the domain in $\C^2$ which corresponds to $\alpha$ in the correspondence between domains in $\mathcal{A}$ and elements of $\mathcal{F}(S^3,\xi_{\mathrm{st}})$ is symplectomorphic to the open ellipsoid $E(a,b)$, see Theorem \ref{eliashberg-hofer} in the appendix below. This is a case of ``invisible symplectic boundaries'', the first examples of which were found by Eliashberg and Hofer in \cite{eh96}.}

{\rm If $\frac{a}{b}$ is Diophantine we do not know whether the $G$-orbit of $\varepsilon_{a,b}$ is closed in some $C^k$-topology. Nor we know whether there are domains in $\mathcal{A}$ which are symplectomorphic to the open ellipsoid $E(a,b)$ but such that the restriction of $\lambda_0$ to $\partial A$ is not conjugate to $\varepsilon_{a,b}$.  These questions are related to an open question of Hermann about area-preserving diffeomorphisms of the annulus (see  \cite[Question 3.2]{her98b}), whose reformulation for flows is the following: Is a Reeb flow on $(S^3,\xi_{\mathrm{st}})$ with only two closed orbits having periods with Diophantine ratio smoothly conjugated to the Reeb flow of an irrational ellipsoid?}
\end{mainex}

It is interesting to investigate the situation in which the infimum defining $d(\alpha,\beta)$ is achieved, as this fact implies the existence of common invariant probability measures and of minimizing geodesics connecting $\alpha$ and $\beta$ in $\mathcal{F}(M,\xi)$. In this context, a continuous curve $\gamma:[0,1] \rightarrow \mathcal{F}(M,\xi)$ is said to be a minimizing geodesic if its length
\[
\mathrm{Length}(\gamma) \coloneqq \sup \Bigl\{ \sum_{j=1}^k d(\gamma(t_{j-1}),\gamma(t_j))\ \Big|\ k\in \N, \; 0 = t_0 < t_1< \dots < t_k = 1 \Bigr\}
\]
coincides with $d(\gamma(0),\gamma(1))$. The precise result is as follows.

\begin{mainthm}
\label{achieved}
Let $\alpha,\beta\in \mathcal{F}(M,\xi)$ and assume that 
\[
d(\alpha,\beta) = \max_M f - \min_M f
\]
for some $f\in C^{\infty}(M)$ and $\varphi\in \mathrm{Cont}_0(M,\xi)$ such that $\varphi^* \beta = e^f \alpha$. Then the following facts hold:
\begin{enumerate}[(i)]
\item There are probability measures which are supported in the level sets $f^{-1}(\min f)$ and $f^{-1}(\max f)$ and are invariant with respect to the Reeb flows of both $\alpha$ and $\varphi^* \beta$.
\item For every path of contactomorphisms $\{\psi_t\}_{t\in [0,1]} \subset \mathrm{Cont}_0(M,\xi)$ such that $\psi_0=\mathrm{id}$ and $\psi_1 = \varphi^{-1}$, the path
\[
\gamma(t) \coloneqq \psi_t^* ( e^{tf} \alpha), \quad t\in [0,1],
\]
is a minimizing geodesic with $\gamma(0)=\alpha$ and $\gamma(1)=\beta$.
\end{enumerate}
\end{mainthm}

In particular, if the contact form $\alpha$ in the above proposition is Zoll, then (i) implies that the Reeb flow of $\beta$ has a closed orbit which is the image by $\varphi$ of a closed orbit of the Reeb flow of $\alpha$ and whose $\beta$-action is $e^{\min f}$ times its $\alpha$-action. As an immediate consequence, we obtain the following statement: If $\alpha\in \mathcal{F}(M,\xi)$ is a Zoll contact form then any $\beta$ in $\mathcal{F}(M,\xi)$ such that the infimum in the definition of $d(\alpha,\beta)$ is achieved satisfies the contact systolic inequality $\rho_{\mathrm{sys}}(\beta) \leq \rho_{\mathrm{sys}}(\alpha)$.

The fact that any co-oriented contact structure $\xi$ admits defining contact forms with arbitrarily large systolic ratio (see \cite{sag21}) shows that the infimum in the definition of $d(\alpha,\beta)$ may not be achieved. Actually, Example \ref{C^1counter2} shows that $d(\alpha,\beta)$ may not be achieved for contact forms $\beta$ which are $C^1$-close to a Zoll contact form $\alpha$. It is however achieved when the contact form $\beta$ is $C^2$-close enough to the Zoll contact form $\alpha$. Indeed, our last result states the following,

\begin{mainthm} 
\label{short-geodesics}
Let $\alpha_0$ be a Zoll contact form defining a co-oriented contact structure $\xi$ on the closed manifold $M$. If the contact form $\alpha\in \mathcal{F}(M,\xi)$ is $C^2$-close enough to $\alpha_0$ then there exist $f\in C^{\infty}(M)$ and $\varphi\in \mathrm{Cont}_0(M,\xi)$ such that 
\[
\varphi^* \alpha = e^f \alpha_0 \quad \mbox{and} \quad d(\alpha_0,\alpha) = \max_{M} f - \min_{M} f = \log \frac{T_{\max}(\alpha)}{T_{\min}(\alpha)},
\]
where $T_{\max}(\alpha)$ and $T_{\min}(\alpha)=\mathrm{sys}(\alpha)$ are the maximum and minimum periods of the short closed orbits of $\alpha$, in the sense of Theorem \ref{qi} (d). In particular, there are minimizing geodesics connecting $\alpha_0$ to $\alpha$. 
\end{mainthm}

This result will be deduced from Theorem \ref{qi} using a sequence of ``elementary'' spectral invariants for elements of $\mathcal{F}(M,\xi)$, whose definition mimics analogous spectral invariants for domains and Hamiltonian diffeomorphisms which were introduced by McDuff and Siegel in \cite{ms23}, Hutchings in \cite{hut22a,hut22b} and one of us in \cite{edt22b}.

\paragraph{Organization of the paper.} In Section \ref{lifts}, we show how suitable compactly supported Hamiltonian diffeomorphisms of the unit ball in $\C^{n-1}$ can be lifted as characteristic flows on the boundary of certain domains in $\C^n$. We use this construction in Section \ref{prop13} to prove Propositions \ref{C^1counter} and \ref{C^1counter2}, and in Section \ref{pf1stcapi} to prove part (i) of Theorem \ref{1stcap}. In Section \ref{norforsec}, we recall and complement the main result from \cite{ab23}, which is then used in Section \ref{qisec} to prove Theorem \ref{qi}. Building on this theorem, we prove statement (ii) of Theorem \ref{1stcap} in Section \ref{proof-of-thm-1}. The proof of Theorem \ref{nthcap} is similar to the proof of Theorem \ref{1stcap} and is sketched in Section \ref{pfnthcap}. Proposition \ref{polydiscs} is proved in Section \ref{secpoly} and Theorem \ref{achieved} is proved in Section \ref{secachieved}. In Section \ref{spesec}, we introduce the spectral invariants which are mentioned above. In Section \ref{secthm5}, we prove Theorem \ref{short-geodesics}. Appendix A contains a discussion of the results which are mentioned in Example \ref{irrational}.

\paragraph{Acknowledgments.} We are grateful to Jean Gutt and Vinicius Ramos for sharing with us the preliminary version of their work \cite{gr23} on $k$-normalized capacities.

A.A.~is partially supported by the DFG under the Collaborative Research Center SFB/TRR 191 - 281071066 (Symplectic Structures in Geometry, Algebra and Dynamics). G.B.~is partially supported by the DFG 
under Germany's Excellence Strategy EXC2181/1 - 390900948 (the Heidelberg STRUCTURES Excellence Cluster). A.A.~and G.B.~gratefully acknowledge support from the Simons Center for Geometry and Physics, Stony Brook University at which some of the research for this paper was performed during the program \textit{Mathematical Billiards: at the Crossroads of Dynamics, Geometry, Analysis, and Mathematical Physics}.

\numberwithin{equation}{section}

\section{From Hamiltonian diffeomorphisms to domains}
\label{lifts}

In this section, which is inspired by analogous 4-dimensional arguments from \cite{edt22}, we show how suitable compactly supported Hamiltonian diffeomorphisms of the unit ball in $\C^{n-1}$ can be lifted as characteristic flows on the boundary of certain domains in $\C^n$. These arguments will be used here to prove Propositions \ref{C^1counter} and \ref{C^1counter2} in Section \ref{prop13} and to prove part (i) of Theorems \ref{1stcap} and \ref{nthcap} in Sections \ref{pf1stcapi} and \ref{pfnthcap}, respectively.

Setting $\T\coloneqq \R/\Z$, we consider the domain
\[
\Omega \coloneqq \{(s,t,w)\in \R \times \T \times \C^{n-1} \mid s > \pi ( |w|^2-1) \}
\]
and the smooth map
\begin{equation}
\label{Phi}
\Phi: \Omega \rightarrow \C^n, \qquad \Phi(s,t,w) \coloneqq e^{2\pi i t} \left( \sqrt{1 + \textstyle{\frac{s}{\pi}} - |w|^2} , w \right),
\end{equation}
which is a diffeomorphism onto $\C^* \times \C^{n-1}$. Recall that $\lambda_0$ denotes the standard primitive \eqref{lambda_0} of the standard symplectic form $\omega_0$ of $\C^n$. Denoting by $\widehat{\lambda}_0$ and $\widehat{\omega}_0$ the analogous forms on $\C^{n-1}$, we have
\begin{equation}
\label{pbl}
\Phi^* \lambda_0 = (\pi + s) \, dt + \widehat{\lambda}_0,
\end{equation}
where $(s,t)$ denote the coordinates in $\R \times \T$. By differentiating \eqref{pbl}, we obtain that $\Phi$ is a symplectomorphism from $(\Omega,ds\wedge dt + \widehat{\omega}_0)$ to $(\C^* \times \C^{n-1},\omega_0)$. Since
\begin{equation}
\label{normPhi}
|\Phi(s,t,w)| = \sqrt{1 + \textstyle{\frac{s}{\pi}} },
\end{equation}
this symplectomorphism maps each domain $\{(s,t,w)\in \Omega \mid s < s_0\}$ to a ball centered at 0 minus the linear subspace $\{z_1=0\}$. In particular, the hypersurface $\{(s,t,w)\in \Omega \mid s=0\} = \{0\} \times \T \times \widehat{B}$, where $\widehat{B}$ denotes the open unit ball in $\C^{n-1}$, is mapped onto the unit sphere $\partial B$ minus the linear subspace $\{z_1=0\}$.

Given a smooth compactly supported time-periodic Hamiltonian $H: \T \times \widehat{B} \rightarrow \R$ such that
\begin{equation}
\label{buonadef}
H(t,w) > - \pi (1-|w|^2) \qquad \forall (t,w)\in \T\times \widehat{B},
\end{equation}
we consider the subset of $\C^n$ given by
\[
D(H)\coloneqq \Phi \bigl( \{ (s,t,w) \in\Omega \mid w\in \widehat{B} , \;  s < H(t,w) \} \bigr) \cup \{z\in \C^n \mid z_1 = 0, \; |z|<1 \},
\]
which is readily seen to be an open neighborhood of $0$ diffeomorphic to a ball. The boundary of $D(H)$ is smooth and is  given by the closure of the image of the graph of $H$ by $\Phi$:
\[
\partial D(H) = \overline{\Phi(\Gamma(H))} = \Phi(\Gamma(H)) \cup \{z\in \C^n \mid z_1=0, \; |z|=1\},
\]
where
\[
\Gamma(H) \coloneqq \{(s,t,w)\in \R\times \T \times \widehat{B} \mid s=H(t,w)\}.
\]
The domain $D(H)$ coincides with $B$ near the subspace $\{z_1=0\}$ and its boundary is transverse to the vector field $\nabla |z_1|$ away from $\{z_1=0\}$. Conversely, any smooth domain with these properties has the form $D(H)$ for a suitable compactly supported smooth function $H$ on $\T\times \widehat{B}$ satisfying \eqref{buonadef}. Moreover, $D(H)$ is $C^k$-close to $B$ when $H$ is $C^k$-small and supported uniformly away from the boundary of $\T\times \widehat{B}$. Denote by 
\[
\phi_H^t: \widehat{B} \rightarrow  \widehat{B}, \qquad \phi_H^0 = \mathrm{id},
\]
the smooth path of compactly supported Hamiltonian diffeomorphisms of $( \widehat{B},\widehat{\omega}_0)$ which is obtained by integrating the time-periodic Hamiltonian vector field $X_H$ given by
\[
\imath_{X_{H_t}} \widehat{\omega}_0 = dH_t, \qquad H_t(w) \coloneqq H(t,w).
\]
We recall that the Calabi invariant of $\varphi\coloneqq\phi_H^1$ is the real number
\[
\mathrm{Cal}(\varphi) \coloneqq \int_{\T \times  \widehat{B}} H \, dt\wedge \widehat{\omega}_0^{n-1}.
\]
The notation is justified by the fact that the above integral does not depend on the choice of the compactly supported Hamiltonian $H$ defining $\varphi$. The Calabi invariant defines a real-valued homomorphism on the group of compactly supported Hamiltonian diffeomorphisms of $( \widehat{B},\widehat{\omega}_0)$. The action of a fixed point $w$ of $\varphi=\phi_H^1$ is given by
\[
\mathcal{A}_{\varphi}(w) \coloneqq \int_{\substack{t \mapsto \phi^t_H(w)\\t\in [0,1]}} \widehat{\lambda}_0 + \int_0^1 H(t,\phi_H^t(w))\, dt.
\]
This quantity does not depend on the choice of the compactly supported Hamiltonian $H$ defining $\varphi$ either. See \cite[Chapter 9 and 10]{ms17} for the proofs of the above mentioned facts about the Calabi invariant and the action. The next result relates the properties of the Hamiltonian diffeomorphism determined by $H$ to those of the domain $D(H)$.

\begin{prop}
\label{lift}
Let $H: \T \times \widehat{B} \rightarrow \R$ be a smooth compactly supported Hamiltonian satisfying \eqref{buonadef} and set $\varphi\coloneqq \phi_H^1$. Then:
\begin{enumerate}[(i)]
\item $\displaystyle\mathrm{vol}(D(H)) = \frac{\pi^n}{n!}  + \frac{1}{(n-1)!} \mathrm{Cal}(\varphi) = \mathrm{vol}(B) +  \frac{1}{(n-1)!} \mathrm{Cal}(\varphi)$.
\item There is a one-to-one correspondence between the periodic points $w$ of $\varphi$ and the closed characteristics $\gamma$ on the boundary of $D(H)$ other than those which foliate the submanifold $\partial D(H)\cap \{z_1=0\} = \partial B \cap \{z_1=0\}$. The corresponding actions are related by
\[
\int_{\gamma} \lambda_0 = k \pi + \mathcal{A}_{\varphi^k}(w),
\]
where $k\in \N$ denotes the minimal period of $w$.
\item Let $\{H^{\lambda}\}_{\lambda\in [0,1]}$ be a smooth family of compactly supported Hamiltonians on $\T\times \widehat{B}$ satisfying  \eqref{buonadef} and such that $\varphi\coloneqq \phi_{H^{\lambda}}^1$ does not depend on $\lambda$. Then there exists a compactly supported symplectomorphism $\psi:\C^n\rightarrow \C^n$ such that $\psi(D(H^0))=D(H^1)$. 
\end{enumerate}
\end{prop}

\begin{proof} (i) The map 
\[
\psi: \T \times \widehat{B} \rightarrow \partial D(H), \quad \psi(t,w) \coloneqq \Phi(H(t,w),t,w),
\]
is a diffeomorphism onto $\partial D(H)\setminus \{z_1=0\}$  and hence
\[
n!\, \mathrm{vol}(D(H)) = \int_{D(H)} \omega_0^n = \int_{\partial D(H)} \lambda_0 \wedge \omega_0^{n-1} = \int_{\T \times \widehat{B}} \psi^*( \lambda_0 \wedge \omega_0^{n-1}).
\]
From \eqref{pbl} we deduce 
\[
\psi^* \lambda_0 = (\pi + H)\, dt + \widehat{\lambda}_0, \qquad \psi^* \omega_0 = d \psi^* \lambda_0 = dH\wedge dt + \widehat{\omega}_0,
\]
and hence
\[
\psi^*( \lambda_0 \wedge \omega_0^{n-1}) = (\pi+H) dt \wedge \widehat{\omega}_0^{n-1} + (n-1) dH \wedge dt\wedge \widehat{\lambda}_0 \wedge \widehat{\omega}_0^{n-2}.
\]
Since $dH \wedge dt\wedge \widehat{\lambda}_0 \wedge \widehat{\omega}_0^{n-2}$ differs from $H \, dt \wedge \widehat{\omega}_0^{n-1}$ by an exact form, we obtain
\[
n!\, \mathrm{vol}(D(H)) = \pi \int_{\T\times \widehat{B}} dt\wedge \widehat{\omega}_0^{n-1} + n \int_{\T\times \widehat{B}} H\, dt \wedge  \widehat{\omega}_0^{n-1} = \pi^n + n \,\mathrm{Cal}(\varphi),
\]
which implies (i).

\bigskip

\noindent (ii) The hypersurface $\Gamma(H)$ is the zero level set of the autonomous Hamiltonian
\begin{equation}
\label{auto}
\widetilde{H}(s,t,w) \coloneqq H(t,w) - s
\end{equation}
on the symplectic manifold $(\R\times \T \times \widehat{B},ds\wedge dt + \widehat{\omega})$, whose Hamiltonian vector field is easily seen to be
\begin{equation}
\label{campo}
X_{\widetilde{H}}(s,t,w) = \partial_t H(t,w)\, \partial_s  + \partial_t  + X_{H_t}(w).
\end{equation}
The closed characteristics of $\Gamma(H)$ are the closed orbits of the flow of $X_{\widetilde{H}}$ on the energy level $\{\widetilde{H}=0\}$. Every orbit of this flow intersects the hypersurface $\{t=0\}$ transversally infinitely many times, and the orbit of $(s,0,w)$ is easily seen to be given by
\begin{equation}
\label{flow}
\phi_{\widetilde{H}}^{\tau}(s,0,w) = \bigl( s + H(\tau,\phi_H^{\tau}(w)) - H(0,w),\tau,\phi_H^{\tau}(w) \bigr), \qquad  \tau \in \R.
\end{equation}
In particular, the orbit of a point $(s,0,w)$ in $\Gamma(H)$ is the curve
\[
\widehat{\gamma}(\tau) = \bigl( H(\tau,\phi_{X_H}^{\tau}(w))  , \tau, \phi_{H}^{\tau}(w) \bigr), \qquad  \tau \in \R.
\]
This orbit is closed if and only if $w$ is a $k$-periodic point of $\phi_X^1$, for some $k\in\N$, and in this case
\[
\int_{[0,k]} \widehat{\gamma}^*( \Phi^* \lambda_0) = \int_{[0,k]} \widehat{\gamma}^*( (\pi+s)\, dt + \widehat{\lambda}_0) = \pi k + \mathcal{A}_{\phi_H^k}(w).
\]
Statement (ii) follows by taking $\gamma\coloneqq\Phi(\widehat\gamma)$.

\bigskip

\noindent (iii) For each $(\lambda,t)\in[0,1]\times\T$, consider the Hamiltonian diffeomorphism
\[
\psi_\lambda^t\colon \widehat B\to\widehat B,\qquad \psi_\lambda^t\coloneqq\phi^t_{H^\lambda}\circ (\phi^t_{H^0})^{-1}.
\]
By assumption, we have
\begin{equation}\label{psi0lambda}
\psi_\lambda^0=\psi_\lambda^1=\mathrm{id}_{\widehat B},\qquad  \forall\lambda\in[0,1]. 
\end{equation}
For each $\lambda\in[0,1]$, the loop of Hamiltonian diffeomorphisms $t\mapsto \psi_\lambda^t$, $t\in\T$, based at $\mathrm{id}_{\widehat B}$ is generated by the Hamiltonian
\begin{equation}\label{Glambda}
G^\lambda\colon\T\times\widehat B \rightarrow \R,\qquad G^\lambda(t,w)=H^\lambda\#\overline{ H^0}(t,w)=H^\lambda(t,w)-H^0(t,(\psi^t_\lambda)^{-1}(w)).
\end{equation}
Notice that $G^\lambda$ is indeed 1-periodic in $t$ since the same is true for $\psi_\lambda^t$. Moreover, $G^0=0$. We define
the diffeomorphism
\begin{equation}\label{tildepsi}
\tilde \psi_\lambda\colon \R\times \T\times\widehat B  \rightarrow  \R\times \T\times\widehat B,\quad \tilde \psi_\lambda(s,t,w)=(s+G^\lambda(t,\psi^t_\lambda(w) ),t,\psi^t_\lambda(w)).
\end{equation}
We have $\tilde\psi_0=\mathrm{id}$ and, using the formula for $G^\lambda$ given in \eqref{Glambda}, that $\widetilde{H}^{\lambda}\circ \tilde\psi_\lambda= \widetilde{H}^0$, where $\widetilde{H}^{\lambda}$ is induced by $H^{\lambda}$ as in \eqref{auto}. In particular, $\tilde\psi_\lambda$ maps $\Gamma(H^0)$ to $\Gamma(H^\lambda)$. 

We now make the claim that the path $\lambda\mapsto \tilde\psi_\lambda$, $\lambda\in[0,1]$, is generated by a smooth path of Hamiltonians $F^\lambda\colon\R\times\T\times\widehat B\to\R$ which are supported in a set of the form $\R \times \T \times K$, where $K$ is a compact subset of $\widehat{B}$. If the claim is true, we can consider the multiplication $E^\lambda\coloneqq\chi^\lambda F^\lambda$ of $F^\lambda$ by a cut-off function $\chi^\lambda$ supported in an arbitrarily small neighborhood of $\Gamma(H^\lambda)$. Since $F^\lambda$ has support in $\R \times \T \times K$, the function $E^\lambda$ has compact support. Since $\Phi$ is a symplectomorphism, the function $E^\lambda\circ \Phi^{-1}$ yields a compactly supported function of the whole $\C^n$. The time-one map $\psi\colon\C^n\to\C^n$ of the Hamiltonian flow of $\{E^\lambda\circ \Phi^{-1}\}_{\lambda\in[0,1]}$ yields the desired symplectomorphism.  

We are left with proving the claim. As a first step, consider the compactly supported Hamiltonian function $F^\lambda_t\colon\widehat B\to\R$, which, for every $t\in\R$, generates the path of compactly supported Hamiltonian diffeomorphisms $\lambda\mapsto\psi^t_\lambda$, $\lambda\in[0,1]$. By \cite[Prop.~3.1.5]{ban97}, we have
\begin{equation}\label{comban}
\partial_t F^\lambda_t-\partial_\lambda G^\lambda_t=\{F^\lambda_t,G^\lambda_t\},
\end{equation}
where $\{F^\lambda_t,G^\lambda_t\}=dG^\lambda_t[X_{F^\lambda_t}]=-dF^\lambda_t[X_{G^\lambda_t}]=-\{G^\lambda_t,F^\lambda_t\}$ is the Poisson bracket.

We now make the subclaim that 
\begin{equation}\label{subclaim}
F^\lambda_1=F^\lambda_0,\qquad \forall\lambda\in[0,1].
\end{equation}
If the subclaim is true, then the function
\[
F^\lambda\colon \R\times\T\times\widehat B\to\R,\qquad F^\lambda(s,t,w)=F^\lambda_t(w)
\]
is well-defined and smooth, by the smoothness of $t\mapsto \psi_{\lambda}^t$ on $\T$.
This function has the desired property. Indeed, by \eqref{tildepsi} and \eqref{comban}, we get
\[
(\partial_\lambda\tilde\psi_\lambda)\circ (\tilde\psi_\lambda)^{-1} =\big(\partial_\lambda G^\lambda_t+\{F^\lambda_t,G^\lambda_t\}\big)\partial_s+X_{F^\lambda_t}=\partial_t F^\lambda_t \partial_s+X_{F^\lambda_t}
\]
and
\[
\imath_{\partial_t F^\lambda_t \partial_s+X_{F^\lambda_t}}(ds\wedge dt+\widehat\omega)=\partial_t F^\lambda_tdt+dF^\lambda_t=dF^\lambda.
\]
We are now left with proving the subclaim \eqref{subclaim}. Since $G^\lambda$ has compact support and $\phi^1_{G^\lambda}(w)=w$ for all $w\in \widehat B$, we have $\mathcal A_{\phi^1_{G^\lambda}}(w)=0$. Differentiating this equality with respect to $\lambda$ and setting 
\[
\gamma_\lambda\colon\T\to\widehat B,\qquad \gamma_\lambda(t)\coloneqq\phi^t_{G^\lambda}(w),\ \forall t\in\T,
\]
using the definition of the action we get
\begin{align}
0=\partial_\lambda\mathcal A_{\phi^1_{G^\lambda}}(w)&=\int_0^1\Big(\widehat\omega(\partial_\lambda\gamma_\lambda(t),\dot\gamma_\lambda(t))+dG^\lambda_t(\gamma_\lambda(t))[\partial_\lambda\gamma_\lambda(t)]\Big)dt+\int_0^1(\partial_\lambda G^\lambda_t)(\gamma_\lambda(t))dt\nonumber\\
&=\int_0^1(\partial_\lambda G^\lambda_t)(\gamma_\lambda(t))dt,\label{partiallambda} 
\end{align}
where we used that $\gamma_\lambda$ is a Hamiltonian trajectory for $G^\lambda$.
On the other hand, by \eqref{comban}, we get
\begin{align*}
\int_0^1(\partial_\lambda G^\lambda_t)(\gamma_\lambda(t))dt=\int_0^1\big(\partial_t F^\lambda_t+dF^\lambda_t[\dot\gamma_\lambda]\big)(\gamma_\lambda(t))dt&=\int_0^1\frac{d}{dt}\Big(F^\lambda_t(\gamma_\lambda(t)\Big)dt\\
&=F^\lambda_1(\gamma_\lambda(1))-F^\lambda_0(\gamma_\lambda(0))\\
&=F^\lambda_1(w)-F^\lambda_0(w),
\end{align*}
which, in combination with \eqref{partiallambda}, yields $0=F^\lambda_1(w)-F^\lambda_0(w)$ for all $\lambda\in[0,1]$ and $w\in\widehat B$, as stated in the subclaim \eqref{subclaim}. 
\end{proof}

\begin{rem}
\label{support}
{\rm Let $\epsilon>0$ and $r\in (0,1)$ be such that $(-\epsilon,\epsilon) \times \T \times \widehat{B}_r$ is contained in $\Omega$ and set
\[
V\coloneqq \Phi( (-\epsilon,\epsilon) \times \T \times \widehat{B}_r ).
\]
The above proof shows that if $H^{\lambda}-H^0$ is supported in $\T\times \widehat{B}_r$ and $|H^{\lambda}|<\epsilon$ for every $\lambda\in [0,1]$, then the symplectomorphism $\psi$ in (iii) can be chosen to be supported in $V$ using suitable cut-off functions $\chi^\lambda$.}
\end{rem}

\section{Proof of Proposition \ref{C^1counter} and Proposition \ref{C^1counter2}} 
\label{prop13}

In this section we show how the first two statements of Proposition \ref{lift} can be used to prove Proposition \ref{C^1counter} from the Introduction. The proof of Proposition \ref{C^1counter2} goes along similar lines and is sketched in Remark \ref{C^1counter2sketch} at the end of this section. 

We consider a compactly supported smooth Hamiltonian $H: \T\times \widehat{B} \rightarrow \R$ such that all the fixed points of $\phi^1_H$ have non-negative action while the Calabi invariant of $\phi^1_H$ is negative. Following \cite{abhs18}, a Hamiltonian with these properties can be constructed in the following way. Consider an autonomous radial Hamiltonian on $\widehat{B}$ of the form
\[
F(w) \coloneqq f(|w|^2),
\]
where $f:[0,1) \rightarrow \R$ is a smooth compactly supported monotonically decreasing function such that $f'(r^2)=-\frac{\pi}{2}$ for every $r\in [0,\frac{1}{2}]$. The identity
\[
\phi_F^1(w) = e^{-2i f'(|w|^2)} w
\]
shows that $\phi_F^1(w)=-w$ if $|w|\leq \frac{1}{2}$. Next, consider an open ball $U\subset \widehat{B}$ contained in the ball of radius $\frac{1}{2}$ centered at the origin and such that $\phi_F^1(U) = - U$ does not intersect $U$. Let $G$ be a smooth function on $\widehat{B}$ with support in $U$ and such that 
\begin{equation}
	\label{negcal}
	\mathrm{Cal}(\phi_G^1) = \int_{\widehat{B}} G\, \widehat{\omega}^{n-1} < - \mathrm{Cal}(\phi_F^1).
\end{equation}
The Hamiltonian diffeomorphism $\phi_F^1\circ \phi_G^1$ is generated by the compactly supported Hamiltonian
\[
H(t,w) = (F\# G)(t,w) = F(w) + G(t, (\phi_F^t)^{-1}(w)). 
\]
The fact that $\phi_F^1$ displaces the support of $G$ from itself implies that the only fixed points of $\phi_H^1$ are fixed points of $\phi_F^1$, and if $w$ is such a fixed point we have
\[
\mathcal{A}_{\phi_H^1}(w) = \mathcal{A}_{\phi_F^1}(w) = f(|w|^2) - |w|^2 f'(|w|^2) \geq 0.
\]
Finally, \eqref{negcal} implies that the Calabi invariant of $\phi_H^1$ is negative. This shows that $H$ has the required properties. The $k$-periodic points of $\phi_H^1$ may have negative action, but by the definition of action we have a bound of the form
\begin{equation}
	\label{higherper}
	\bigl| \mathcal{A}_{\phi_H^k}(w) \bigr| \leq c k \qquad \forall k\in \N,
\end{equation}
for every $k$-periodic point $w$ of $\phi_H^1$, for a suitable number $c$.

We now see $H$ as a smooth Hamiltonian on $\T\times \C^{n-1}$ with support in $\T\times \widehat{B}$. The family of rescaled Hamiltonians
\[
H^{\lambda}(t,w) \coloneqq \lambda^2 H\bigl(t, {\textstyle \frac{w}{\lambda}} \bigr), \quad\lambda>0,
\]
$C^1$-converges to the zero function for $\lambda\rightarrow 0$. Therefore, there exists $\lambda_0>0$ such that $H^{\lambda}$ satisfies \eqref{buonadef} and $A_{\lambda} \coloneqq D(H^{\lambda})$ belongs to $\mathcal{A}$ if $\lambda\in (0,\lambda_0]$, and $A_{\lambda}$ $C^1$-converges to $B$ for $\lambda\rightarrow 0$. The identity
\[
X_{H^{\lambda}}(t,w) = \lambda X_H \bigl(t,{\textstyle \frac{w}{\lambda}} \bigr)
\]
implies that the conformally symplectic diffeomorphism $w\mapsto \lambda \cdot w$ conjugates the Hamiltonian dynamics of $H$ and $H^{\lambda}$, meaning that
\[
\phi^t_{H^{\lambda}} (w) = \lambda \cdot \phi^t_{H} \bigl( {\textstyle \frac{w}{\lambda}} \bigr).
\]
Therefore, $w\in \widehat{B}$ is a fixed point of $\phi_{H^{\lambda}}^1$ if and only if $\frac{w}{\lambda}$ is a fixed point of $\phi_H^1$, and in this case we have
\[
\mathcal{A}_{\phi_{H^{\lambda}}^1} (w) = \lambda^2 \mathcal{A}_{\phi_H^1} \bigl( {\textstyle \frac{w}{\lambda}} \bigr) \geq 0.
\]
If $\gamma$ is the closed characteristic on $\partial A_{\lambda}$ corresponding to a fixed point of $\phi_{H^{\lambda}}^1$, statement (ii) in Proposition \ref{lift} gives us the bound
\[
\int_{\gamma} \lambda_0 = \pi + \mathcal{A}_{\phi_{H^{\lambda}}^1} (w) \geq \pi.
\]
If $\gamma$ is the closed characteristic on $\partial A_{\lambda}$ corresponding to a periodic point of $\phi_{H^{\lambda}}^1$ of minimal period $k\geq 2$, then using also \eqref{higherper} we obtain 
\[
\int_{\gamma} \lambda_0 = k\pi + \mathcal{A}_{\phi_{H^{\lambda}}^k} (w) = k \pi + \lambda^2 \mathcal{A}_{\phi_H^k} \bigl( {\textstyle \frac{w}{\lambda}} \bigr)\geq k \pi - \lambda^2 c k \geq \pi,
\]
provided that $\lambda^2\leq \frac{\pi}{2c}$. Together with the fact that the characteristics of $\partial A_{\lambda}$ in $\{z_1=0\}$ are closed with action $\pi$, up to reducing the size of $\lambda_0$ we deduce that $\mathrm{sys}(A_{\lambda}) = \pi$ for every $\lambda\in (0,\lambda_0]$. On the other hand,
\[
\mathrm{Cal}(\phi_{H^{\lambda}}^1) = \lambda^{2n}\, \mathrm{Cal}(\phi_H^1) < 0,
\]
so by statement (i) in Proposition \ref{lift} we obtain
\[
\mathrm{vol}(A_{\lambda}) < \frac{\pi^n}{n!}  \qquad \forall \lambda\in (0,\lambda_0],
\]
and hence
\[
\mathrm{sys}(A_{\lambda})^n = \pi^n > n! \, \mathrm{vol}(A_{\lambda}) \geq \underline{c}(A_{\lambda})^n \qquad \forall \lambda\in (0,\lambda_0],
\]
proving the strict inequality
\[
\underline{c}(A_{\lambda}) < \mathrm{sys}(A_{\lambda}) \qquad  \forall \lambda\in (0,\lambda_0],
\]
which is stated in Proposition \ref{C^1counter}.

\begin{rem}
	\label{C^1counter2sketch}
	{\rm The proof of Proposition \ref{C^1counter2} is similar and we just sketch it. Assume without loss of generality that the orbits of the Zoll contact form $\alpha_0$ have minimal period 1. Then we can find an embedding $\varphi: \T \times \widehat{B}_r \rightarrow M$ such that
		\[
		\varphi^* \alpha_0 = dt + \widehat{\lambda}_0.
		\]
		Here, $\widehat{B}_r$ denotes the open ball of radius $r$ in $\C^{n-1}$. We modify $\alpha_0$ inside the image of $\varphi$ and obtain a smooth family of contact forms $\{\alpha_{\lambda}\}_{\lambda\in (0,\lambda_0]}$ such that
		\[
		\varphi^* \alpha_{\lambda} = ( 1 + H^{\lambda}) \,  dt + \widehat{\lambda}_0,
		\]
		where $\{H^{\lambda}\}$ is the family of rescaled Hamiltonians constructed above and $\lambda_0>0$ is so small that $H^{\lambda}$ is supported in $\T \times \widehat{B}_r$ for all $\lambda\in (0,\lambda_0]$. The contact forms $\alpha_{\lambda}$ converge to $\alpha_0$ in the $C^1$-topology for $\lambda\rightarrow 0$. By the properties of $H^{\lambda}$, we have $\mathrm{sys}(\alpha_{\lambda})=1 = \mathrm{sys}(\alpha_0)$ but the volume of $(M,\alpha_{\lambda}\wedge d\alpha_{\lambda}^{n-1})$ is strictly smaller than that of $(M,\alpha_0 \wedge d\alpha_0^{n-1})$.}
\end{rem}

\section{Proof of Theorem \ref{1stcap} (i)} 
\label{pf1stcapi}

The proof of statement (i) of Theorem \ref{1stcap} uses the construction from Section \ref{lifts} together with the following result.

\begin{prop}
\label{modham}
Let $m\in\mathbb N$ and denote by $B_r$ the open ball of radius $r$ centered at $0$ in $\C^m$ and by $B_r^{\mathrm{c}}$ its complement. For every $r>0$ and $\epsilon>0$ there exists $\delta>0$ such that if $H$ is a compactly supported Hamiltonian on $\T \times \C^{m}$ satisfying $\|H\|_{C^2} < \delta$ and such that $0$ is a fixed point of $\phi_H^1$ with $\mathcal{A}_{\phi_H^1}(0)=0$, then there exists a smooth family of Hamiltonians $\{H^{\lambda}\}_{\lambda\in [0,1]}\subset C^{\infty}(\T \times \C^{m})$ such that:
\begin{enumerate}[(i)]
\item $H^0=H$;
\item $H^{\lambda}=H$ on $\T \times B^{\mathrm{c}}_r$ for every $\lambda\in [0,1]$;
\item $\phi_{H^{\lambda}}^1 = \phi_H^1$ for every $\lambda\in [0,1]$;
\item $\Vert H^{\lambda}\Vert_{C^0}< \epsilon$ for every $\lambda\in [0,1]$;
\item $|H^1(t,z)| \leq \epsilon |z|^2$ for every $(t,z)\in \T\times \C^{m}$.
\end{enumerate}
\end{prop}

The proof of this proposition is given at the end of this section. Here, we show how Propositions \ref{lift} and \ref{modham} imply statement (i) of Theorem \ref{1stcap}.

Our aim is to prove that if $A\in \mathcal{A}$ is $C^2$-close enough to the unit ball $B$, then there exists a symplectomorphism $\psi: \C^n \rightarrow \C^n$ mapping $A$ into the cylinder whose systole coincides with the systole of $A$. Up to rescaling, we may assume that $\mathrm{sys}(A)=\pi$, so that $\psi$ is required to map $A$ into the standard cylinder $Z=\{z\in \C^n \mid |z_1|<1\}$. A closed characteristic $\gamma$ on $\partial A$ achieving $\mathrm{sys}(A)$ is $C^0$-close to some closed characteristic on $\partial B$. Thus, up to applying a unitary transformation of $\C^n$ and a small translation we may assume that $\gamma$ passes through the point $(1,0,\dots,0)\in \C^n$ and is $C^0$-close to the curve 
\[
\gamma_0: \T\rightarrow \C^n, \qquad t\mapsto (e^{2\pi i t},0,\dots,0).
\]
Let $\Phi: \Omega \rightarrow \C^*\times \C^{n-1}$ be the symplectomorphism introduced in \eqref{Phi}. Noting that $\Phi(0,t,0)=\gamma_0(t)$ and that the set
\begin{equation}
\label{U}
U\coloneqq \bigl(- {\textstyle \frac{\pi}{2}, \frac{\pi}{2}} \bigr) \times \T \times \widehat{B}_{\frac{2}{3}}
\end{equation}
is contained in $\Omega$, we consider the open neighborhood $V\coloneqq \Phi(U)$ of $\gamma_0$. Here, $\widehat{B}_r$ denotes the open ball of radius $r$ centered at the origin of $\C^{n-1}$. Being  close to $\gamma_0$, the curve $\gamma$ is contained in $V$. Moreover, the form of $\Phi$ implies the identities
\begin{equation}
\label{chart}
B \cap V= \Phi(U \cap \{s<0\}), \qquad Z \cap V = \Phi(U \cap \{ s< \pi |w|^2 \}).
\end{equation}
Since $A$ is $C^2$-close to $B$, the hypersurface $\partial A\cap V$ is the image via $\Phi$ of the graph of a $C^2$-small smooth function on $ \T \times \widehat{B}_{\frac{2}{3}}$. We can extend this function to a $C^2$-small smooth function $H$ on $\T\times \C^{n-1}$ having support in $\T \times \widehat{B}$ and satisfying \eqref{buonadef}, thus obtaining
\begin{equation}
\label{AcapV}
A \cap V = D(H) \cap V,
\end{equation}
where $D(H)$ denotes the domain which we introduced in the previous section. As it passes through $(1,0,\dots,0)=\Phi(0,0,0)$, the closed characteristic $\gamma_0$ on $\partial A\cap V = \partial D(H) \cap V$ corresponds via $\Phi$ to the fixed point $0$ of $\phi^1_H$ and
\[
\mathcal{A}_{\phi^1_H}(0) = \int_{\gamma} \lambda_0 - \pi = 0,
\]
as shown in Proposition \ref{lift} (ii). Assuming that $A$ is $C^2$-close enough to $B$ - and hence $H$ is $C^2$-small enough - we apply Proposition \ref{modham} with $m=n-1$ to get a smooth family of Hamiltonian functions $H^{\lambda} : \T \times \C^{n-1} \rightarrow \R$, $\lambda\in [0,1]$, such that:
\begin{enumerate}[(i)]
\item $H^0=H$;
\item $H^{\lambda}=H$ on $\T \times \widehat{B}^{\mathrm{c}}_{\frac{1}{3}}$ for every $\lambda\in [0,1]$;
\item $\phi_{H^{\lambda}}^1 = \phi_H^1$ for every $\lambda\in [0,1]$;
\item $\Vert H^{\lambda}\Vert_{C^0}< \frac{\pi}{2}$ for every $\lambda\in [0,1]$;
\item $H^1(t,w) \leq \pi |w|^2$ for every $(t,w)\in \T\times \C^{n-1}$.
\end{enumerate}
By properties (i) and (iii), Proposition \ref{lift} (iii) implies that there exists a symplectomorphism $\psi: \C^n \rightarrow \C^n$ such that $\psi(D(H))=D(H^1)$. Thanks to properties (ii), (iv) and Remark \ref{support}, we can assume that $\psi$ is supported in $V$ and hence \eqref{AcapV} implies
\[
\psi(A) = (A\setminus V) \cup (D(H^1)\cap V).
\]
Since the closure of $B\setminus V$ is contained in $Z$, $C^0$-closeness of $A$ to $B$ implies that $A\setminus V$ is contained in $Z$. By (v) and the second identity in \eqref{chart}, $D(H^1)\cap V$ is also contained in $Z$. We conclude that $\psi(A)$ is contained in $Z$, hence proving statement (i) in Theorem \ref{1stcap}.

\bigskip

There remains to prove Proposition \ref{modham}. The proof of this proposition makes use of some properties of generating functions which we now recall, referring to \cite{bp94} for more details. Every compactly supported symplectomorphism $\varphi: \C^m \rightarrow \C^m$ which is sufficiently $C^1$-close to the identity is represented by a unique $C^2$-small, compactly supported, smooth generating function $S:\C^m \rightarrow \R$ which is characterized by the equation
\begin{equation}
\label{genfun}
i(z-\varphi(z)) = \nabla S \bigl( {\textstyle \frac{z+\varphi(z)}{2} } \bigr) \qquad \forall z\in \C^m.
\end{equation}
Conversely, every $C^2$-small, compactly supported, smooth function $S$ defines by the above identity a compactly supported symplectomorphism $\varphi: \C^m \rightarrow \C^m$ which is $C^1$-close to the identity. If $\{\varphi_t\}_{t\in [0,1]}$ is a smooth isotopy starting at the identity and consisting of compactly supported symplectic diffeomorphisms of $\C^m$ which are $C^1$-close to the identity, then $\varphi_t = \phi_H^t$, 
where the smooth family of compactly supported Hamiltonians $\{H_t\}_{t\in [0,1]}$ is related to the smooth family $\{S_t\}_{t\in [0,1]}$ of compactly supported generating functions for $\{\varphi_t\}_{t\in [0,1]}$ by the Hamilton--Jacobi equation
\begin{equation}
\label{HJ}
\partial_t S_t (z) = H \bigl( t, z + {\textstyle \frac{i}{2}} \nabla S_t(z) \bigr)\qquad  \forall (t,z)\in [0,1]\times \C^m.
\end{equation}
Let $\varphi: \C^m \rightarrow \C^m$ be a compactly supported symplectomorphism which is $C^1$-close to the identity and denote by $S:\C^m \rightarrow \R$ the corresponding generating function. Let $z$ be a fixed point of $\varphi$. By \eqref{genfun}, $z$ is a critical point of $S$. The path of generating functions $\{tS\}_{t\in [0,1]}$ produces a compactly supported symplectic isotopy from the identity to $\varphi$ such that $z$ is a fixed point of all maps in this isotopy. If $H$ denotes the compactly supported Hamiltonian on $[0,1]\times \C^m$ associated with this isotopy, \eqref{HJ} implies that $H(t,z)=S(z)$ for every $t\in [0,1]$ and hence 
\begin{equation}
\label{action-S}
\mathcal{A}_{\varphi}(z) = S(z).
\end{equation}
After these preliminaries, we can prove Proposition \ref{modham}.

\begin{proof}[Proof of Proposition \ref{modham}] In this proof, whenever we say that some function is $C^k$-small, we understand ``provided that $H$ is $C^2$-small''.

We first prove this proposition under the additional assumption that $H_t$ is identically zero for every $t$ in a neighborhood of $0$ in $\T$. 

Since $H$ is $C^2$-small, each symplectomorphism $\phi_H^t$ is $C^1$-close to the identity and hence generated by a unique $C^2$-small generating function $S_t^0$. It follows from the Hamilton--Jacobi equation \eqref{HJ} that $S_t^0$ is also $C^2$-small when viewed as a function on $[0,1]\times \C^{m}$. 
By our additional assumption on $H$, $\phi_H^t=\mathrm{id}$ for $t$ close to $0$ and $\phi_H^t=\phi_H^1$ for $t$ close enough to $1$. Therefore, $S_t^0=0$ for $t$ close to $0$ and $S_t^0=S_1^0$ for $t$ close to 1. 

Fix a smooth monotonically increasing function $\eta: [0,1]\rightarrow [0,1]$ such that $\eta(t)=0$ for $t$ close to $0$ and $\eta(t)=1$ for $t$ close to $1$.
By modifying $S_t^0$, we can find a smooth family of functions $S_t^1$ such that:
\begin{enumerate}[(a)]
\item $S_t^1 = \eta(t) S_1^0$ inside the ball $B_{\frac{r}{3}}$, for every $t\in [0,1]$;
\item $S_t^1 = S_t^0$ outside the ball $B_{\frac{2}{3} r}$, for every $t\in [0,1]$;
\item $S_t^1=0$ for $t$ close to $0$ and $S_t^1 = S_1^0$ for $t$ close to $1$;
\item $S_t^1$ is $C^2$ small when viewed as a function on $[0,1]\times \C^m$.
\end{enumerate}
For every $\lambda\in [0,1]$, we define
\begin{equation*}
S_t^\lambda \coloneqq (1-\lambda) S_t^0 + \lambda S_t^1.
\end{equation*}
By (d), this function is $C^2$-small when viewed as a function on $[0,1]^2\times \C^m$. Moreover, $S_t^\lambda=0$ for $t$ sufficiently close to $0$ and $S_t^\lambda=S_1^0$ for $t$ sufficiently close to $1$. For every $\lambda\in [0,1]$, let $H^\lambda$ be the unique compactly supported Hamiltonian generating the Hamiltonian isotopy associated with the family of generating functions $\{S_t^\lambda\}_{t\in [0,1]}$. It follows from the Hamilton--Jacobi equation that $H^\lambda$ is $C^1$-small, and in particular we can ensure that (iv) holds. Clearly, $H^0 = H$ and by (b) $H^\lambda_t$ agrees with $H_t$ outside $B_r$, for every $t\in [0,1]$, so (i) and (ii) hold. By the second assertion in (c), we have $\phi_{H^\lambda}^1= \phi_H^1$ for all $\lambda\in [0,1]$, which proves (iii). Since $S_t^\lambda$ is constant in $t$ for $t$ close to $0$ and for $t$ close to $1$, the Hamiltonian $H^\lambda_t$ vanishes for $t$ near $0$ and $1$. Therefore, it descends to a $1$-periodic Hamiltonian. It remains to check (v). By (a), the Hamilton--Jacobi equation \eqref{HJ} implies that
\begin{equation}
\label{HJ2}
H^1 \bigl(t, z + {\textstyle \frac{i}{2}} \nabla S_t^1(z) \bigr) = \partial_t S_t^1(z) = \eta'(t) S_1^0(z)
\end{equation}
for every $z \in B_{\frac{r}{3}}$. Since $0$ is a fixed point of $\phi_H^1$ of zero action, the function $S_1^0$ and its derivative both vanish at $0$, see \eqref{action-S}. Therefore, we get
\[
|S_1^0(z)|\leq \tfrac12\Vert S_1^0\Vert_{C^2}|z|^2,\qquad \forall z\in \C^m.
\]
The fact that $S_1^0$ is $C^2$-small implies that the diffeomorphism
\[
\theta_t\colon \C^m\to\C^m, \quad\theta_t(z)\coloneqq z + {\textstyle \frac{i}{2}} \nabla S_t^1(z),\ \forall z\in \C^m,
\]
is $C^1$-close to the identity and fixes the origin. Thus, the same is true for the inverse diffeomorphism, and we get
\[
|\theta^{-1}_t(w)|\leq 2|w|,\qquad\forall w\in\C^{m}.
\]
Using \eqref{HJ2}, we obtain for all $(t,w)\in\T\times \C^m$
\[
|H^1(t,w)|=|\eta'(t)S_1^0(\theta^{-1}_t(w))|\leq \Vert\eta\Vert_{C^1}\tfrac12\Vert S_1^0\Vert_{C^2}|\theta^{-1}_t(w)|^2\leq \Vert\eta\Vert_{C^1}\tfrac12\Vert S_1^0\Vert_{C^2}4|w|^2,
\] 
which implies the bound (v) since $S_1^0$ is $C^2$-small.

\bigskip

We now show how the general case can be deduced from the special one considered above. Using again the function $\eta$ whose properties are described above, we introduce the Hamiltonian
\[
\widehat{H}(t,z) \coloneqq \eta'(t) H(\eta(t),z), \qquad \forall (t,z)\in [0,1]\times \C^m,
\]
which vanishes for $t$ close to $0$ and for $t$ close to $1$, and in particular can be seen as 1-periodic in time, is $C^2$-small and satisfies $\phi_{\widehat{H}}^1=\phi^1_H$. We claim that there exists a family of smooth Hamiltonians $K^\lambda: \T \times \C^m \rightarrow \R$, $\lambda\in [0,1]$, such that:
\begin{enumerate}[(a')]
\item $K^0 = H$;
\item $K^1=\widehat{H}$ on $\T \times B_{\frac{r}{2}}$;
\item $K^\lambda=H$ on $\T \times B_r^{\mathrm{c}}$ for every $\lambda\in [0,1]$;
\item $\phi_{K^\lambda}^1 = \phi_H^1$ for every $\lambda\in [0,1]$;
\item $\Vert K^\lambda\Vert_{C^0}<\epsilon$ for every $\lambda\in [0,1]$.
\end{enumerate}
Once this is proven, the conclusion of Proposition \ref{modham} in the general case follows. Indeed, by applying the above special case to the Hamiltonian $\widehat{H}$, we obtain a smooth family of Hamiltonians $\{\widehat{H}^{\lambda}\}_{\lambda\in [0,1]}$ on $\T\times \C^m$ which satisfies:
\begin{enumerate}[(i')]
\item $\widehat{H}^0=\widehat{H}$;
\item $\widehat{H}^{\lambda}=\widehat{H}$ on $\T \times B^{\mathrm{c}}_{\frac{r}{4}}$ for every $\lambda\in [0,1]$;
\item $\phi_{\widehat{H}^{\lambda}}^1 = \phi_{\widehat{H}}^1$ for every $\lambda\in [0,1]$;
\item $\Vert\widehat{H}^{\lambda}\Vert_{C^0} < \epsilon$ for every $\lambda\in [0,1]$;
\item $|\widehat{H}^1(t,z)| \leq \epsilon |z|^2$ for every $(t,z)\in \T\times \C^m$.
\end{enumerate}
For $\lambda\in [0,1]$, we define
\[
L^{\lambda} \coloneqq \left\{ \begin{array}{ll} \widehat{H}^{\lambda} & \mbox{on } \T \times B_{\frac{r}{2}}, \\[0.5ex] K^1 & \mbox{on } \T \times B_{\frac{r}{4}}^{\mathrm{c}}. \end{array} \right.
\]
This yields  a smooth family of Hamiltonians $\{L^{\lambda}\}_{\lambda\in [0,1]}$ on $\T\times \C^m$. One easily verifies that $K^1=L^0$ and that the 
smooth concatenation 
\[
H^{\lambda} \coloneqq \left\{ \begin{array}{ll} K^{\eta(2\lambda)} & \mbox{for } \lambda\in [0,\frac{1}{2}], \\ L^{\eta(2\lambda-1)} & \mbox{for } \lambda\in [\frac{1}{2},1] , \end{array} \right.
\]
satisfies the desired properties (i)-(v).

\bigskip

There remains to construct a family of smooth Hamiltonians $K^\lambda:\T\times \C^m \rightarrow \R$ satisfying (a')-(e'). For $\lambda\in [0,1]$, we define $\eta^\lambda(t) \coloneqq (1-\lambda)t + \lambda\eta(t)$. Note that $\eta^\lambda$ descends to a smooth map from $\mathbb{T}$ to itself. Therefore, we may define a smooth family of Hamiltonians $\widetilde{H}^\lambda:\T\times \C^m\rightarrow \R$, $\lambda\in [0,1]$, by the formula
\begin{equation*}
\widetilde{H}^\lambda(t,z) \coloneqq (\eta^\lambda)'(t) H(\eta^\lambda(t),z).
\end{equation*}
Note that $\widetilde{H}^0=H$ and $\widetilde{H}^1=\widehat{H}$ and that $\phi_{\widetilde{H}^\lambda}^1 = \phi_H^1$ for all $\lambda$. Choose a family of smooth functions $F^\lambda:\T \times \C^m\rightarrow \R$, $\lambda\in [0,1]$, such that
\begin{equation}
\label{FF}
F^\lambda= \widetilde{H}^\lambda \quad \mbox{on } \T \times B_{\frac{7}{10}r}, \qquad F^\lambda=H \quad \mbox{on } \T \times B_{\frac{8}{10}r}^{\mathrm{c}}.
\end{equation}
We can choose $F^0$ to be equal to $H$ everywhere because $\widetilde{H}^0 = H$. Since $H$ and $\widetilde{H}^\lambda$ are $C^2$-small, $F^\lambda$ can be made $C^2$-small as well. Therefore, the compactly supported symplectic diffeomorphism
\[
\psi^\lambda\coloneqq (\phi_{F^\lambda}^1)^{-1} \circ \phi_H^1
\]
is $C^1$-close to the identity and hence has a $C^2$-small compactly supported generating function $S^\lambda:\C^m \rightarrow \R$. By choosing $H$ to be $C^2$-small enough, we can ensure that
\begin{equation}
\label{piccoli}
\|X_H\|_{C^0} < {\textstyle \frac{r}{10}}, \qquad \|X_{F^\lambda}\|_{C^0} < {\textstyle \frac{r}{10}},
\end{equation}
and, together with \eqref{FF} and the fact that the time-$1$-map of the Hamiltonian $\widetilde{H}^\lambda$ is equal to $\phi^1_H$, we deduce that $\psi^\lambda$ is the identity on $B_{\frac{6}{10}r}$ and on $B_{\frac{9}{10}r}^{\mathrm{c}}$. The identity \eqref{genfun} guarantees that $\nabla S^\lambda$ vanishes on these two sets and hence
\begin{equation}
\label{valS}
S^\lambda = c^\lambda \quad \mbox{on } B_{\frac{6}{10}r}, \qquad  S^\lambda = 0 \quad \mbox{on } B_{\frac{9}{10}r}^{\mathrm{c}},
\end{equation}
for a suitable constant $c^\lambda\in \R$ which we will later show to be zero. Now let $\psi^\lambda_t$ be the symplectomorphism which is generated by $\eta(t) S^\lambda$. Then $\psi^\lambda_t$ depends smoothly on $t$, $\psi_t^\lambda=\mathrm{id}$ for $t$ close to $0$ and $\psi^\lambda_t=\psi^\lambda$ for $t$ close to 1.

Let $G^\lambda: [0,1]\times \C^m \rightarrow \R$ be the compactly supported Hamiltonian associated with the symplectic isotopy $\{\psi^\lambda_t\}_{t\in [0,1]}$. Then $G^\lambda_t=0$ for $t$ close to either $0$ or $1$, and hence we can see $G^\lambda$ as a smooth function on $\T\times \C^m$. The Hamilton--Jacobi equation \eqref{HJ} reads
\[
\eta'(t) S^\lambda(z) = G^\lambda_t \bigl( z + {\textstyle \frac{i}{2}} \eta(t) \nabla S^\lambda(z) \bigr).
\]
It implies that $G^\lambda$ is $C^1$-small. Moreover, thanks to \eqref{valS} it implies:
\begin{eqnarray}
\label{Gin}
&G^\lambda(t,z) = c^\lambda \, \eta'(t) \qquad \forall (t,z)\in \T \times B_{\frac{6}{10}r}, \\
\label{Gout}
&G^\lambda(t,z) = 0  \qquad \forall (t,z)\in \T \times B_{\frac{9}{10}r}^{\mathrm{c}}.
\end{eqnarray}
We set
\[
K^\lambda(t,z) \coloneqq (F^\lambda\# G^\lambda)(t,z) = F^\lambda(t,z) + G^\lambda \bigl( t, (\phi_{F^\lambda}^t)^{-1}(z) \bigr).
\]
This function is smooth on $\T\times \C^m$, compactly supported and satisfies
\[
\phi_{K^\lambda}^1 = \phi_{F^\lambda}^1 \circ \phi_{G^\lambda}^1 = \phi_{F^\lambda}^1 \circ \psi^\lambda = \phi_H^1,
\]
proving (d'). By \eqref{FF}, \eqref{piccoli} and \eqref{Gout}, $K^\lambda$ satisfies (c'). By \eqref{FF}, \eqref{piccoli} and \eqref{Gin}, we obtain
\[
K^\lambda(t,z) = \widetilde{H}^\lambda (t,z) + c^\lambda\, \eta'(t) \qquad \forall (t,z)\in \T \times B_{\frac{r}{2}}.
\]
The Hamiltonians $K^\lambda$ and $\widetilde{H}^\lambda$ induce the same time-one map, so the independence of the action of the fixed point $0$ from the choice of the Hamiltonian implies that the number $c^\lambda$ in the above identity is zero, as claimed above. This shows that $K^1$ satisfies (b'). Property (e') follows from the fact that $F^\lambda$ and $G^\lambda$ are both $C^0$-small. Since $F^0=H$, we have $G^0=0$ and thus $K^0=H$, showing (a'). This concludes the proof of Proposition \ref{modham}.
\end{proof}

\section{Averaging and a preliminary normal form}
\label{norforsec}

We denote by $\Omega^k(M)$, $k\geq 0$, the space of smooth $k$-forms on the manifold $M$. We assume that $M$ is closed, odd dimensional and endowed with a Zoll contact form $\alpha_0\in \Omega^1(M)$. We denote by $R_{\alpha_0}$ the Reeb vector field of $\alpha_0$ and by $\theta_t$ its flow. By the Zoll property, the flow $\theta_t$ defines a free, smooth $S^1$-action on $M$. Here, we identify $S^1$ with $\R/T_0 \Z$, where $T_0=\mathrm{sys}(\alpha_0)$ is the minimal period of the Reeb orbits of $\alpha_0$. Whenever we talk about $S^1$-invariant or $S^1$-equivariant objects on $M$, we refer to this $S^1$-action.

On the space $\Omega^k(M)$ we have the averaging operator mapping each $\beta\in \Omega^k(M)$ to the $S^1$-invariant $k$-form
\[
\overline{\beta} \coloneqq \frac{1}{T_0} \int_0^{T_0} \theta_t^* \beta\, dt.
\]
Averaging commutes with differentiation:
\[
\overline{d\beta} = d \overline{\beta} \qquad \forall \beta\in \Omega^k(M).
\]
Moreover, if $Y$ is an $S^1$-invariant vector field on $M$ (for instance $Y=R_{\alpha_0}$), we have
\[
\overline{\imath_Y \beta} = \imath_Y \overline{\beta}  \qquad \forall \beta\in \Omega^k(M).
\]
We shall also need to average linear endomorphisms $F: T^*M \rightarrow T^*M$ lifting the identity. The averaged endomorphism $\overline{F}: T^*M \rightarrow T^*M$ is defined by
\[
\overline{F} \coloneqq \frac{1}{T_0} \int_0^{T_0} \theta_t^* F\, dt,
\]
where the pull-back of $F$ by a diffeomorphism $\theta: M \rightarrow M$ is given by
\[
(\theta^* F)(x)[p] \coloneqq d\theta(x)^*\Big( F(\theta(x)) [ p \circ d\theta(x)^{-1} ]\Big), \qquad \forall x\in M, \; \forall p\in T_x^* M.
\]
If $F$ is as above and $\beta\in \Omega^1(M)$ is $S^1$-invariant, then the 1-form $F[\beta]$ satisfies
\[
\overline{F[\beta]} = \overline{F}[\beta].
\]
We fix an arbitrary $S^1$-invariant Riemannian metric on $M$. The $C^k$-norms of tensors on $M$ and the covariant derivatives we shall occasionally use in our proofs are induced by this metric. 

Our proof of Theorem \ref{qi} relies on the following normal form which is proven in \cite[Theorem 2]{ab23}.

\begin{thm}
\label{theorem:normal_form_abbondandolo_benedetti}
Let $\alpha_0$ be a Zoll contact form on a closed manifold $M$ with orbits having minimal period $T_0$. There is $\delta_0>0$ such that if $\alpha$ is a contact form on $M$ satisfying $\|\alpha-\alpha_0\|_{C^2}<\delta_0$, then there exists a diffeomorphism $u:M\rightarrow M$ such that
\begin{equation}
\label{norfor}
u^*\alpha = S\alpha_0 + \eta + df,
\end{equation}
where:
\begin{enumerate}[(i)]
\item $S$ is a smooth positive function on $M$ that is invariant under the Reeb flow of $\alpha_0$;
\item $f$ is a smooth function on $M$ with average zero along each orbit of $R_{\alpha_0}$;
\item $\eta$ is a smooth one-form on $M$ satisfying $\iota_{R_{\alpha_0}}\eta=0$;
\item $\iota_{R_{\alpha_0}}d\eta = F[dS]$ for a smooth endomorphism $F:T^*M\rightarrow T^*M$ lifting the identity;
\item $\iota_{R_{\alpha_0}}df=\iota_ZdS$ for a smooth vector field $Z$ on $M$ taking values in the contact distribution $\operatorname{ker}\alpha_0$ and having average zero along each orbit of $R_{\alpha_0}$;
\item $dS=-B[V]$, where $V$ is a smooth $S^1$-invariant section of $\ker \alpha_0$ and $B: \ker \alpha_0 \rightarrow (\ker \alpha_0)^*$ is a smooth isomorphism lifting the identity.
\end{enumerate}
Moreover, for every integer $k\geq 0$, there is a modulus of continuity $\omega_k$ such that
\begin{equation}
\label{eq:normal_form_abbondandolo_benedetti_estimate}
\begin{split}
\max \bigl\{ \operatorname{dist}_{C^{k+1}}(u,\operatorname{id}), \|S-1\|_{C^{k+1}}, \|f\|_{C^{k+1}}, \|\eta\|_{C^k}, \|d\eta\|_{C^k}, \|F\|_{C^k}, \|Z\|_{C^k}, \\   \|V\|_{C^{k+1}}, \|B-B_0\|_{C^k}
\bigr\} \leq \omega_k(\|\alpha-\alpha_0\|_{C^{k+2}}),
\end{split}
\end{equation}
where $B_0 : \ker \alpha_0 \rightarrow (\ker \alpha_0)^*$ is the isomorphism given by the non-degenerate bilinear form $d\alpha_0|_{\ker \alpha_0 \times \ker \alpha_0}$. Finally, the map $\alpha\mapsto (u,S,\eta,f)$ is smooth and maps $\alpha_0$ to $(\mathrm{id},1,0,0)$.
\end{thm}

Here, by modulus of continuity we mean a monotonically increasing continuous function $\omega:[0,+\infty) \rightarrow [0,+\infty)$ such that $\omega(0)=0$. 

\begin{rem}
\label{deduced}
{\rm Actually, statement (vi) and the corresponding bounds for $V$ and $B$ in \eqref{eq:normal_form_abbondandolo_benedetti_estimate} are not present in the formulation of \cite[Theorem 2]{ab23}, but they are easily recoverable from its proof (see in particular equation (2.25) together with (2.4) and (2.24) therein).

The smoothness which is mentioned at the end of the theorem is not explicitly stated in \cite[Theorem 2]{ab23} either, but can be deduced from its proof. This amounts to checking that in \cite[Theorem B.1]{ab23} the vector fields $U$, $V$ and the function $h$ depend smoothly on the vector field $X$, and take the values $U=0$, $V=0$, $h=1$ when $X=X_0$. The triplet $(U,V,h)$ is determined by $X$ by means of a functional equation of the form $\Phi_X(U,V,h)=0$, where the dependence of $\Phi$ on $X$ is affine and $\Phi_0(0,0,1)=0$. A standard argument involving the parametric inverse mapping theorem and \cite[Lemma B.4]{ab23} gives us the required smooth dependence on $X$ of the solution $(U,V,h)$ of this equation.}
\end{rem}
We will also need the following additional fact. 

\begin{prop}
\label{proposition:normal_form_abbondandolo_benedetti_short_orbits}
In the setting of Theorem \ref{theorem:normal_form_abbondandolo_benedetti}, the following additional properties hold: for every $\epsilon>0$ there is a positive number $\delta\leq \delta_0$ such that if $\|\alpha-\alpha_0\|_{C^2}<\delta$ then:
\begin{enumerate}[(i)]
\setcounter{enumi}{6}
\item any closed orbit  of $R_{\alpha}$ is either long, meaning that its minimal period is larger than $\frac{1}{\epsilon}$, or short, meaning that its minimal period is contained in the interval $(T_0 - \epsilon, T_0 + \epsilon)$;  
\item the short closed orbits of $R_{\alpha}$ are precisely the images via $u$ of those orbits $\gamma$ of the free $S^1$-action $\theta_t$ consisting of critical points of $S$, and the minimal period of such an orbit is $T_0 S(\gamma)$.
\end{enumerate}
\end{prop}

The above proposition sharpens \cite[Proposition 1 and Remark 3.1]{ab23}, in which the fact that all short closed orbits of $R_{\alpha}$ are given by the critical points of $S$ is proven assuming that $\alpha$ is $C^3$-close to $\alpha_0$. The  proof of Proposition \ref{proposition:normal_form_abbondandolo_benedetti_short_orbits} uses the following lemma.

\begin{lem}
\label{non-autonomous}
There exists $\nu>0$ such that the following is true. Let $V$ be an autonomous smooth vector field on $\R^k$ with $\|V\|_{C^1} < \nu$ and $A: [0,1]\times \R^k \rightarrow \mathrm{Hom}(\R^k,\R^k)$ a smooth map into the space of linear endomorphisms of $\R^k$ with $\|A\|_{C^0} < \nu$. Let us define the time-dependent vector field $W_t\coloneqq (\operatorname{id}+A_t)V$ and let $(\chi_t)_{t\in [0,1]}$ be its flow. Then the fixed points of $\chi_1$ are precisely the zeros of the vector field $V$. 
\end{lem}

\begin{proof}
Every zero of $V$ is also a zero of $W_t$ for all $t$ and hence a fixed point of $\chi_1$. Let $p\in \R^k$ be a point such that $V(p)\neq 0$. We need to show that $\chi_1$ does not fix $p$ if $\nu$ is small enough. We claim that there is $R>0$ such that for all $(t,x)\in [0,1]\times B_R(p)$, we have 
\[
\mathrm{(a)}\quad |W_t(x)|< R,\qquad \mathrm{(b)}\quad |W_t(x)-V(p)|<|V(p)|,
\]
where $B_R(p)$ denotes the open ball of radius $R$ centered at $p$. Let us show how to use the claim to prove the lemma. Call $\gamma\colon[0,1]\to \R^k$ the trajectory of $W_t$ starting at $p$, that is, $\gamma(t)=\chi_t(p)$ for all $t\in[0,1]$. We need to show that $\gamma$ is not closed. Let $s\in [0,1]$ be such that $\gamma([0,s])$ is contained in $B_R(p)$. By (a), 
\[
|\gamma(s)-p|=|\gamma(s)-\gamma(0)|\leq\int_0^s|\dot\gamma(t)|dt=\int_0^s|W_t(\gamma(t))|dt< Rs.
\]
Thus, $\gamma(s)\in B_{sR}(p)$ and we deduce that $\gamma([0,1])$ is contained in $B_R(p)$. Denoting by $\langle\cdot,\cdot\rangle$ the euclidean inner product in $\R^k$, for all $t\in[0,1]$ we have by (b)
\begin{align*}
\frac{d}{dt}\langle \gamma(t),V(p)\rangle=\langle \dot\gamma(t),V(p)\rangle=\langle W_t(\gamma(t)),V(p)\rangle&\geq |V(p)|^2-|W_t(\gamma(t)-V(p)||V(p)|\\
&>|V(p)|^2-|V(p)||V(p)|\\
&= 0.
\end{align*}
Therefore, $t\mapsto \langle \gamma(t),V(p)\rangle$ is a strictly increasing function and hence $\gamma$ is not closed.

We are left to prove the claim. Set $r\coloneqq |V(p)|>0$. Let $\nu>0$ and write $R=ar$ for some positive number $a>0$ to be determined. We have
\[
|V(x)-V(p)| \leq R\|V\|_{C^1} < a\nu r, \qquad \forall x\in B_R(p).
\]
Thus,
\[
|V(x)| \leq |V(x)-V(p)| + |V(p)| < (1+a\nu)r, \qquad \forall x\in B_R(p).
\]
For all $(t,x)\in [0,1]\times B_R(p)$, it follows that
\begin{align*}
|W_t(x)| &< (1+\nu)(1+a\nu)r,\\
|W_t(x) - V(p)| &< |A_t(x)V(x)| + a\nu r < (1+a+a\nu)\nu r.
\end{align*}
To achieve (a) and (b), we need to show that for every sufficiently small $\nu>0$, there exists $a>0$ such that
\[
(1+\nu)(1+a\nu)<a \quad \text{and} \quad (1+a+a\nu)\nu <1.
\]
If $\nu>0$ is sufficiently small, these two inequalities are equivalent to
\[
\frac{1+\nu}{1-\nu-\nu^2} < a < \frac{1-\nu}{\nu(1+\nu)}
\]
and clearly possess a solution $a>0$.
\end{proof}

\begin{proof}[Proof of Proposition \ref{proposition:normal_form_abbondandolo_benedetti_short_orbits}]
If the contact form $\alpha$ is $C^2$-close to $\alpha_0$ then the Reeb vector field $R_{\alpha}$ is $C^1$-close to $R_{\alpha_0}$. Statement (vii) is then a consequence of a general fact about $C^1$-perturbations of vector fields inducing a free $S^1$-action, see \cite[Corollary 1]{ban86}. 

By differentiating \eqref{norfor} and contracting along $R_{\alpha_0}$, we obtain the identity
\[
\imath_{R_{\alpha_0}} u^* \alpha = \imath_{R_{\alpha_0}} ( dS \wedge \alpha_0 + S\, d\alpha_0 + d\eta ) = - dS + F[dS],
\]
which shows that the Reeb vector field of $u^* \alpha$ is parallel to $R_{\alpha_0}$ on the critical set of $S$. Therefore, every orbit $\gamma$ of $\theta_t$ consisting of critical points of the $S^1$-invariant function $S$ is a closed orbit of $R_{u^* \alpha}$ of minimal period 
\[
\int_{\gamma} u^* \alpha = \int_{\gamma} ( S\alpha_0 + \eta + df) = T_0 S(\gamma).
\]
By \eqref{eq:normal_form_abbondandolo_benedetti_estimate}, this number is close to $T_0$ when $\|\alpha-\alpha_0\|_{C^2}$ is small, so $u(\gamma)$ is a short closed orbit of $R_{\alpha}$ if $\|\alpha-\alpha_0\|_{C^2}<\delta$ with $\delta=\delta(\epsilon)$ small enough. There remains to show that, up to reducing the number $\delta$ if necessary, every short closed orbit of $R_{\alpha}$ is obtained in this way or, equivalently, that every short closed orbit of $u^* R_{\alpha}$ is an orbit of the $S^1$-action $\theta_t$ consisting of critical points of $S$.

The diffeomorphism $u$ is given by a version of a theorem of Bottkol \cite{bot80} which is proven in \cite[Theorem 2.1]{ab23}. By this result, $u$ satisfies 
\begin{equation}
\label{after-bottkol}
h \, u^* R_{\alpha} = R_{\alpha_0} - Q[V],
\end{equation}
where $h$ is a smooth $S^1$-invariant function on $M$, $V$ is a smooth $S^1$-invariant vector field which is orthogonal to $R_{\alpha_0}$ with respect to the chosen $S^1$-invariant metric on $M$, and $Q$ is an endomorphism of $TM$ lifting the identity. Moreover
\begin{equation}
\label{bottkol-bounds}
\max\{ \|h-1\|_{C^1},\|V\|_{C^1}, \|Q-\mathrm{id}\|_{C^0} \} \leq \omega(\|\alpha-\alpha_0\|_{C^2}),
\end{equation}
for a suitable modulus of continuity $\omega$, and the zeros of the $S^1$-invariant vector field $V$ are precisely the critical points of the $S^1$-invariant function $S$, see \cite[Equation (2.25)]{ab23}.

Therefore, we have to show that if $\|\alpha-\alpha_0\|_{C^2}$ is small enough, then the short closed orbits of the vector field $u^* R_{\alpha}$ given by \eqref{after-bottkol}, where $h$, $Q$ and $V$ satisfy the above conditions, are orbits of $\theta_t$ which are contained in the set of zeroes of $V$.

The free $S^1$-action $\theta_t$ defines the smooth $S^1$-bundle
\[
S^1=\R/T_0 \Z \rightarrow M \stackrel{\pi}{\longrightarrow} B.
\]
Let $\{B_j\}_{j\in \{1,\dots,k\}}$ be a covering of $B$ consisting of open embedded balls and consider a second covering $\{B_j'\}_{j\in \{1,\dots,k\}}$ with $\overline{B_j'} \subset B_j$ for every $j$. By \eqref{bottkol-bounds}, the vector field $u^* R_{\alpha}$ is $C^0$-close to $R_{\alpha_0}$ when $\|\alpha-\alpha_0\|_{C^2}$ is small, so we can assume that any short closed orbit of $u^* R_{\alpha}$ which meets $\pi^{-1}(B_j')$ is fully contained in $\pi^{-1}(B_j)$.  This allows us to fix $j$ and work in $\pi^{-1}(B_j)$.

By a suitable bundle trivialization, we can identify $\pi^{-1}(B_j)$ with the product $S^1\times B_j$ in such a way that, denoting by $t$ the variable in $S^1$, we have
\[
R_{\alpha_0} = \partial_t \qquad \mbox{and} \qquad V(t,x) = V(x) \in T_x B_j  \quad \forall (t,x)\in S^1\times B_j.
\]
From \eqref{after-bottkol}, we obtain that on $\pi^{-1}(B_j)\cong S^1\times B_j$ the vector field $u^* R_{\alpha}$ satisfies
\[
g \, u^* R_{\alpha} = \partial_t + Q'[V],
\]
where the smooth function $g:S^1\times B_j \rightarrow \R$ is $C^0$-close to the constant function $1$ and $Q'$ is an endomorphism of $TB_j$ which lifts the identity and is $C^0$-close to the identity. Therefore, up to a time reparametrization which is determined by the function $g$, the short closed orbits of $u^* R_{\alpha}$ which are contained in $\pi^{-1}(B_j)$ correspond to the $T_0$-periodic orbits of the non-autonomous $T_0$-periodic vector field $Q'(t,x)[V(x)]$ which are contained in $B_j$. The desired conclusion that these $T_0$-periodic orbits are just the zeros of $V$ when $\|\alpha-\alpha_0\|_{C^2}$ is sufficiently small now follows from Lemma \ref{non-autonomous}. This concludes the proof of statement (viii).
\end{proof}

\section{Proof of Theorem \ref{qi}}
\label{qisec}

In this section, we prove Theorem \ref{qi}. Let $\alpha$ be a contact form on $M$ which is $C^2$-close to the Zoll contact form $\alpha_0$. By Theorem \ref{theorem:normal_form_abbondandolo_benedetti}, there exist a diffeomorphism $u\in\operatorname{Diff}(M)$, a positive $S^1$-invariant function $S\in C^\infty(M)$, a $1$-form $\eta\in\Omega^1(M)$ and a function $f\in C^\infty(M)$ satisfying the properties listed in Theorem \ref{theorem:normal_form_abbondandolo_benedetti}. In particular,
\[
u^*\alpha = S\alpha_0 + \eta + df.
\]

We claim that $S\alpha_0+\eta$ is a contact form strictly contactomorphic to $S\alpha_0+\eta+df$, i.e., that there exists a diffeomorphism $v: M \rightarrow M$ such that
\begin{equation}
\label{eq:normal_form_new_proof_chi}
v^*(S\alpha_0+\eta+df) = S\alpha_0+\eta.
\end{equation}
To see this, let us define the family of $1$-forms
\[
\beta_t\coloneqq S\alpha_0 + \eta + t \, df, \qquad t\in[0,1].
\]
Note that since $\alpha$ is $C^2$-close to $\alpha_0$, it follows from \eqref{eq:normal_form_abbondandolo_benedetti_estimate} that $S-1$ and $f$ are $C^1$-small and that $\eta$ and $d\eta$ are $C^0$-small. This fact implies that $\beta_t$ is $C^0$-close to $\alpha_0$ and that $d\beta_t$ is $C^0$-close to $d\alpha_0$. Therefore, $\beta_t$ is a contact form for all $t\in[0,1]$. Hence there exists a unique time-dependent vector field $B_t$ on $M$ which is parallel to the Reeb vector field $R_{\beta_t}$ and satisfies $\iota_{B_t}\beta_t+f=0$. Let $v_t$ denote the flow generated by $B_t$. We compute:
\[
\partial_t ( v_t^*\beta_t ) = v_t^*(\mathcal{L}_{B_t}\beta_t+\partial_t\beta_t) = v_t^*(d\iota_{B_t}\beta_t+\iota_{B_t}d\beta_t+df) = v_t^*d(\iota_{B_t}\beta_t+f)=0.
\]
Therefore, we have $v_t^*\beta_t=\beta_0$ for all $t\in[0,1]$ and hence the diffeomorphism $v\coloneqq v_1$ satisfies identity \eqref{eq:normal_form_new_proof_chi}.\\

Next, let us define the family of $1$-forms
\[
\gamma_t\coloneqq S\alpha_0 + (1-t)\overline{\eta} + t\eta,\qquad t\in[0,1],
\]
where $\overline{\eta}$ is the average of $\eta$ with respect to the $S^1$-action. Again, since $S-1$ is $C^1$-small and $\eta$ and $d\eta$ are $C^0$-small, $\gamma_t$ is a contact form for all $t\in[0,1]$. We apply Moser's homotopy argument: Let $C_t$ be the unique time-dependent vector field on $M$ satisfying
\[
\begin{cases}
C_t\in \operatorname{ker}\gamma_t,\\
(\iota_{C_t}d\gamma_t+\partial_t\gamma_t)|_{\operatorname{ker}\gamma_t}=0.
\end{cases}
\]
Let $w_t$ be the flow generated by $C_t$. Then
\[
w_t^*\gamma_t = e^{g_t}\gamma_0
\]
where $g_t$ is the unique solution of
\begin{equation}
\label{defg}
\begin{cases}
\partial_t g_t = w_t^* (\iota_{R_{\gamma_t}}\partial_t\gamma_t)\\
g_0=0.
\end{cases}
\end{equation}
Set $g\coloneqq g_1$ and $w\coloneqq w_1$. Then we have
\[
w^*(S\alpha_0+\eta)= e^g(S\alpha_0+\overline{\eta}).
\]
We proceed with the following claim:

\begin{claim}
\label{claim:normal_form_new_proof_g_estimate}
The function $g$ satisfies the pointwise bound
\begin{equation}
\label{bound-on-g}
|g|\leq \sigma(\|\alpha-\alpha_0\|_{C^2})\cdot|dS|^2,
\end{equation}
for some modulus of continuity $\sigma$. In particular, $g$ and $dg$ vanish on the critical set of $S$.
\end{claim}

\begin{proof}
In order to prove \eqref{bound-on-g}, we begin by showing that
\begin{equation}
\label{eq:normal_form_new_proof_difference_eta_avg_estimate}
|\eta-\overline{\eta}| \leq \frac{1}{2}T_0\|F\|_{C^0}\cdot |dS|.
\end{equation}
It follows from the properties of $\eta$ and $F$ listed in Theorem \ref{theorem:normal_form_abbondandolo_benedetti} that
\[
\mathcal{L}_{R_{\alpha_0}}\eta = d\iota_{R_{\alpha_0}}\eta + \iota_{R_{\alpha_0}}d\eta = F[dS].
\]
Together with the $S^1$-invariance of $S$, we deduce the pointwise estimate
\[
|\theta_t^*\eta-\eta| = \Bigl|\int_0^t \partial_\tau \theta_\tau^*\eta\, d\tau \Bigr| = \Bigl|\int_0^t \theta_\tau^* (F[dS]) \, d\tau \Bigr| \leq |t|\cdot \|F\|_{C^0}\cdot |dS|,
\]
from which the desired estimate \eqref{eq:normal_form_new_proof_difference_eta_avg_estimate} follows:
\[
|\eta-\overline{\eta}| =  \Bigl|\frac{1}{T_0}\int_0^{T_0}(\eta - \theta_t^*\eta) \, dt \Bigr| 
 \leq \frac{1}{T_0} \int_0^{T_0} |t|\cdot \|F\|_{C^0}\cdot  |dS|\, dt \leq  \frac{1}{2}T_0\|F\|_{C^0}\cdot |dS|.
\]

Recall that the vector field $C_t$ is characterized by $C_t\in \operatorname{ker}\gamma_t$ and
\[
(\iota_{C_t}d\gamma_t + \eta-\overline{\eta})|_{\operatorname{ker}\gamma_t}=0.
\]
Since $\gamma_t$ is $C^0$-close to $\alpha_0$ and $d\gamma_t$ is $C^0$-close to $d\alpha_0$, this implies together with \eqref{eq:normal_form_new_proof_difference_eta_avg_estimate} that
\begin{equation}
\label{eq:normal_form_new_proof_C_estimate}
|C_t| \leq b \, \|F\|_{C^0}\cdot|dS|
\end{equation}
for some constant $b>0$ which can be chosen uniform among all $\alpha$ which are sufficiently $C^2$-close to $\alpha_0$.

Let us split $R_{\gamma_t} = X_t + Y_t$ into a vector field $X_t$ parallel to $R_{\alpha_0}$ and a vector field $Y_t$ taking values in the contact distribution $\operatorname{ker}\alpha_0$. We can write $X_t = a_t \cdot R_{\alpha_0}$ where $a_t\in C^\infty(M)$ is  $C^0$-close to the constant function $1$. We compute
\begin{eqnarray*}
    \iota_{X_t} d\gamma_t  & = & \iota_{X_t} (dS\wedge \alpha_0 + S d\alpha_0 + (1-t)d\overline{\eta} + td\eta) \\
    & = & a_t (- dS + (1-t)\overline{F}[dS] +  t F[dS] ).
\end{eqnarray*}
Here we use that $\iota_{R_{\alpha_0}}dS=0$ because $S$ is $S^1$-invariant. Moreover, we use the identity $\iota_{R_{\alpha_0}}d\eta = F[dS]$ from Theorem \ref{theorem:normal_form_abbondandolo_benedetti}. Combining the above computation with the identity $0=\iota_{R_{\gamma_t}}d\gamma_t = \iota_{X_t+Y_t}d\gamma_t$, we obtain
\[
    \iota_{Y_t} d\gamma_t = - a_t (-dS + (1-t)\overline{F}[dS] +  t F[dS] ).
\]
Since $\operatorname{ker}\gamma_t$ is $C^0$-close to $\operatorname{ker}\alpha_0$ and $d\gamma_t$ is $C^0$-close to $d\alpha_0$, this implies the pointwise estimate
\[
    |Y_t| \leq b|dS|
\]
for some constant $b>0$ independent of the contact form $\alpha$. In particular, we obtain, in combination with \eqref{eq:normal_form_new_proof_difference_eta_avg_estimate}, that
\[
    |\iota_{R_{\gamma_t}}\partial_t\gamma_t| = |\iota_{X_t+Y_t} (\eta-\overline{\eta})| = |\iota_{Y_t}(\eta-\overline{\eta})| \leq |Y_t||\eta-\overline{\eta}| \leq b\|F\|_{C^0}\cdot |dS|^2,
\]
for some constant $b$ independent of $\alpha$. Using \eqref{defg}, this yields the estimate
\begin{equation}
\label{gdS}
    |g| \leq \int_0^1 |\iota_{R_{\gamma_t}}\partial_t\gamma_t|\circ w_t \, dt
    \leq b\|F\|_{C^0} \int_0^1 |dS(w_t)|^2 \, dt
\end{equation}
In order to bound the integrand in the above expression, we argue as follows. By assertion (vi) in Theorem \ref{theorem:normal_form_abbondandolo_benedetti}, we have
\[
dS = - B[V],
\]
where, provided $\alpha$ is $C^{k+2}$-close to $\alpha_0$, $V$ is a $C^{k+1}$ small, $S^1$-invariant section of $\ker \alpha_0$, and $B:\ker \alpha_0 \rightarrow (\ker\alpha_0)^*$ is an isomorphism which is $C^k$-close to the isomorphism $B_0$ induced by the non-degenerate bilinear form $d\alpha_0$ on $\ker \alpha_0$. Since we are assuming that $\alpha$ is $C^2$-close to $\alpha_0$, $V$ is $C^1$-small and $B$ is $C^0$-close to $B_0$. In particular, both $B$ and $B^{-1}$ are uniformly bounded and we obtain the pointwise bounds
\begin{equation}
\label{pwbds}
\frac{1}{c} |V| \leq |dS| \leq c |V|,
\end{equation}
for a suitable number $c\geq 1$. Therefore, from \eqref{gdS} we deduce the bound
\[
|g| \leq bc^2\|F\|_{C^0} \int_0^1 |V(w_t)|^2 \, dt.
\]
Next we notice that
\[
\begin{split}
    \frac{d}{dt} |V(w_t)|^2  &=  2\langle V, \nabla_{C_t}V\rangle\circ w_t 
    \leq 2  \|V \|_{C^1} |V(w_t)| \cdot |C_t(w_t)| \\
    &\leq  2b\|V\|_{C^1} \|F\|_{C^0}  |V(w_t)| \cdot |dS(w_t)| \leq 2 bc\|V\|_{C^1} \|F\|_{C^0}  |V(w_t)|^2,
\end{split}
\]
where we have used estimate \eqref{eq:normal_form_new_proof_C_estimate} and the second inequality in \eqref{pwbds}. It follows from Gronwall's inequality that, for all $t\in[0,1]$,
\[
 |V(w_t)|^2 \leq  e^{2bc\|V\|_{C^1}\|F\|_{C^0} t}\cdot |V|^2.
\]
Thus
\begin{equation}
\label{eq:normal_form_new_proof_g_estimate}
    |g|\leq  bc^2\|F\|_{C^0}\cdot \int_0^1 e^{2bc\|V\|_{C^1}\|F\|_{C^0} t}\,  dt \cdot |V|^2 
    \leq  bc^2\|F\|_{C^0}\cdot  e^{2bc\|V\|_{C^1}\|F\|_{C^0}} \cdot|V|^2.
\end{equation}
From \eqref{eq:normal_form_abbondandolo_benedetti_estimate} it follows that
\[
\max\bigl\{ \|V\|_{C^1}, \|F\|_{C^0} \bigr\}\leq\omega_0( \|\alpha-\alpha_0\|_{C^2}).
\]
Inequality \eqref{eq:normal_form_new_proof_g_estimate} and the first estimate in \eqref{pwbds} therefore imply a pointwise bound of the form 
\[
|g|\leq \sigma(\|\alpha-\alpha_0\|_{C^2})\cdot|dS|^2,
\]
for some modulus of continuity $\sigma$. This concludes the proof of Claim \ref{claim:normal_form_new_proof_g_estimate}.
\end{proof}

Our next task is to prove the following claim.

\begin{claim}
\label{claim:min_and_max}
If $\|\alpha-\alpha_0\|_{C^2}$ is small enough, then 
\begin{equation}
\label{minmax}
\min_M Se^g = \min_M S, \qquad \max_M S e^g = \max_M S.
\end{equation}
\end{claim}

\begin{proof}
We claim that if $\|\alpha-\alpha_0\|_{C^2}$ is small enough, then $S$ satisfies the pointwise bound
\begin{equation}
\label{pwbdS}
S \geq \min S + \frac{1}{4c^4} |dS|^2,
\end{equation}
where $c$ is the number appearing in \eqref{pwbds}.  

As already recalled, the vector field $V$ is $C^1$-small when $\|\alpha-\alpha_0\|_{C^2}$ is small. Therefore, when $\|\alpha-\alpha_0\|_{C^2}$ is sufficiently small the following fact holds: for any $p\in M$ with $r:= |V(p)| > 0$, setting $R:=3rc$ we have
\begin{equation}
\label{VoBR}
\frac{r}{2} \leq |V| \leq 2r \qquad \mbox{on } B_R(p),
\end{equation}
where $B_R(p)\subset M$ denotes the ball of radius $R$ centered at $p$.

The bound \eqref{pwbdS} trivially holds at critical points of $S$. Let $p\in M$ be a point such that $dS(p)\neq 0$. By \eqref{pwbds} the number $r= |V(p)|$ is positive and \eqref{VoBR} holds with $R=3rc$. Now let $\gamma:\R \rightarrow M$ be the solution of the Cauchy problem
\[
\gamma'(t) = - \nabla S(\gamma(t)), \qquad \gamma(0)=p.
\]
If $\gamma([0,s])$ is contained in $B_R(p)$ for some $s\in [0,1]$, then by \eqref{pwbds} and \eqref{VoBR} we have
\[
\begin{split}
\mathrm{dist}(\gamma(s),p) = \mathrm{dist}(\gamma(s),\gamma(0)) &\leq \int_0^s |\gamma'(t)|\, dt = \int_0^s |dS(\gamma(t))|\, dt \\ &\leq c \int_0^s |V(\gamma(t))|\, dt \leq 2rcs \leq 2rc < R.
\end{split}
\]
The above inequality implies that $\gamma([0,1])\subset B_R(p)$. Therefore, using again \eqref{pwbds} and \eqref{VoBR}, we obtain the chain of inequalities
\[
\begin{split}
\min S \leq S(\gamma(1)) &= S(p) - \int_0^1 |dS(\gamma(t))|^2\, dt \leq S(p) - \frac{1}{c^2} \int_0^1 |V(\gamma(t))|^2\, dt \\ & \leq S(p) - \frac{1}{4c^2}  r^2=S(p) - \frac{1}{4c^2}  |V(p)|^2 \leq S(p) - \frac{1}{4c^4}  |dS(p)|^2,
\end{split}
\]
proving \eqref{pwbdS}.

We can now verify the first identity in \eqref{minmax}. Let $q\in M$ be a point at which $S$ achieves its minimum. Then $dS(q)=0$ and hence $g(q)=0$ thanks to Claim \ref{claim:normal_form_new_proof_g_estimate}. From this, we obtain the inequality
\[
\min_M Se^g \leq S(q) e^{g(q)} = S(q) = \min_M S.
\]
In order to prove the opposite inequality, we consider the pointwise bound
\begin{equation}
\label{bound-e^gS}
S e^g  \geq S (1+g)  \geq S - 2 |g| \geq \min_M S + \Bigl( \frac{1}{4 c^4} - 2 \sigma(\|\alpha-\alpha_0\|_{C^2}) \Bigr) |dS|^2.
\end{equation}
Here, the first inequality follows from the convexity of the exponential function, the second one from the fact that we can assume that $S\leq 2$, as by \eqref{eq:normal_form_abbondandolo_benedetti_estimate} $S$ is $C^0$-close to the constant function 1 when $\|\alpha-\alpha_0\|_{C^2}$ is small, and the third inequality follows from \eqref{bound-on-g} and \eqref{pwbdS}. 

If $\|\alpha-\alpha_0\|_{C^2}$ is small enough, then the term in brackets at the end of \eqref{bound-e^gS} is non-negative and hence
 \[
 \min_M Se^g \geq \min_M S.
 \]
This proves the first identity in \eqref{minmax}. The proof of the second one is analogous.
\end{proof}

Let us define the family of one-forms
\[
\delta_t \coloneqq S\alpha_0 + t\overline{\eta}, \qquad t\in[0,1].
\]
Since $S-1$ is $C^1$-small and $\eta$ and $d\eta$ are $C^0$-small when $\alpha$ is $C^2$-close to $\alpha_0$, see \eqref{eq:normal_form_abbondandolo_benedetti_estimate}, this is a contact form for all $t\in[0,1]$. Let $D_t$ denote the time-dependent vector field on $M$ characterized by 
\[
\begin{cases}
D_t\in \operatorname{ker}\delta_t,\\
(\iota_{D_t}d\delta_t+\partial_t\delta_t)|_{\operatorname{ker}\delta_t}=0.
\end{cases}
\]
We observe that since both $S$ and $\overline{\eta}$ are $S^1$-invariant, the vector field $D_t$ is $S^1$-invariant as well. Let $\psi_t$ be the flow generated by $D_t$. Since $D_t$ is $S^1$-invariant, this flow is $S^1$-equivariant. Moreover, we have
\[
\psi_t^*\delta_t = e^{d_t}\delta_0,
\]
where $d_t\in C^{\infty}(M)$ is the unique solution of
\[
\begin{cases}
\partial_t d_t = \psi_t^* (\iota_{R_{\delta_t}}\partial_t\delta_t)\\
d_0=0.
\end{cases}
\]
We claim that $\psi_1^*\delta_1=(S\circ\psi_1)\alpha_0$. Indeed, we compute
\[
\iota_{R_{\alpha_0}}\psi_1^*\delta_1 = \psi_1^*(\iota_{(\psi_1)_*R_{\alpha_0}}\delta_1) = \psi_1^*(\iota_{R_{\alpha_0}}\delta_1)= S\circ\psi_1.
\]
On the other hand, we have
\[
\iota_{R_{\alpha_0}}\psi_1^*\delta_1 = \iota_{R_{\alpha_0}}e^{d_1}\delta_0 = S e^{d_1}.
\]
Therefore, we have $S\circ\psi_1=S e^{d_1}$. This implies that
\[
\psi_1^*\delta_1 = e^{d_1}\delta_0 = S e^{d_1}\alpha_0 = (S\circ\psi_1)\alpha_0.
\]
Let us abbreviate $\psi\coloneqq \psi_1$. Then the above computations show:
\[
\psi^*(S\alpha_0+\overline{\eta}) = (S\circ\psi)\alpha_0.
\]

Let us define the diffeomorphism $\varphi$ of $M$ by $\varphi\coloneqq u\circ v\circ w\circ\psi$. Then
\begin{align*}
\varphi^*\alpha &=  \psi^*w^*v^*u^*\alpha = \psi^*w^*v^*(S\alpha_0+\eta+df) = \psi^*w^*(S\alpha_0+\eta) = 
 \psi^*(e^g(S\alpha_0+\overline{\eta}))\\
 &= \big((e^gS)\circ\psi\big)\alpha_0.
\end{align*}
Therefore, setting
\[
T\coloneqq S\circ \psi, \qquad h\coloneqq g\circ \psi,
\]
we obtain that the diffeomorphism $\varphi$ satisfies the desired identity
\[
\varphi^* \alpha = T e^h \, \alpha_0.
\]

We now check that properties (a)-(e) from Theorem \ref{qi} hold. Since $S$ is $S^1$-invariant and $\psi$ is $S^1$-equivariant, the function $T$ is $S^1$-invariant, proving (a). By Claim \ref{claim:normal_form_new_proof_g_estimate}, both $g$ and $dg$ vanish on the critical set of $S$. Therefore, the same is true for the pull-backs $h$ and $T$ under the diffeomorphism $\psi$, proving property (b). Property (c) follows from Claim \ref{claim:min_and_max}. Property (d) follows from property (vii) in Proposition \ref{proposition:normal_form_abbondandolo_benedetti_short_orbits}. By property (viii) from the same proposition, the short closed orbits of $R_{u^*\alpha}$ are precisely the orbits $\gamma$ of the free $S^1$-action $\theta_t$ that consist of critical points of $S$, and the minimal period of such an orbit is $T_0 S(\gamma)=T_0 T(\psi^{-1}(\gamma))$. Since $u^* \alpha = S\alpha_0+\eta+df$ and $v^*u^*\alpha=S\alpha_0+\eta$ induce the same Reeb flow up to a time reparametrization which preserves the periods of closed orbits, the same is true for the short closed Reeb orbits of $v^*u^*\alpha$.
By estimate \eqref{eq:normal_form_new_proof_C_estimate}, the diffeomorphism $w$ fixes the critical set of $S$, so the same fact continues to hold for the short closed Reeb orbits of $w^* v^*u^*\alpha$. By applying the $S^1$-equivariant diffeomorphism $\psi$, we obtain that the short closed Reeb orbits of $\varphi^*\alpha$ are the  orbits of $\theta_t$ forming the image by $\psi^{-1}$ of the critical set of $S$. The latter set is precisely the critical set of $T$ and property (e) follows. 

The last claim in Theorem \ref{qi} is that the map $\alpha\mapsto (\varphi,T,h)$ we have just constructed is smooth and maps $\alpha_0$ to $(\mathrm{id},1,0)$. This follows from the corresponding statement about the map $\alpha\mapsto (u,S,\eta,f)$ in Theorem \ref{theorem:normal_form_abbondandolo_benedetti} and from the fact that the tuple $(v,w,g,\psi)$ we constructed above depends smoothly on the defining data $(S,\eta,f)$ and takes the value $(\mathrm{id},\mathrm{id},0,\mathrm{id})$ for $(S,\eta,f)=(1,0,0)$. This concludes the proof of Theorem \ref{qi}. 

\section{Proof of Theorem \ref{1stcap} (ii)}
\label{proof-of-thm-1}

In this section, we deduce Theorem \ref{1stcap} (ii) from Theorem \ref{qi}. Recall that $\mathcal{A}$ denotes the set of domains in $\C^n$ of the form
\[
A_f = \{ r z \mid z\in S^{2n-1},\ 0\leq r< f(z) \},
\]
where $f$ is a positive smooth function on $S^{2n-1}$, and that the $C^k$-distance on $\mathcal{A}$ is induced by the $C^k$-distance on the space of positive smooth functions on $S^{2n-1}$. For every $A$ in $\mathcal{A}$, we denote by
\[
\alpha_A\coloneqq \lambda_0|_{\partial A}
\]
the restriction of the standard primitive of $\omega_0$ to the boundary of $A$. In the case of the unit ball $A_1=B$, we find the standard contact form
\[
\alpha_0\coloneqq \alpha_{B} = \lambda_0|_{S^{2n-1}}
\]
on $S^{2n-1}$, which is Zoll with all orbits of period $\pi$. 
Given a domain $A=A_f$ in $\mathcal{A}$, the radial projection
\[
\rho_A : S^{2n-1} \rightarrow \partial A, \qquad z\mapsto f(z) z,
\]
satisfies
\begin{equation}
\label{pbr}
\rho_A^* \alpha_{A} = f^2 \alpha_0.
\end{equation}
This identity shows that $A\in \mathcal{A}$ is $C^k$-close to $B$ if and only if $\rho_A^* \alpha_A$ is $C^k$-close to $\alpha_0$.

We now proceed with the proof of statement (ii) in Theorem \ref{1stcap}. Given a smooth positive function $f$ on $S^{2n-1}$, we denote by
\[
\{f_t \coloneqq 1 + t (f-1)\}_{t\in [0,1]}
\]
the smooth homotopy connecting the constant function $1$ to $f$ and we consider the corresponding smooth path of domains $\{A_{f_t}\}_{t\in [0,1]}$ in $\mathcal{A}$, connecting $B$ to $A_f$. We assume that the domain $A\coloneqq A_f$ is sufficiently $C^2$-close to $B$, so that 
\[
\|\rho_{A_{f_t}}^* \alpha_{A_{f_t}} - \alpha_0\|_{C^2} = \|(f^2_t-1)\alpha_0\|_{C^2} < \delta \quad \forall t\in [0,1],
\]
where $\delta$ is the positive number appearing in Theorem \ref{qi}. Note that $\{\rho_{A_{f_t}}^* \alpha_{A_{f_t}}\}_{t\in [0,1]}$ is a smooth path of contact forms on $S^{2n-1}$ connecting $\alpha_0$ to $\rho_A^* \alpha_A$. By Theorem \ref{qi}, there exists a smooth family of diffeomorphism $\{\varphi_t: S^{2n-1} \rightarrow S^{2n-1}\}_{t\in [0,1]}$ such that
\begin{equation}
\label{coniuga}
\varphi^*_t (\rho_{A_{f_t}}^* \alpha_{A_{f_t}}) = T_t e^{h_t}\, \alpha_0,
\end{equation}
where $\{T_t\}_{t\in [0,1]}$ and $\{h_t\}_{t\in [0,1]}$ are smooth families of functions on $S^{2n-1}$ satisfying properties (a)-(e) of Theorem \ref{qi} together with $\varphi_0=\mathrm{id}$, $T_0=1$ and $h_0=0$. By (e) we have
\[
\mathrm{sys}(A) = \mathrm{sys}(\alpha_{A}) = \mathrm{sys}(\rho_A^* \alpha_{A} ) =  \mathrm{sys}(\rho_{A_1}^* \alpha_{A_1} ) =
\mathrm{sys}(\alpha_0)  \min_{S^{2n-1}} T_1 = \pi \min_{S^{2n-1}} T_1.
\]
Therefore, Theorem \ref{1stcap} (ii) will be proven if we can find a symplectomorphism $\phi: \C^n \rightarrow \C^n$ mapping the open ball $B'$ of radius
\[
R\coloneqq \sqrt{ \min_{S^{2n-1}} T_1}
\]
centered at the origin into $A$. By property (c) of Theorem \ref{qi}, we have
\[
R = \sqrt{ \min_{S^{2n-1}} T_1e^{h_1} } = \min_{S^{n-1}} \sqrt{T_1} e^{\frac{h_1}{2}}
\]
and hence the domain
\[
A'\coloneqq A_{\sqrt{T_1} e^{\frac{h_1}{2}}}
\]
contains $B'$. Hence, it is enough to find a symplectomorphism $\phi\colon\C^n\to\C^n$ such that $\phi(A')= A$. Consider the smooth path of domains
\[
\{A'_t\coloneqq A_{\sqrt{T_t} e^{\frac{h_t}{2}}}\}_{t\in [0,1]},
\]
which satisfies $A'_0=B$ and $A'_1=A'$. By \eqref{coniuga} and \eqref{pbr}, the diffeomorphisms
\[
\psi_t: \partial A'_t \rightarrow \partial A_{f_t}, \qquad \psi_t\coloneqq \rho_{A_{f_t}} \circ \varphi_t \circ  \rho_{A'_t}^{-1},
\]
satisfy 
\[
\psi^*_t ( \lambda_0|_{\partial A_{f_t}} ) = ( \rho_{A'_t}^{-1} )^{*} ( \varphi^*_t  ( \rho_{A_{f_t}}^* \alpha_{A_{f_t}}  )) = ( \rho_{A'_t}^{-1} )^{*} ( T_t e^{h_t} \alpha_0 ) = \alpha_{A'_t}= \lambda_0|_{\partial A'_t}.
\]
The fact that the diffeomorphism $\psi_t$ preserves the restriction of $\lambda_0$ implies that its positively 1-homogeneous extension
\[
\tilde{\psi}_t: \C^n\setminus \{0\} \rightarrow \C^n\setminus \{0\}, \qquad \tilde{\psi}_t( r z ) \coloneqq r \psi_t(z) \quad \forall z\in \partial A'_t, \; \forall r>0,
\]
preserves $\lambda_0$. By construction, $\tilde{\psi}_t$ depends smoothly on $t$ and satisfies 
\[
\tilde{\psi}_0=\mathrm{id}, \qquad \tilde{\psi}_1(A'\setminus\{0\}) = A\setminus \{0\}. 
\]
Let $X_t$ be the vector field generating $\tilde{\psi}_t$. Since $\tilde{\psi}_t$ preserves $\lambda_0$, we have
\[
0 = \mathcal{L}_{X_t} \lambda_0 = \imath_{X_t} \omega_0 + d \imath_{X_t} \lambda_0,
\]
and hence $X_t$ is the Hamiltonian vector field of the Hamiltonian $H_t:= - \lambda_0(X_t)$ on $\C^n\setminus \{0\}$. Multiplying $H_t$ by a smooth function on $\C^n$ which is supported in $\C^n\setminus \{0\}$ and equals $1$ outside of a sufficiently small neighborhood of $0$, we obtain a new time-dependent Hamiltonian on $\C^n$ whose time-one map is a global symplectomorphism  $\phi: \C^n \rightarrow \C^n$ which coincides with $\tilde{\psi}_1$ on $\C^n\setminus A'$, and hence satisfies $\phi(A')=\tilde{\psi}_1(A') = A$. This concludes the proof of Theorem \ref{1stcap} (ii).

\section{Proof of Theorem \ref{nthcap}}
\label{pfnthcap}

Let $A\in \mathcal{A}$ be $C^2$-close to $B$. Using the notation of Section \ref{proof-of-thm-1}, we deduce that the contact form $\rho_A^* \alpha_A$ on $S^{2n-1}$ is $C^2$-close to the Zoll contact form $\alpha_0$ and hence there exists a diffeomorphism $\varphi:S^{2n-1} \rightarrow S^{2n-1}$ such that
\begin{equation}
\label{varphi}
\varphi^* (\rho^*_A \alpha_A) = T e^h \,\alpha_0,
\end{equation}
where the smooth functions $T: S^{2n-1} \rightarrow \R$ and $h: S^{2n-1} \rightarrow \R$ satisfy all the requirements of Theorem \ref{qi}. In particular, $\partial A$ has a closed characteristic $\gamma$ whose action equals the number
\[
T_{\max}(A)\coloneqq\pi \max_{S^{2n-1}} T.
\] 
Up to rescaling, we can assume that
\begin{equation}
\label{normass}
T_{\max}(A) = T_{\max}(B) = \pi,
\end{equation}
and we shall prove the following facts:
\begin{enumerate}[(i')]
\item There exists a symplectomorphism $\phi: \C^n \rightarrow \C^n$ mapping the ellipsoid $E\coloneqq E(\pi,\frac{\pi}{2}, \dots,\frac{\pi}{2})$ into $A$.
\item There exists a symplectomorphism $\psi: \C^n \rightarrow \C^n$ mapping $A$ into $B$.
\end{enumerate}
From the monotonicity property of the capacity $c$ we then have
\[
c(E)  \leq c(A)  \leq c(B).
\]
Since $c$ is $n$-normalized, its values at $E$ and $B$ coincide with the corresponding values of the $n$-th Ekeland-Hofer capacity, which assigns to both $E$ and $B$ the value $\pi$. Therefore, $c(E)=c(B)$ and both inequalities in the above expression are equalities. This proves claims (i) and (ii) from Theorem \ref{nthcap}. There remains to prove (i') and (ii').

\bigskip

The proof of (i') is similar to the proof of Theorem \ref{1stcap} (i) from Section \ref{pf1stcapi}, and here we use the notation which we introduced there. Up to the application of a unitary isomorphism, we may assume that the closed characteristic $\gamma$, which by \eqref{normass} has action $\pi$, passes through the point $(1,0,\dots,0)\in \C^n$ and is contained in the open set $V=\Phi(U)$, see \eqref{Phi} and \eqref{U}. By the explicit form \eqref{Phi} of the symplectomorphism $\Phi$, we have the identities
\begin{equation}
\label{formaE}
B\cap V =  \Phi(U\cap \{s<0\}), \qquad E\cap V = \Phi\bigl(U\cap \{s< - \pi |w|^2 \} \bigr).
\end{equation}
Since $A$ is $C^2$-close to $B$, we can find a $C^2$-small smooth function $H:\T \times \C^{n-1} \rightarrow \R$ with support in $\T\times \widehat{B}$ satisfying \eqref{buonadef} and such that
\begin{equation}
\label{AcapV2}
A \cap V = D(H)\cap V.
\end{equation}
The closed characteristic $\gamma$ of action $\pi$ corresponds to the fixed point $0$ of $\phi^1_H$, which has action $0$, see Proposition \ref{lift} (ii). By Proposition \ref{modham}, there exists a smooth family of Hamiltonians $\{H^{\lambda}\}_{\lambda\in [0,1]}$ on $\T \times \C^{n-1}$ which satisfies:
 \begin{enumerate}[(i)]
\item $H^0=H$;
\item $H^{\lambda}=H$ on $\T \times \widehat{B}^{\mathrm{c}}_{\frac{1}{3}}$ for every $\lambda\in [0,1]$;
\item $\phi_{H^{\lambda}}^1 = \phi_H^1$ for every $\lambda\in [0,1]$;
\item $|H^{\lambda}|< \frac{\pi}{2}$ for every $\lambda\in [0,1]$;
\item $H^1(t,w) \geq - \pi |w|^2$ for every $(t,w)\in \T\times \C^{n-1}$.
\end{enumerate}
By (i) and (iii), Proposition \ref{lift} (iii) gives us a symplectomorphism $\widetilde{\phi}: \C^n \rightarrow \C^n$ such that $\tilde{\phi}(D(H))=D(H^1)$, which by (ii), (iv) and Remark \ref{support} is supported in $V$ and by \eqref{AcapV2} satisfies
\[
\tilde{\phi}(A) = (A\setminus V) \cup (D(H^1)\cap V).
\]
Since the closure of $E\setminus V$ is contained in $B$, $C^0$-closeness of $A$ to $B$ implies that $E\setminus V$ is contained in $A\setminus V$. By (v) and the second identity in \eqref{formaE}, $E\cap V$ is contained in $D(H^1)\cap V$, so the above identity implies that $E$ is contained in $\tilde{\phi}(A)$. We conclude that the symplectomorphism $\phi\coloneqq \tilde{\phi}^{-1}$ satisfies $\phi(E)\subset A$, proving (i').

\bigskip

The proof of (ii') is similar to the proof of Theorem \ref{1stcap} from Section \ref{proof-of-thm-1}. Arguing as in that proof, we can apply Theorem \ref{qi} to a 1-parameter family of contact forms joining $\alpha_0$ and $\rho^*_A \alpha_A$ and obtain a symplectomorphism $\psi: \C^n \rightarrow \C^n$ mapping $A$ onto the domain
\[
A' \coloneqq A_{\sqrt{T} e^{\frac{h}{2}}},
\]
see \eqref{varphi}.
By \eqref{normass} and the fact that the maximum of $T e^h$ coincides with the maximum of $T$ (see statement (c) in Theorem \ref{qi}), we get that 
\[
\max_{S^{2n-1}} (\sqrt{T}e^{\frac{h}{2}})=1.
\]
Hence, $A'\subset A_1=B$, proving (ii').

\section{Proof of Proposition \ref{polydiscs}}\label{secpoly}

We now prove Proposition \ref{polydiscs}. Up to exchanging the coordinates and rescaling, we may assume that $a=1\leq b$, so that
\[
\widehat{c}_k(P(1,b)) = k\min\{1,b\}=k,\qquad \forall k\in\N,
\]
see \cite[Proposition 5]{eh90} and \cite[Example 1.7]{gh18}. 

Next we recall that the ellipsoid $E(1,2)$ symplectically embeds into $P(1,1)$ (see \cite[Theorem 1.3]{fm15} and \cite[Lemma 8.2]{sch18}). The existence of this symplectic embedding is quite remarkable, because $E(1,2)$ and $P(1,1)$ have the same volume and are not symplectomorphic, as $\widehat{c}_3(E(1,2))=2$ and $\widehat{c}_3(P(1,1))=3$.

By the ``extension after retraction principle'' from \cite[Proposition 1.3]{sch02}, for every $\theta<1$ there is a global symplectomorphism of $\C^2$ mapping $E(\theta,2\theta)$ into $P(1,1)$, and hence also into $P(1,b)$. Therefore, 
\[
\underline{c}_k(P(1,b)) \geq \widehat{c}_k(E(\theta,2\theta))  \qquad \forall k\in \N,
\]
and by taking the supremum over all $\theta<1$ we obtain
\begin{equation}
\label{lb}
\underline{c}_k(P(1,b)) \geq \widehat{c}_k(E(1,2))  \qquad \forall k\in \N.
\end{equation}
For every $\alpha>1$ we choose $\beta$ large enough, so that $P(1,b) \subset E(\alpha,\beta)$. Then for every $k\in \N$ we have
\[
\overline{c}_k(P(1,b)) \leq \widehat{c}_k(E(\alpha,\beta)) \leq k \alpha,
\]
and by taking the infimum over all $\alpha>1$ we obtain
\begin{equation}
\label{ub}
\overline{c}_k(P(1,b)) \leq k \qquad \forall k\in \N.
\end{equation}
On the other hand, $k=\widehat c_k(P(1,b))\leq \overline{c}_k(P(1,b))$ and we conclude that 
\[
\overline{c}_k(P(1,b))=k,\qquad\forall k\in\N.
\]
Since $\widehat{c}_2(E(1,2))=2$ and every 2-normalized capacity $c$ satisfies $\underline{c}_2\leq c \leq \overline{c}_2$, the inequalities \eqref{lb} and \eqref{ub} for $k=2$ imply that 
\[
c(P(1,b)) = 2 = 2 \min\{1,b\},
\]
proving statement (i) of Proposition \ref{polydiscs}.

In order to prove statement (ii), we use Hutchings' ECH-capacities $c_k^{\rm ECH}$ for four-dimensional domains from \cite{hut11}. By \cite[Proposition 1.2]{hut11}, the $k$-th ECH-capacity of the ellipsoid $E(a,b)$ is the $(k+1)$-th smallest element in the list $(ha+jb)_{h,j\geq 0}$, again allowing repetitions. In particular, $c_k^{\rm ECH}$ is not a $k$-normalized capacity in the sense of this paper.  

Let $E(\alpha,\beta)$ be an ellipsoid which symplectically embeds into $P(1,1)$. As observed in \cite[Remark 1.8]{hut11}, the ECH-capacities $P(1,1)$ and $E(1,2)$ coincide, and hence
\[
c_k^{\rm ECH}(E(\alpha,\beta)) \leq c_k^{\rm ECH}(P(1,1)) = c_k^{\rm ECH}(E(1,2)) \qquad \forall k\in \N.
\]
As proven by McDuff in \cite{mcd11}, the ECH-capacities form a complete list of obstructions for the problem of symplectically embedding an ellipsoid into another, and hence the above inequalities imply that $E(\alpha,\beta)$ symplectically embeds into $E(1,2)$. Therefore,
\[
\widehat{c}_k(E(\alpha,\beta)) \leq  \widehat{c}_k(E(1,2)) \qquad \forall k\in \N,
\]
and by taking the supremum over the space of all ellipsoids $E(\alpha,\beta)$ which embed into $P(1,1)$ by a globally defined symplectomorphism of $\C^n$ we obtain the inequalities
\[
\underline{c}_k(P(1,1)) \leq  \widehat{c}_k(E(1,2)) \qquad \forall k\in \N.
\]
Together with \eqref{lb}, we deduce the identities
\[
\underline{c}_k(P(1,1)) = \widehat{c}_k(E(1,2)) \qquad \forall k\in \N.
\]
Moreover, $\widehat{c}_k(E(1,2))\leq k$ with equality if and only if $k$ equals 1 or 2. This proves statement
(ii) of Proposition \ref{polydiscs}.

\section{Proof of Theorem \ref{achieved}} \label{secachieved}

The proof of statement (i) in Theorem \ref{achieved} makes use of the following well known alternative (see \cite[Theorem II.26]{sul76} for a proof based on the Hahn--Banach theorem and the appendix of \cite{ls94} for a more elementary proof). 

\begin{thm}
\label{alternative}
Let $X$ be a smooth vector field on the closed manifold $M$ and let $K$ be a compact subset of $M$. Then there exists either a probability measure which is supported in $K$ and is invariant under the flow of $X$ or a smooth real function $h$ on $M$ such that $dh[X]>0$ on $K$.
\end{thm}

Let $\xi$ be a co-oriented contact structure on the closed manifold $M$. A smooth vector field on $M$ is said to be a contact vector field if its flow consists of contactomorphisms of $(M,\xi)$. Contact vector fields are in one-to-one correspondence with smooth real functions on $M$: A smooth real function $H$ on $M$ (the contact Hamiltonian) and a defining contact form $\alpha$ for $\xi$ define the contact vector field $X$ by the identities
\begin{equation}
\label{defX}
\imath_X d\alpha = \bigl( \imath_{R_{\alpha}} dH \bigr) \alpha - dH, \qquad \imath_X \alpha = H,
\end{equation}
see \cite[Section 2.3]{gei08}. We shall make use of the following formula for the conformal factor of the flow of contactomorphisms in terms of the generating contact Hamiltonian.

\begin{lem}
\label{formula}
Let $X$ be the contact vector field induced by the contact Hamiltonian function $H\in C^{\infty}(M)$ together with the contact form $\alpha\in \mathcal{F}(M,\xi)$, and denote by $\phi^t_X$ its flow. Then
\[
(\phi^t_X)^* \alpha = \exp \Bigl( \int_0^t (\imath_{R_{\alpha}} dH) \circ \phi_X^s \, ds\Bigr) \alpha ,
\]
for every $t\in \R$.
\end{lem}

\begin{proof}
We define $h_t$ by the identity $(\phi^t_X)^* \alpha = e^{h_t} \alpha$. Then $h_0=0$ and differentiation in $t$ gives us
\[
\begin{split}
\partial_t h_t \, e^{h_t} \alpha &= \frac{d}{dt} (\phi^t_X)^* \alpha = (\phi^t_X)^* \mathcal{L}_X \alpha = (\phi^t_X)^* \bigl( \imath_X d\alpha + d \imath_X \alpha \bigr) \\ &=  (\phi^t_X)^* \bigl( ( \imath_{R_{\alpha}} dH) \alpha \bigr) = (\imath_{R_{\alpha}} dH) \circ \phi_X^t \, e^{h_t} \alpha,
\end{split}
\]
where we have used \eqref{defX}. Then
\[
h_t = \int_0^t \partial_s h_s \, ds = \int_0^t (\imath_{R_{\alpha}} dH) \circ \phi_X^s \, ds,
\]
as claimed.
\end{proof}

We can now prove Theorem \ref{achieved}.

\begin{proof}[Proof of Theorem \ref{achieved}.] We are assuming that the contact forms $\alpha,\beta\in \mathcal{F}(M,\xi)$ satisfy
\[
d(\alpha,\beta) = \max_M f - \min_M f \quad \mbox{with} \quad \varphi^* \beta = e^f \alpha,
\]
for some $\varphi\in \mathrm{Cont}_0(M,\xi)$.

\medskip

\noindent (i) If $f$ is constant, then $R_{\varphi^* \beta}$ is a constant multiple of $R_{\alpha}$ and hence the probability measure on $M$ which is determined by the volume form $\alpha\wedge d\alpha^{n-1}$ is invariant for the flows of both these Reeb vector fields. Therefore, we can assume that the compact sets
\[
K_{\min} \coloneqq f^{-1}(\min_M f) \qquad \mbox{and} \qquad  K_{\max} \coloneqq f^{-1}(\max_M f)
\]
are disjoint. Since $df$ vanishes on $K_{\min}$ and $K_{\max}$, the identity
\[
d(\varphi^* \beta) = d(e^f \alpha) = e^f (df\wedge \alpha + d\alpha)
\]
implies that
\[
\begin{split}
R_{\varphi^* \beta} &= e^{-\min f} R_{\alpha} \quad \mbox{on} \quad K_{\min}, \\
R_{\varphi^* \beta} &= e^{-\max f} R_{\alpha} \quad \mbox{on} \quad K_{\max}.
\end{split}
\]
Therefore, any invariant measure for the Reeb flow of $\varphi^* \beta$ which is supported in either $K_{\min}$ or $K_{\max}$ is invariant also for the Reeb flow of $\alpha$. We shall show the existence of a $R_{\varphi^* \beta}$-invariant probability measure supported in $K_{\min}$, the case of $K_{\max}$ being analogous. 

We argue by contradiction and assume that the flow of $R_{\varphi^* \beta}$ admits no invariant probability measure which is supported in $K_{\min}$. Then Theorem \ref{alternative} implies the existence of a smooth real function $H$ on $M$ such that
\[
dH\bigl[R_{\varphi^* \beta}\bigr] >0 \quad \mbox{on} \quad K_{\min}.
\]
By multiplication by a smooth function having the value 1 near $K_{\min}$ and $0$ near $K_{\max}$, we can assume that $H=0$ on a neighborhood $V$ of $K_{\max}$. Let $X$ be the contact vector field which is induced by the contact Hamiltonian $H$ and the contact form $\varphi^* \beta = e^f \alpha$. By Lemma \ref{formula}, we have
\[
(\phi_X^t)^* ( \varphi^* \beta) = e^{g_t} \alpha,
\]
where
\[
g_t\coloneqq f + \int_0^t \bigl( \imath_{R_{\varphi^* \beta}} dH\bigr) \circ \phi_X^s\, ds  = f + t \, dH\bigl[R_{\varphi^* \beta}\bigr] + O(t^2)
\]
for $t\rightarrow 0$. We can find an open neighborhood $U$ of $K_{\min}$ and a positive number $\delta$ such that
\begin{align*}
&dH\bigl[R_{\varphi^* \beta}\bigr]  \geq \delta  &\mbox{on}&\quad U, \\
&f \geq \min f + \delta  &\mbox{on}& \quad M\setminus U, \\
&f \leq \max f - \delta  &\mbox{on}& \quad M\setminus V.
\end{align*}
On $U$, we then have
\[
g_t \geq \min f + \delta t + O(t^2),
\]
while on $M\setminus U$ we have
\[
g_t \geq \min f + \delta - c t + O(t^2),
\]
where $c$ is the supremum norm of the function $dH\bigl[R_{\varphi^* \beta}\bigr]$. We conclude that
\begin{equation}
\label{st1}
g_t \geq \min f + {\textstyle \frac{\delta}{2}} t \quad \mbox{on} \quad M,
\end{equation}
for $t\geq 0$ small enough. Since $H$ vanishes on $V$, so does $X$ and hence $g_t=f$ on $V$ for every $t\in \R$. On $M\setminus V$ we have
\[
g_t \leq \max f - \delta + ct + O(t^2).
\]
Therefore,
\[
g_t \leq \max f  \quad \mbox{on} \quad M,
\]
for $t\geq 0$ small enough. Together with \eqref{st1}, this implies that for $t>0$ sufficiently small the contactomorphism 
\[
\psi\coloneqq \varphi\circ \phi_X^t\in \mathrm{Cont}_0(M,\xi)
\]
satisfies
\[
\psi^* \beta = e^{g_t} \alpha
\]
with
\[
\max g_t - \min g_t < \max f - \min f,
\]
contradicting the assumption that $d(\alpha,\beta) = \max f - \min f$.

\medskip

\noindent (ii) From the definition of the pseudo-metric $d$ we obtain
\begin{equation}
\label{stosc}
d(e^{g_1} \alpha, e^{g_2} \alpha ) =d(e^{g_1} \alpha, e^{g_2-g_1} e^{g_1} \alpha) \leq \max_M (g_2-g_1) - \min_M (g_2-g_1),
\end{equation}
for every $g_1$ and $g_2$ in $C^{\infty}(M)$.
Given the path $\{\psi_t\}_{t\in [0,1]}$ in $\mathrm{Cont}_0(M,\xi)$ such that $\psi_0=\mathrm{id}$ and $\psi_1= \varphi^{-1}$, we consider the curve
\[
\gamma(t) \coloneqq \psi_t^*(e^{tf} \alpha), \qquad t\in [0,1],
\]
which satisfies $\gamma(0)=\alpha$ and $\gamma(1)=\beta$. For any subdivision $0=t_0<t_1< \dots<t_k=1$ we have, by the invariance of the psedo-metric $d$ under the action of $\mathrm{Cont}_0(M,\xi)$ and by \eqref{stosc},
\[
\begin{split}
d(\alpha,\beta) &\leq \sum_{j=1}^k d(\gamma(t_{j-1}),\gamma(t_j)) =  \sum_{j=1}^k d(\psi_{t_{j-1}}^*(e^{t_{j-1}f} \alpha),\psi_{t_j}^*(e^{t_j f} \alpha)) =  \sum_{j=1}^k d(e^{t_{j-1}f} \alpha,e^{t_j f} \alpha) \\ & \leq \sum_{j=1}^k \max_M (t_j f - t_{j-1} f) -  \min_M (t_j f - t_{j-1} f) = \max_M f - \min_M f = d(\alpha,\beta).
\end{split}
\]
The above inequalities are then equalities and hence $\gamma$ has length $d(\alpha,\beta)$ and is therefore a minimizing geodesic.
\end{proof}

\section{Elementary spectral invariants}
\label{spesec}

The proof of Theorem \ref{short-geodesics} in Section \ref{secthm5} is based on a sequence of spectral invariants for contact forms on a closed contact manifold, whose definition and properties are discussed in the present section. These spectral invariants are inspired by the symplectic capacities defined by McDuff--Siegel in \cite{ms23}, by those introduced by Hutchings in \cite{hut22a,hut22b} and by the spectral invariants for Hamiltonian diffeomorphisms defined by one of us in \cite{edt22b}. They are more elementary than, say, the spectral invariants which are induced by contact homology, in the sense that their definition and the proof of their properties require only compactness results for pseudoholomorhpic curves. Their computation involves some knowledge of Gromov--Witten invariants, see Proposition \ref{zoll_computation} below.

Before discussing them, we recall some basic facts and fix notation concerning the symplectization of a contact manifold.
Let $\xi$ be a co-oriented contact structure on the $(2n-1)$-dimensional manifold $M$. A cotangent vector $x \in T^*_p M$ is said to define $\xi(p)$ if $\ker x =\xi(p)$ and $x$ is positive on vectors in $T_p M$ which are positively transverse to $\xi(p)$. We denote by 
\[
\widetilde{M} \coloneqq \{ x \in T^*M \mid x \mbox{ defines } \xi(\pi(x)) \}
\]
the symplectization of $M$, where $\pi: T^*M \rightarrow M$ denotes the footpoint projection. The symplectization $\widetilde{M}$ is a $2n$-dimensional submanifold of $T^*M$. It carries a Liouville $1$-form $\widetilde{\lambda}$ and a symplectic form $\widetilde{\omega} = d\widetilde{\lambda}$ given by the restriction of the standard Liouville $1$-form and symplectic form of $T^*M$. Moreover, $\widetilde{M}$ carries a canonical $\R$-action given by
\[
\R \times \widetilde{M} \rightarrow \widetilde{M}, \qquad (a,\eta) \mapsto e^a \eta.
\]
We denote by $\partial_a$ the vector field on $\widetilde{M}$ generating this action. The choice of a contact form $\alpha\in \mathcal{F}(M,\xi)$ induces the diffeomorphism
\[
\varphi_{\alpha}: \R \times M \rightarrow \widetilde{M}, \qquad (a,p) \mapsto e^a \alpha(p),
\]
such that $\varphi_{\alpha}^* \widetilde\omega=d(e^a \alpha)$ where $a$ is the variable in the first factor of $\R\times M$. Under the identification of $\widetilde{M}$ with $\R\times M$ defined by $\varphi_{\alpha}$, the canonical action of $\R$ is given by translations in the first factor (this justifies the choice of denoting the generating vector field by $\partial_a$) and the contact-type hypersurface 
\[
M_{\alpha} \coloneqq \{ \alpha(p) \mid p\in M \} \subset \widetilde{M}
\]
corresponds to the set $\{a=0\}$. The contact form $\alpha$ and the Reeb vector field $R_{\alpha}$ can be seen as a translation invariant 1-form and vector field on $\R\times M$, and by the diffeomorphism $\varphi_{\alpha}$ induce an $\R$-invariant 1-form and an $\R$-invariant vector field on $\widetilde{M}$:
\[
\widetilde{\alpha} := {\varphi_{\alpha}}_* \alpha, \qquad \widetilde{R}_{\alpha} := {\varphi_{\alpha}}_* R_{\alpha}.
\]

The vector field $\widetilde{R}_{\alpha}$ is tangent to the hypersurface $M_{\alpha}$ and all its images by the canonical $\R$-action, and spans the characteristic distribution of these hypersurfaces. 

Given contact forms $\alpha, \beta\in \mathcal{F}(M,\xi)$ with $\alpha<\beta$, we denote by $\widetilde{M}_{\alpha}^{\beta}\subset \widetilde{M}$ the symplectic cobordism which is given by the region in $\widetilde{M}$ bounded by $M_{\alpha}$ and $M_{\beta}$, that is,
\[
\widetilde{M}_{\alpha}^{\beta} \coloneqq \{ x \in \widetilde{M} \mid \alpha(\pi(x)) < x < \beta(\pi(x)) \mbox{\scalebox{.9}{ on vectors which are positively transverse to }} \xi(\pi(x))\}.
\]
The diffeomorphism $\varphi_\beta$ exhibits the region $\widetilde{M}_\beta^+\subset \widetilde{M}$ above the hypersurface $M_\beta$ as a positive cylindrical end attached to the symplectic cobordism $\widetilde{M}_\alpha^\beta$. Similarly, we can regard the region $\widetilde{M}_\alpha^-\subset \widetilde{M}$ below the hypersurface $M_\alpha$ as a negative cylindrical end of $\widetilde{M}_\alpha^\beta$ via the diffeomorphism $\varphi_\alpha$.

Let us recall the following terminology from \cite{behwz03}. An almost complex structure $J$ on $\widetilde{M}$ is called {\it adjusted} to a contact form $\alpha\in \mathcal{F}(M,\xi)$ if it satisfies the following properties:
\begin{enumerate}
\item $J$ is invariant under the canonical $\R$-action on $\widetilde{M}$.
\item $J\partial_a = \widetilde{R}_\alpha$.
\item The codimension-two distribution $\xi_{\alpha}:= {\varphi_{\alpha}}_*((0)\times \xi)$ is invariant under $J$. Moreover, the restriction of $J$ to $\xi_{\alpha}$ is compatible with $d\widetilde\alpha$.
\end{enumerate}

Suppose that $J$ is an $\alpha$-adjusted almost complex structure on $\widetilde{M}$. Let $u:(\Sigma,j)\rightarrow (\widetilde{M},J)$ be a $J$-holomorphic map whose domain is a punctured Riemann surface. We define the {\it action} $\mathcal{A}_\alpha(u)$ to be
\begin{equation*}
\mathcal{A}_\alpha(u) \coloneqq \int_{\Sigma} u^*d\widetilde\alpha \in [0,+\infty].
\end{equation*}
Note that since $u$ is $J$-holomorphic and $J$ is adjusted to $\alpha$, the $2$-form $u^*d\widetilde\alpha$ is everywhere non-negative on $\Sigma$. Thus $\mathcal{A}_\alpha(u)$ is indeed well-defined an non-negative. Moreover, we define
\begin{equation*}
E_\alpha(u) \coloneqq \sup_{\rho} \int_{\Sigma} u^* \bigl( \rho\circ a_{\alpha}\, da_{\alpha}\wedge\widetilde\alpha \bigr) \in [0,+\infty], \qquad \mbox{with } a_{\alpha} := a \circ \varphi_{\alpha}^{-1},
\end{equation*}
where the supremum is taken over all compactly supported functions $\rho:\R\rightarrow \R_{\geq 0}$ whose integral is equal to $1$. As before, the $2$-form $u^* (\rho\circ a_{\alpha}) da_{\alpha}\wedge\widetilde\alpha$ is everywhere non-negative because $J$ is $\alpha$-adjusted and $u$ is $J$-holomorphic. We say that $u$ is a {\it finite-energy} curve if both $\mathcal{A}_\alpha(u)$ and $E_{\alpha}(u)$ are finite. If $u$ is a finite-energy curve and $\alpha$ is non-degenerate, then $u$ is positively or negatively asymptotic to trivial cylinders over periodic Reeb orbits of $\alpha$ at every puncture of $\Sigma$ which is not a removable singularity of $u$. We define the {\it positive action} $\mathcal{A}_\alpha^+(u)$ to be the sum of the actions (with respect to $\alpha$) of the positive asymptotic Reeb orbits of $u$. Similarly, we define the {\it negative action} $\mathcal{A}_\alpha^-(u)$ to be the sum of the actions of the negative asymptotic Reeb orbits of $u$. By Stokes' theorem, we have $\mathcal{A}_\alpha(u) = \mathcal{A}_\alpha^+(u) - \mathcal{A}_\alpha^-(u)$.

Let $\alpha<\beta$ be contact forms and assume that $J$ is a compatible almost complex structure on $\widetilde{M}$ whose restriction to $\widetilde{M}_\beta^+$ is $\beta$-adjusted and whose restriction to $\widetilde{M}_\alpha^-$ is $\alpha$-adjusted. We define the {\it action}
\begin{equation*}
\mathcal{A}_\alpha^\beta(u)\coloneqq \int_{u^{-1}(\widetilde{M}_\beta^+)} u^*d\widetilde\beta
+ \int_{u^{-1}(\widetilde{M}_\alpha^\beta)} u^*\widetilde{\omega}
+ \int_{u^{-1}(\widetilde{M}_{\alpha}^-)} u^*d\widetilde\alpha \in [0,+\infty].
\end{equation*}
All three summands in this expression are well-defined and non-negative. For the terms involving the cylindrical ends $\widetilde{M}_\beta^+$ and $\widetilde{M}_\alpha^-$ this follows from the fact that $J$ is adjusted and for the cobordism $\widetilde{M}_\alpha^\beta$ this uses that $J$ is compatible. In addition, we define
\begin{equation*}
E_\alpha^\beta(u) \coloneqq \sup_{\rho^+} \int_{u^{-1}(\widetilde{M}_\beta^+)} u^* \bigl( \rho^+\circ a_{\beta}\, da_{\beta} \wedge \widetilde\beta \bigr)
+ \sup_{\rho^-} \int_{u^{-1}(\widetilde{M}_\alpha^-)} u^* \bigl( \rho^-\circ a_{\alpha}\,  da_{\alpha}\wedge \widetilde\alpha \bigr) \in [0,+\infty].
\end{equation*}
Here the suprema are taken over all compactly supported, non-negative functions $\rho^\pm$ on $\R_{\geq 0}$ and $\R_{\leq 0}$, respectively, with integral equal to $1$. Again, a curve $u$ is said to have finite energy if both $\mathcal{A}_\alpha^\beta(u)$ and $E_\alpha^\beta(u)$ are finite. If $u$ has finite energy and $\alpha$ and $\beta$ are non-degenerate, then $u$ is positively asymptotic to periodic Reeb orbits of $\beta$ or negatively asymptotic to periodic Reeb orbits of $\alpha$ at punctures which are not removable singularites. The positive action $\mathcal{A}_\beta^+(u)$ and the negative action $\mathcal{A}_\alpha^-(u)$ are defined as before and we have $\mathcal{A}_\alpha^\beta(u) = \mathcal{A}_\beta^+(u) - \mathcal{A}_\alpha^-(u)$ by Stokes' theorem.

Given a contact form $\alpha\in \mathcal{F}(M,\xi)$, define $\mathcal{J}(\alpha)$ to be the set all compatible almost complex structures $J$ on the symplectization $\widetilde{M}$ with the property that there exist real numbers $s^-<s^+$ such that $J$ is adjusted to $\alpha$ on both $\widetilde{M}_{e^{s^\pm}\alpha}^\pm$. Here the numbers $s^\pm$ are allowed to depend on $J$. For a given $J\in \mathcal{J}(\alpha)$, there are many choices of $s^\pm$ making it adjusted to $\alpha$ on $\widetilde{M}_{e^{s^\pm}\alpha}^\pm$. While the action $\mathcal A^{e^{s^+}\alpha}_{e^{s^-}\alpha}$ and the energy $E^{e^{s^+}\alpha}_{e^{s^-}\alpha}$ themselves depend on the choice of $s^\pm$, the property of being a finite-energy $J$-holomorphic curve does not.

Suppose now that $\alpha\in\mathcal{F}(M,\xi)$ is non-degenerate, that is, all the periodic orbits of the Reeb vector field $R_\alpha$ are non-degenerate. Given an almost complex structure $J\in \mathcal{J}(\alpha)$ and $k$ points $x_1,\dots,x_k\in \widetilde{M}$, we define $\mathcal{M}^J(x_1,\dots,x_k)$ to be the set of all pseudoholomorphic curves $u:(\Sigma,j)\rightarrow (\widetilde{M},J)$ with the following properties:
\begin{enumerate}
\item The domain $(\Sigma,j)$ is a finite union of Riemann spheres with finitely many punctures.
\item The map $u$ is non-constant on every connected component of $\Sigma$.
\item The points $x_1,\dots,x_k$ are contained in the image of $u$.
\item $u$ has finite energy.
\item $u$ has at least one negative asymptotic orbit.
\end{enumerate}
By Stokes' theorem and exactness of $\widetilde{M}$, any curve $u\in \mathcal{M}^J(x_1,\dots,x_k)$ has at least one positive asymptotic orbit. We recall that the positive action $\mathcal{A}^+_{\alpha}(u)$ is the sum of the $\alpha$-actions of these orbits. Given a non-degenerate $\alpha\in \mathcal{F}(M,\xi)$, we define
\begin{equation}\label{eq:ckalpha}
c_k(\alpha) \coloneqq \sup_{\substack{J\in \mathcal{J}(\alpha) \\ x_1,\dots,x_k\in\widetilde{M}}} \inf_{u\in \mathcal{M}^J(x_1,\dots,x_k)} \mathcal{A}^+_{\alpha}(u)\in (0,+\infty].
\end{equation}

The proof of the basic properties of the functions $c_k$ which are listed in the next theorem is similar to the analogous proofs in \cite{ms23, hut22a, hut22b, edt22b}.

\begin{thm}
\label{theorem:spectral_invariants}
The functions given by \eqref{eq:ckalpha} have unique extensions $c_k\colon\mathcal{F}(M,\xi) \rightarrow [0,+\infty]$, $k\geq 0$, satisfying the following properties:
\begin{enumerate}
\item (Scaling) $c_k(r\cdot \alpha)= r\cdot c_k(\alpha)$ for all positive real numbers $r$.
\item (Increasing) $c_k(\alpha)\leq c_\ell(\alpha)$ for $k\leq \ell$.
\item (Sublinearity) $c_{k+\ell}(\alpha) \leq c_k(\alpha) + c_\ell(\alpha)$.
\item (Invariance) $c_k(\varphi^*\alpha) = c_k(\alpha)$ for all $\varphi\in \operatorname{Cont}(M,\xi)$.
\item (Monotonicity): $c_k(\alpha)\leq c_k(\beta)$ whenever $\alpha\leq \beta$.
\item (Lipschitz): If $c_k$ is finite, then $c_k$ is Lipschitz for the $C^0$-norm on $\mathcal{F}(M,\xi)$.
\item (Spectrality): If $c_k(\alpha)$ is finite, then it is a finite sum of actions of periodic Reeb orbits of $\alpha$.
\item (Packing): Suppose that $\alpha< \beta$ and assume that the ball $B^{2n}(A)$ of symplectic width $A$ admits a symplectic embedding into $\widetilde{M}^{\beta}_{\alpha}$. Then
\begin{equation*}
c_{k+1}(\beta) \geq c_k(\alpha) + A.
\end{equation*}
\end{enumerate}
\end{thm}

\begin{rem}
{\rm It follows from the monotonicity and scaling properties that, for fixed $(M,\xi)$ and $k$, the spectral invariant $c_k$ is either finite on all of $\mathcal{F}(M,\xi)$ or identically equal to $+\infty$. If $c_k$ is finite, then the Weinstein conjecture holds for $(M,\xi)$.}
\end{rem}

\begin{proof}
It is enough to check the above properties for non-degenerate contact forms. Then, the Lipschitz property allows us to extend $c_k$ continuously to all, possibly degenerate, contact forms and, by continuity, all properties of $c_k$ continue to hold in the degenerate case.

Let $f_r$ be the diffeomorphism of $\mathbb{R}\times M$ given by $f_r(a,p)= (ra,p)$. The pull-back by $f_r$ yields a bijection
\begin{equation*}
f_r^* : \mathcal{J}(r\alpha) \rightarrow \mathcal{J}(\alpha)
\end{equation*}
and, for every $J\in \mathcal{J}(r\alpha)$ and every $x_1,\dots,x_k\in \widetilde{M}$, a bijection
\begin{equation*}
f_r^* : \mathcal{M}^J(x_1,\dots,x_k) \rightarrow \mathcal{M}^{f_r^*J}(f_r^{-1}(x_1),\dots,f_r^{-1}(x_k)).
\end{equation*}
For every $u\in \mathcal{M}^J(x_1,\dots,x_k)$, the positive action of $u$ with respect to $r\alpha$ is related to 
the positive action of $f_r^*(u)$ with respect to $\alpha$ by $\mathcal{A}^+_{r\alpha}(u)=r\mathcal{A}^+_{\alpha}(f_r^*u)$. This implies the scaling property $c_k(r\alpha) = rc_k(\alpha)$.

It is clear from the definition that $c_k$ is increasing in $k$ since adding point constraints can only make the set of relevant curves smaller.

Sublinearity is also clear because the union of two curves with $k$ and $\ell$ point constraints, respectively, yields a curve with $k+\ell$ point constraints.

Invariance follows because a contactomorphism $\varphi\in \operatorname{Cont}(M,\xi)$ lifts to a symplectomorphism of the symplectization which induces bijections between the relevant sets of auxiliary data and curves.

The Lipschitz property follows from scaling and monotonicity.

Spectrality follows from the fact that the action spectrum of $\alpha$ is a closed subset of $\R$.

The proofs of the monotonicity and packing property are more substantial and rely on a neck-stretching argument and the SFT compactness theorem. Compared to analogous arguments given in \cite{hut22a,edt22b}, the proofs are slightly simplified because we have an apriori genus bound and can apply the SFT compactness theorem directly. As these two proofs are very similar, we just sketch the proof of the packing property.

We fix $\alpha < \beta$ and a symplectically embedded ball $B= B^{2n}(A)$ in $\widetilde{M}_{\alpha}^{\beta}$. Moreover, we fix an almost complex structure $J\in \mathcal{J}(\alpha)$ and $k$ points $x_1,\dots, x_k\in \widetilde{M}$. Let $\epsilon>0$ be arbitrarily small. Our goal is to find a $J$-holomorphic curve $u\in \mathcal{M}^J(x_1,\dots,x_k)$ satisfying
\begin{equation*}
\mathcal{A}^+_{\alpha}(u) + A \leq c_{k+1}(\beta) + \epsilon.
\end{equation*}
Since the constant $\epsilon$, the almost complex structure $J$, and the point constraints are arbitrary, this implies the desired packing inequality.

Note that if we translate the almost complex structure $J$ and the points $x_i$ using the canonical $\R$-action on $\widetilde{M}$, the new moduli space $\mathcal{M}^J(x_1,\dots,x_k)$ is obtained from the old one by applying the same translation. The positive action $\mathcal{A}_\alpha^+(u)$ is not affected by translation. After translating $J$ and the points $x_i$ sufficiently far downwards, we can therefore assume w.l.o.g. that $J$ is adjusted to $\alpha$ on an open neighbourhood of $\widetilde{M}_\alpha^+$ and that all points $x_i$ lie below the hypersurface $M_\alpha$.

Fix a real number $S$ such that $J$ is adjusted to $\alpha$ on an open neighbourhood of $\widetilde{M}_{e^S\alpha}^-$ and such that all $x_i$ lie above the hypersurface $M_{e^S\alpha}$. Let $T\in\R$ be such that the hypersurface $M_{e^T\beta}$ lies below the hypersurface $M_{e^S\alpha}$. Let $J_\alpha^+$ denote the $\alpha$-adjusted almost complex structure whose restriction to $\widetilde{M}_\alpha^+$ agrees with $J$ and let $J_\alpha^-$ denote the $\alpha$-adjusted almost complex structure whose restriction to $\widetilde{M}_{e^s\alpha}^-$ agrees with $J$.

Let $J_\alpha^\beta$ be a compatible almost complex structure which agrees with $J_\alpha^+$ on a neighbourhood of $\widetilde{M}_\alpha^-$ and which is equal to a $\beta$-adjusted almost complex structure $J_\beta^+$ on $\widetilde{M}_\beta^+$. Moreover, we assume that the restriction of $J_\alpha^\beta$ to the ball $B$ agrees with the standard integrable almost complex structure on Euclidean space.

Let $J_\beta^\alpha$ be a compatible almost complex structure agreeing with $J_\alpha^-$ on a neighbourhood of $\widetilde{M}_{e^S\alpha}^+$ and whose restriction to $\widetilde{M}_{e^T\beta}^-$ agrees with a $\beta$-adjusted almost complex structure $J_\beta^-$. Let us fix small tubular neighbourhoods $\operatorname{nb}(M_\alpha)$ and $\operatorname{nb}(M_{e^S\alpha})$ and sequences of diffeomorphisms
\begin{equation*}
\varphi_\nu^+: \widetilde{M}_{e^{-\nu}\alpha}^{e^\nu\alpha} \rightarrow \operatorname{nb}(M_\alpha), \qquad \varphi_\nu^- : \widetilde{M}_{e^{-\nu}\alpha}^{e^\nu\alpha}\rightarrow \operatorname{nb}(M_{e^S\alpha})
\end{equation*}
with the following properties:
\begin{enumerate}
\item $\varphi_\nu^\pm$ maps translates of the hypersurface $M_\alpha$ to translates of $M_\alpha$.
\item $\pi\circ \varphi_\nu^\pm = \pi$ where $\pi:\widetilde{M}\rightarrow M$ is the canonical projection.
\item $d\varphi_\nu^\pm \partial_a = \partial_a$ near the boundary of $\widetilde{M}_{e^{-\nu}\alpha}^{e^\nu\alpha}$.
\end{enumerate}
Let us now define a sequence of almost complex structures $J^\nu\in \mathcal{J}(\beta)$ as follows:
\begin{enumerate}
\item $J^\nu$ agrees with $J_\alpha^\beta$ above $\operatorname{nb}(M_\alpha)$.
\item $J^\nu$ agrees with $(\varphi_\nu^+)_*J_\alpha^+$ on $\operatorname{nb}(M_\alpha)$.
\item $J^\nu$ agrees with $J$ between $\operatorname{nb}(M_\alpha)$ and $\operatorname{nb}(M_{e^S\alpha})$.
\item $J^\nu$ agrees with $(\varphi_\nu^-)_*J_\alpha^-$ on $\operatorname{nb}(M_{e^S\alpha})$.
\item $J^\nu$ agrees with $J_\beta^\alpha$ below $\operatorname{nb}(M_{e^S\alpha})$.
\end{enumerate}
Let $x_{k+1}$ denote the center of $B$. By the definition of $c_{k+1}(\beta)$ given in \eqref{eq:ckalpha}, there exists a sequence of $J^\nu$-holomorphic curves $u^\nu\in \mathcal{M}^{J^\nu}(x_1,\dots,x_{k+1})$ such that
\begin{equation*}
\mathcal{A}^+_{\beta}(u^\nu) \leq c_{k+1}(\beta) + \epsilon.
\end{equation*}
The SFT compactness theorem implies that, after possibly passing to a subsequence, the sequence of curves $u^\nu$ converges to a pseudo-holomorphic building $\underline{u} = (u_1,\dots,u_N)$. Each $u_i$ consists of finitely many pseudoholomorphic nodal punctured finite-energy spheres. In the following we will ignore nodal points and ghost bubbles and regard each $u_i$ as a finite union of non-constant pseudoholomorphic punctured finite-energy spheres. There exist indices $1\leq r < s < t \leq N$ such that
\begin{enumerate}
\item $u_i$ is $J_\beta^+$-holomorphic and positively and negatively asymptotic to closed orbits of $R_\beta$ for all $1\leq i < r$.
\item $u_r$ is $J_\alpha^\beta$-holomorphic and passes through $x_{k+1}$. It is positively asymptotic to closed orbits of $R_\beta$ and negatively asymptotic to closed orbits of $R_\alpha$.
\item $u_i$ is $J_\alpha^+$-holomorphic and positively and negatively asymptotic to closed orbits of $R_\alpha$ for all $r<i<s$.
\item $u_s$ is $J$-holomorphic and passes through $x_1,\dots,x_k$. It is positively and negatively asymptotic to closed orbits of $R_\alpha$.
\item $u_i$ is $J_\alpha^-$-holomorphic and positively and negatively asymptotic to closed orbits of $R_\alpha$ for all $s<i<t$.
\item $u_t$ is $J_\beta^\alpha$-holomorphic. It is positively asymptotic to closed orbits of $R_\alpha$ and negatively asymptotic to closed orbits of $R_\beta$.
\item $u_i$ is $J_\beta^-$-holomorphic and positively and negatively asymptotic to closed orbits of $R_\beta$ for all $t<i\leq N$.
\end{enumerate}
Note that since every $u^\nu$ has at least one negative end, each $u_i$ must have at least one negative end. Therefore, $u_s\in\mathcal{M}^J(x_1,\dots,x_k)$. We need to estimate the positive action $\mathcal{A}^+_\alpha(u_s)$. For every $i$, the negative asymptotic orbits of $u_i$ match the positive asymptotic orbits of $u_{i+1}$. This implies that
\begin{equation*}
\sum\limits_{i=1}^{r-1} \mathcal{A}_\beta(u_i) + \mathcal{A}_\alpha^\beta(u_r) + \sum\limits_{i=r+1}^{s-1} \mathcal{A}_\alpha(u_i) +\mathcal{A}_\alpha^+(u_s)  = \mathcal{A}_\beta^+(u_1).
\end{equation*}
Since $\mathcal{A}_\beta^+(u^\nu) \leq c_{k+1}(\beta) + \varepsilon$ for all $\nu$ and the positive ends of $u^\nu$ are converging to the positive ends of $u_1$, we have
\begin{equation*}
\mathcal{A}_\beta^+(u_1) \leq c_{k+1}(\beta) + \varepsilon.
\end{equation*}
The actions $\mathcal{A}_\beta(u_i)$ are non-negative for $1\leq i<r$. Similarly, the actions $\mathcal{A}_\alpha(u_i)$ are non-negative for $r<i<s$. Since $u_r$ passes through the center of the ball $B$ equipped with the standard complex structure, $\mathcal{A}_\alpha^\beta(u_r)\geq A$ by the monotonicity lemma. Combining the above estimates yields
\begin{equation*}
\mathcal{A}_\alpha^+(u_s) + A \leq c_{k+1}(\beta) + \varepsilon
\end{equation*}
as desired.
\end{proof}

In the next result, we determine the value of the first two spectral invariants for Zoll contact forms.

\begin{prop}
\label{zoll_computation}
Let $\alpha_0\in \mathcal{F}(M,\xi)$ be a Zoll contact form with minimal period $\operatorname{sys}(\alpha_0)$. Then we have
\begin{equation*}
c_0(\alpha_0) = c_1(\alpha_0) = \operatorname{sys}(\alpha_0).
\end{equation*}
\end{prop}

\begin{proof}
Let $\alpha$ be a $C^\infty$-small non-degenerate perturbation of $\alpha_0$. Fix an almost complex structure $J\in \mathcal{J}(\alpha)$ and a point $x\in \widetilde{M}$. Our goal is to show that there exists a $J$-holomorphic curve $u\in \mathcal{M}^J(x)$ with positive action bounded from above by
\[
\mathcal{A}^+_{\alpha}(u)\leq \tfrac{3}{2}\operatorname{sys}(\alpha_0).
\]Since $J$ and $x$ are arbitrary, this implies that $c_1(\alpha)\leq \frac{3}{2}\operatorname{sys}(\alpha_0)$. It then follows from continuity of $c_1$ that $c_1(\alpha_0)\leq \frac{3}{2}\operatorname{sys}(\alpha_0)$. The increasing and the spectrality properties imply that we must have $c_0(\alpha_0) = c_1(\alpha_0) = \operatorname{sys}(\alpha_0)$.

We construct the pseudoholomorphic curve $u$ by neck-stretching pseudoholomorphic spheres in a ruled symplectic manifold $(X,\omega)$ which contains $M_\alpha$ as a contact type hypersurface. The symplectic manifold $(X,\omega)$ is constructed by performing a so-called symplectic cut, see \cite{ler95}, as follows. Identify the symplectization $\widetilde{M}$ with $\R\times M$ by means of the diffeomorphism $\varphi_{\alpha_0}$, so that $\widetilde\omega$ is given by $d(e^s\alpha_0)$. Fix a constant $S>0$. The boundary components of the compact region $[-S,S]\times M$ in the symplectization are foliated by closed characteristic leaves. The symplectic manifold $(X,\omega)$ is obtained from $[-S,S]\times M$ by collapsing each of these leaves to a point. Let $B$ denote the quotient of $M$ by the free $S^1$-action induced by the Zoll Reeb flow. The space $X$ naturally fibers over $B$. The fibers are symplectic $2$-spheres of the form $([-S,S]\times \gamma) / \sim$ where $\gamma$ is a closed Reeb orbit of $\alpha_0$ and $\sim$ collapses the circles $\left\{ \pm S \right\}\times \gamma$ to points. Under the quotient, the two boundary components of $[-S,S]\times M$ yield smooth symplectic divisors $B_\pm$ of $X$ both diffeomorphic to $B$. Clearly, $X$ contains $M_\alpha$ as a contact type hypersurface. The function $(s,p)\mapsto -\operatorname{sys}(\alpha_0)e^s$ on $\R\times M$ induces a smooth Hamiltonian $H$ on $X$ which generates a Hamiltonian $S^1$-action. The two divisors $B_\pm$ are fixed point sets of this action and the restriction of the action to the complement $X\setminus (B_+\cup B_-)$ is free. In particular, the assumptions in \cite[Proposition 4.3]{mcd09} are satisfied. Let $F\in H_2(X;\Z)$ denote the homology class of a fiber. Then \cite[Proposition 4.3]{mcd09} implies that the Gromov--Witten invariant $\langle \mathrm{pt} \rangle_{F}^X$ does not vanish. We conclude that, for every $\omega$-compatible almost complex structure $\widetilde{J}$ on $X$ and for every point $p\in X$, there exists a nodal $\widetilde{J}$-holomorphic sphere passing through $p$ and representing the homology class $F$.

Let us now choose a constant $S'$ large enough such that the almost complex structure $J\in \mathcal{J}(\alpha)$ fixed at the beginning of the proof is adjusted to $\alpha$ on neighbourhoods of the sets $\widetilde{M}_{e^{\pm S'}\alpha}^\pm$. Moreover, we assume that the point $x$ strictly lies between the two hypersurfaces $M_{e^{\pm S'}\alpha}$. We set the constant $S$ appearing in the construction of $X$ to be $S\coloneqq S'+\ln(3/2)$. Since $\alpha$ is $C^\infty$-close to $\alpha_0$, the hypersurfaces $M_{e^{\pm S'}\alpha}$ lie between the hypersurfaces $M_{e^{\pm S}\alpha_0}$ and do not intersect them. 

We define a sequence of compatible almost complex structures $J^\nu$ on $X$ as follows. Let $J_\alpha^\pm$ denote the $\alpha$-adjusted almost complex structure whose restriction to $\widetilde{M}_{e^{\pm S'}}^\pm$ agrees with $J$. Let $X_\pm\subset X$ denote the connected component of $X\setminus M_{e^{\pm S'}\alpha}$ containing $B_\pm$. Let $\widehat{X}_\pm$ denote the symplectic completion of $X_\pm$ obtained by attaching the cylindrical end $\widetilde{M}_{e^{\pm S'}\alpha}^{\mp}$. Let $J_\pm$ be a compatible almost complex structure on $\widehat{X}_\pm$ agreeing with $J_\alpha^\pm$ in a neighbourhood of the cylindrical end $\widetilde{M}_{e^{\pm S'}\alpha}^{\mp}$. 

Let us fix small tubular neighbourhoods $\operatorname{nb}(M_{e^{\pm S'}\alpha})$ and sequences of diffeomorphisms
\begin{equation*}
\varphi_\nu^\pm: \widetilde{M}_{e^{-\nu}\alpha}^{e^\nu\alpha} \rightarrow \operatorname{nb}(M_{e^{\pm S'}\alpha})
\end{equation*}
with the following properties:
\begin{enumerate}
\item $\varphi_\nu^\pm$ maps translates of the hypersurface $M_\alpha$ to translates of $M_\alpha$.
\item $\pi\circ \varphi_\nu^\pm = \pi$ where $\pi:\widetilde{M}\rightarrow M$ is the canonical projection.
\item $d\varphi_\nu^\pm \partial_a = \partial_a$ near the boundary of $\widetilde{M}_{e^{-\nu}\alpha}^{e^\nu\alpha}$.
\end{enumerate}
Let us now define the compatible almost complex structure $J^\nu$ on $X$ as follows:
\begin{enumerate}
\item $J^\nu$ agrees with $J_+$ above $\operatorname{nb}(M_{e^{S'}\alpha})$.
\item $J^\nu$ agrees with $(\varphi_\nu^+)_*J_\alpha^+$ on $\operatorname{nb}(M_{e^{S'}\alpha})$.
\item $J^\nu$ agrees with $J$ between $\operatorname{nb}(M_{e^{S'}\alpha})$ and $\operatorname{nb}(M_{e^{-S'}\alpha})$.
\item $J^\nu$ agrees with $(\varphi_\nu^-)_*J_\alpha^-$ on $\operatorname{nb}(M_{e^{-S'}\alpha})$.
\item $J^\nu$ agrees with $J_-$ below $\operatorname{nb}(M_{e^{-S'}\alpha})$.
\end{enumerate}
By the above discussion, we may choose, for every $\nu$, a $J^\nu$-holomorphic sphere $u^\nu$ representing the homology class of the fiber $F$ and passing through the point $x\in X$. By the SFT compactness theorem, this sequence converges, after possibly passing to a subsequence, to a pseudoholomorphic building $\underline{u} = (u_1,\dots,u_N)$. Every $u_i$ consists of finitely many nodal punctured finite-energy pseudoholomorphic spheres. As before, we ignore nodes and ghost bubbles and regard each $u_i$ as a finite union of punctured finite energy spheres.

The homological intersection number $F\cdot B_{\pm}$ is equal to $1$. This implies that every $u^\nu$ intersects both $B_\pm$. Therefore, $u_1$ is a $J_+$-holomorphic curve in $\widehat{X}_+$ and its intersection number with $B_+$ is equal to $1$. Similarly, $u_N$ is a $J_-$-holomorphic curve in $\widehat{X}_-$ intersecting $B_-$ with intersection number $1$. Since the negative asymptotic limit orbits of $u_i$ match the positive asymptotic limit orbits of $u_{i+1}$, we conclude that every $u_i$ for $1<i<N$ must have both positive and negative asymptotic orbits. Let $1<r<N$ be the index such that $u_r$ is a $J$-holomorphic curve in $\widetilde{M}$ passing through $x\in\widetilde{M}$. Since $u_r$ has a negative asymptotic orbit, we have $u_r\in \mathcal{M}^J(x)$. It remains to estimate $\mathcal{A}_\alpha^+(u)$.

The complement of $B_+$ in $\widehat{X}_+$ is exact. Indeed, restricting the natural Liouville $1$-form $\widetilde{\lambda}$ on $\widetilde{M}$ to the interior of $\widetilde{M}_{e^S\alpha_0}^-$ yields a primitive $\lambda$ of the symplectic form $\omega_+$ on $\widehat{X}_+\setminus B_+$. Clearly, the primitive $\lambda$ does not extend over $B_+$. In fact, if $\Sigma$ is a compact oriented surface, possibly with boundary, and $v:\Sigma\rightarrow \widehat{X}_+$ is a smooth map sending $\partial\Sigma$ into the complement of $B_+$, then
\begin{equation*}
\int_{\Sigma}v^*\omega_+ = \int_{\partial\Sigma} v^*\lambda + e^S\operatorname{sys}(\alpha_0) v\cdot B_+.
\end{equation*}
In order to make sense of the homological intersection numer $v\cdot B_+$, we regard $v$ as a $2$-chain with fixed boundary in the complement of $B_+$. Applying this identity to $u_1$ yields
\begin{equation*}
\int_{u_1^{-1}(\widetilde{M}^-_{e^{S'}\alpha})} (\varphi_\alpha^{-1}\circ u_1)^* d\alpha + \int_{u_1^{-1}(X_+)} u_1^*\omega_+ = -\mathcal{A}_{e^{S'}\alpha}^-(u_1) + e^S\operatorname{sys}(\alpha_0).
\end{equation*}
Since $u_1$ is pseudoholomorphic, the left-hand side of this identity is non-negative. We deduce that
\begin{equation}
e^{S'}\mathcal{A}_\alpha^-(u_1) = \mathcal{A}_{e^{S'}\alpha}^-(u_1)\leq e^S\operatorname{sys}(\alpha_0).
\end{equation}
It follows from our choice $S=S'+\ln(3/2)$ that $\mathcal{A}_\alpha^-(u_1) \leq \frac{3}{2}\operatorname{sys}(\alpha_0)$.

The curve $u_i$ is a $J_\alpha^+$-holomorphic curve in $\widetilde{M}$ for $1<i<r$. We have
\begin{equation*}
\sum\limits_{i=2}^{r-1} \mathcal{A}_\alpha(u_i) + \mathcal{A}_\alpha^+(u_r) = \mathcal{A}_{\alpha}^-(u_1).
\end{equation*}
Since the actions $\mathcal{A}_{\alpha}(u_i)$ are non-negative for $1<i<r$, we deduce that
\begin{equation*}
\mathcal{A}_\alpha^+(u_r) \leq \mathcal{A}_{\alpha}^{-}(u_1) \leq \frac{3}{2}\operatorname{sys}(\alpha_0).
\end{equation*}
This concludes the proof of the proposition.
\end{proof}

Next, we compute the first two spectral invariants for contact forms which are $C^2$-small, $S^1$-invariant perturbations of a Zoll one.

\begin{prop}
\label{c_kqi}
Let $\alpha_0\in \mathcal{F}(M,\xi)$ be a Zoll contact form with minimal period $\mathrm{sys}(\alpha_0)$. If the smooth function $S:M\rightarrow \R$ is $C^2$-close to the constant function 1 and $S^1$-invariant, then 
\[
c_0(S \alpha_0) = \mathrm{sys}(\alpha_0) \min_M S, \qquad c_1(S \alpha_0) = \mathrm{sys}(\alpha_0) \max_M S.
\]
\end{prop}

\begin{proof}
We abbreviate $\max S\coloneqq\max_M S$ and $\min S\coloneqq\min_M S$. Our proof is based on the following claims.

\medskip

\noindent{\sc Claim 1.} If $S$ is $C^2$-close to 1 and the positive number $\overline{S}> \max S$ is close to $\max S$, then there is a symplectic embedding of the $2n$-dimensional ball of symplectic width $\mathrm{sys}(\alpha_0) (\max S - \min S)$ into the cobordism $\widetilde{M}_{S \alpha_0}^{\overline{S} \alpha_0}$.

\medskip

\noindent{\sc Claim 2.} If $S$ is $C^2$-close to 1 and the positive number $\underline{S}< \min S$ is close to $\min S$, then there is a  symplectic embedding of the $2n$-dimensional ball of symplectic width $\mathrm{sys}(\alpha_0) (\max S - \min S)$ into the cobordism $\widetilde{M}_{\underline{S} \alpha_0}^{S \alpha_0}$.

\medskip

Before proving these claims, we show how they allow us to determine $c_0(S \alpha_0)$ and $c_1(S \alpha_0)$. By Claim 1, the packing property applied to $\widetilde{M}_{S \alpha_0}^{\overline{S} \alpha_0}$, the scaling property and Proposition \ref{zoll_computation}, we have
\[
c_0(S\alpha_0) + \mathrm{sys}(\alpha_0)(\max S- \min S) \leq c_1 (\overline{S} \alpha_0) = \overline{S}\,  c_1(\alpha_0) = \overline{S} \, \mathrm{sys}(\alpha_0),
\]
and hence, letting $\overline S$ tend to $\max S$, we get
\[
c_0(S\alpha_0) \leq \mathrm{sys}(\alpha_0) \min S.
\]
Since the right-hand side is the minimum of the action spectrum of $S\alpha_0$, the spectrality property implies that the above inequality is an equality. Similarly, the packing property applied to $\widetilde{M}_{\underline{S} \alpha_0}^{S\alpha_0}$, the scaling property and Proposition \ref{zoll_computation} imply
\[
c_1(S {\alpha_0}) \geq c_0(\underline{S} \alpha_0) + \mathrm{sys}(\alpha_0)  (\max S - \min S) = \underline{S}\, c_0(\alpha_0) + \mathrm{sys}(\alpha_0)  (\max S - \min S).
\]
Letting $\underline S$ tend to $\min S$, we get $c_1(S\alpha_0)\geq \mathrm{sys}(\alpha_0) \max S$. By monotonicity, we have
\[
c_1(S\alpha_0) \leq c_1( (\max S) \, \alpha_0) = \mathrm{sys}(\alpha_0)  \max S,
\]
and the desired formula for $c_1(S \alpha_0)$ follows.

\medskip

There remains to prove the above claims. We prove Claim 1, the proof of Claim 2 being analogous. Up to rescaling, we may assume that $\mathrm{sys}(\alpha_0) = 1$.
We identify the symplectization $\widetilde{M}$ with $\R\times M$ via $\varphi_{\alpha_0}$, so that $\widetilde{\omega} = d (e^a \alpha_0)$ and 
\[
M_{S\alpha_0}^{\overline{S} \alpha_0} = \{(a,q)\in \R\times M \mid S(q) < e^a < \overline{S} \}.
\]
Denote by $p: M \rightarrow Y$ the $\T$-bundle, $\T\coloneqq \R/\Z$, whose base is the quotient of $M$ by the free action of the Reeb flow of $\alpha_0$. Then $Y$ is a $(2n-2)$-dimensional manifold with a symplectic form $\omega$ satisfying $p^* \omega = d\alpha_0$. By means of a Darboux trivializing chart, we can identify a neighborhood $(U_y,\omega)$ of any $y$ in $Y$ with a ball $\widehat{B}_{\rho}$ of radius $\rho>0$ in $(\C^{n-1},\widetilde{\omega}_0)$, where $\rho$ is independent of $y$. Moreover, we can identify $p^{-1}(U_y)$ with $\T \times \widehat{B}_\rho$ in such a way that $\alpha_0$ corresponds to $\widehat{\lambda}_0 + dt$. Here, $t$ denotes the variable in $\T$ and $\widehat{\lambda}_0$ is the standard primitive of the standard symplectic form $\widehat{\omega}_0$ on $\C^{n-1}$ (we are using the notation from Section \ref{lifts}).

Correspondingly, we obtain an identification of the open subset $\widetilde{U}_y\coloneqq \R \times p^{-1}(U_y)$ of $\widetilde{M}$ with $\R \times \T \times \widehat{B}_{\rho}$, with respect to which the 1-form $e^a \alpha_0$ reads $e^a (dt + \widehat{\lambda}_0)$. The diffeomorphism
\[
\Psi_y: \{ (s,t,w)\in (0,+\infty)\times \T \times \C^{n-1} \mid |w| < s^{\frac{1}{2}} \rho\} \rightarrow \R \times \T \times \widehat{B}_{\rho} = \widetilde{U}_y \subset \widetilde{M}
\]
given by $\Psi_y(s,t,w)\coloneqq(\log s,t,s^{-\frac12}w)$ satisfies
\[
\Psi_y^* (e^a \alpha_0) = \Psi_y^*( e^a(dt + \widehat{\lambda}_0)) = s\, dt + \widehat{\lambda}_0,
\]
and hence
\[
\Psi_y^* \widetilde{\omega} = \Psi_y^* d  (e^a \alpha_0)  = ds\wedge dt + \widetilde{\omega}_0.
\]
Now we choose $y\in Y$ so that the $\T$-invariant function $S$ achieves its minimum at the circle $p^{-1}(y)$. In the above identifications, we then have $S(t,0) = \min S$ for every $t\in \T$, 
\[
\widetilde{M}_{S \alpha_0}^{\overline{S} \alpha_0} \cap \widetilde{U}_y = \{ (a,t,u)\in \R \times \T\times \widehat{B}_\rho \mid S(t,u) < e^a < \overline{S} \},
\]
and hence it is enough to symplectically embed a $2n$-ball of symplectic width $\max S - \min S$ into the open subset
\[
V \coloneqq \Psi^{-1}_y ( \widetilde{M}_{S \alpha_0}^{\overline{S} \alpha_0} \cap \widetilde{U}_y ) = \{ (s,t,w) \in (0,+\infty) \times \T \times \C^{n-1} \mid S(t,s^{-\frac{1}{2}} w) < s < \overline{S}, \; |w|< s^{\frac{1}{2}} \rho\}
\]
of $((0,+\infty)\times \T \times \C^{n-1}, ds\wedge dt + \widehat{\omega}_0)$. Since $S(t,0) = \min S$, we have
\[
S(t,u) \leq \min S + c \|\nabla^2 S\|_{C^0} |u|^2 \qquad \forall (t,u)\in \T\times \widehat{B}_{\rho},
\]
for a suitable constant $c>0$. By this bound, the inequality $s>S(t,s^{-\frac{1}{2}}w)$ appearing in the representation of $V$ is implied by
\[
s > \min S + c \|\nabla^2 S\|_{C^0} \frac{|w|^2}{s},
\]
which is equivalent to
\[
s > \frac{1}{2} \min S \left( 1 + \sqrt{1 + \frac{4c \|\nabla^2 S\|_{C^0}}{(\min S)^2} |w|^2 } \right).
\]
Since $\sqrt{1+r}\leq 1 + \frac{r}{2}$ for every $r\geq 0$, the above inequality is implied by
\[
s> \min S + \frac{c}{\min S} \|\nabla^2 S\|_{C^0} |w|^2,
\] 
and we conclude that $V$ contains the set
\[
V' \coloneqq \{(s,t,w) \in (0,+\infty) \times \T \times \C^{n-1} \mid f(|w|^2) < s < \overline{S} \},
\]
where
\[
f(r) \coloneqq \max \left\{ \min S + \frac{c\|\nabla^2 S\|_{C^0}}{\min S} r, \frac{r}{\rho^2} \right\}.
\]
If $S$ is $C^2$-close to $1$ and $\overline{S}$ is close to $\max S$, then both $\overline{S} - \min S$ and $\|\nabla^2 S\|_{C^0}$ are small and we can guarantee that
\[
f(r) \leq \min S + \pi r \qquad \forall r\in \bigl[0, {\textstyle \frac{1}{\pi}} ( \overline{S} - \min S) \bigr].
\]
This implies that $V'$ contains the set
\[
V''\coloneqq \bigl\{ (s,t,w) \in (0,+\infty) \times \T \times \C^{n-1} \mid \min S + \pi |w|^2 < s < \overline{S} \bigr\}.
\]
Let $\Phi: \Omega\rightarrow \C^* \times \C^{n-1}$ be the symplectomorphism defined in \eqref{Phi}. By \eqref{normPhi}, the symplectomorphism
\[
(s,t,w) \mapsto \Phi\bigr( s - \min S - \pi,t,w \bigr)
\]
maps $V''$ onto $B(\rho)\setminus\{z_1=0\}$, where $B(\rho)$ is the ball of $\C^n$ of symplectic width $\rho\coloneqq\overline{S} - \min S$ and $\{z_1=0\}$ is a complex hyperplane. By Traynor's embedding construction (see Construction 3.2 in \cite{tra95}), for every $\rho'\in(0,\rho)$, there is a symplectic embedding $\Psi\colon B(\rho')\to B(\rho)\setminus \{z_1=0\}$. In the simple case we are considering here, the embedding is constructed as follows. First, let $\psi\colon D(\rho')\to D(\rho)\setminus\{0\}$ be an area preserving embedding such that 
	\[
	\pi|\psi(\zeta)|^2\leq \pi |\zeta|^2+(\rho-\rho'),\qquad \forall\,\zeta\in D(\rho').
	\]
Then, \[
\Psi(\zeta_1,\ldots,\zeta_n)\coloneqq (\psi(\zeta_1),\zeta_2,\ldots,\zeta_n),\qquad (\zeta_1,\ldots,\zeta_n)\in B(\rho')
\]
is the required symplectic embedding.
Choosing $\rho'=\max S-\min S<\overline S-\min S=\rho$, we deduce that $V\supset V''$ contains a symplectic ball of width $\max S - \min S$.
\end{proof}

\section{Proof of Theorem \ref{short-geodesics}}
\label{secthm5}
Let $\alpha\in \mathcal{F}(M,\xi)$ be $C^2$-close to the Zoll contact form $\alpha_0\in \mathcal{F}(M,\xi)$. By Theorem \ref{qi}, there is $\varphi\in \mathrm{Cont}_0(M,\xi)$ such that
\[
\varphi^* \alpha = T e^h \alpha_0,
\]
where the functions $T$ and $h$ satisfy the conditions (a)-(e) listed in that theorem. We denote by $T_{\min}(\alpha)=\mathrm{sys}(\alpha)$ and $T_{\max}(\alpha)$ the minimum and the maximum of the periods of the short closed orbits of $R_{\alpha}$, which by (e) satisfy
\begin{equation}
\label{Tminmax}
T_{\min}(\alpha) = \mathrm{sys}(\alpha_0) \min T, \qquad  T_{\max}(\alpha) = \mathrm{sys}(\alpha_0) \max T.
\end{equation}
Then $\varphi^* \alpha = e^f \alpha_0$ with
\[
f\coloneqq h + \log T,
\]
and by property (c) of Theorem \ref{qi} we have
\[
\max_M f - \min_M f = \log \frac{T_{\max}(\alpha)}{T_{\min}(\alpha)}.
\]
This shows that
\[
d(\alpha_0,\alpha) \leq \log \frac{T_{\max}(\alpha)}{T_{\min}(\alpha)}.
\]
In order to conclude our proof, we need to prove the reverse inequality
\begin{equation}
\label{reverse}
d(\alpha_0,\alpha) \geq \log \frac{T_{\max}(\alpha)}{T_{\min}(\alpha)}.
\end{equation}
We shall deduce this from the following claim.

\medskip

\noindent {\sc Claim.} $c_0(\alpha)=T_{\min}(\alpha)$ and $c_1(\alpha)=T_{\max}(\alpha)$.

\medskip

Postponing its proof, we show how this claim implies \eqref{reverse}. Let $\psi\in \mathrm{Cont}_0(M,\xi)$ and let $g$ be the smooth positive function such that
\[
\psi^* \alpha = e^g \alpha_0.
\]
Denoting by 
\[
a\coloneqq e^{\min_ M g} = \min_M e^g, \qquad b\coloneqq e^{\max_M g} = \max_M e^g,
\]
we have
\[
a \alpha_0 \leq \psi^* \alpha \leq b \alpha_0.
\]
Using Proposition \ref{zoll_computation}, the scaling, monotonicity and invariance properties of the spectral invariants, and the above claim, we compute
\[
a\, \mathrm{sys}(\alpha_0) = c_0 (a \, \alpha_0) \leq c_0(\psi^* \alpha) = c_0(\alpha) = T_{\min}(\alpha),
\]
and
\[
b\, \mathrm{sys}(\alpha_0) = c_1 (b \, \alpha_0) \geq c_1(\psi^* \alpha) = c_1(\alpha) = T_{\max}(\alpha).
\]
Therefore
\[
\max_M g - \min_M g = \log \frac{b}{a} \geq \log \frac{T_{\max}(\alpha)}{T_{\min}(\alpha)}.
\]
Since the contactomorphism $\psi\in \mathrm{Cont}_0(M,\xi)$ was arbitrary, the above inequality implies \eqref{reverse}.

\medskip

There remains to prove the claim. In the notation from the proof of Theorem \ref{qi}, we have
\[
v^*u^*\alpha = S\alpha_0 + \eta.
\]
By invariance, this implies that
\[
c_k(\alpha) = c_k(S\alpha_0+\eta),
\]
where $c_k$ is viewed as a function on $\mathcal{F}(M,\operatorname{ker}(S\alpha_0+\eta))$ on the right hand side. For $t\in [0,1]$, we define
\[
\gamma_t \coloneqq S\alpha_0 + t\eta.
\]
Since $S$ is $C^1$-close to $1$ and $\eta$ and $d\eta$ are both $C^0$-small, $\gamma_t$ is a contact form for all $t$. Moreover, $\gamma_t$ and $d\gamma_t$ are $C^0$-close to $\alpha_0$ and $d\alpha_0$, respectively.

\medskip
\noindent {\sc Claim.} The short closed Reeb orbits of $\gamma_t$ and their periods are independent of $t\in [0,1]$.

\begin{proof}
Let us split the Reeb vector field $R_{\gamma_t}$ of $\gamma_t$ into a component $X_t$ parallel to $R_{\alpha_0}$ and a component $Y_t$ tangent to $\xi = \operatorname{ker}\alpha_0$. Since $\gamma_t$ and $d\gamma_t$ are $C^0$-close to $\alpha_0$ and $d\alpha_0$, the component $X_t$ is $C^0$-close to $R_{\alpha_0}$ and the component $Y_t$ is $C^0$-small. Let us write $X_t= a_t R_{\alpha_0}$ for some function $a_t$ which is $C^0$-close to $1$. Then a direct computation shows that
\[
\iota_{X_t}d\gamma_t = a_t(-dS + tF[dS])
\]
and therefore
\[
\iota_{Y_t}d\gamma_t = a_t(dS - tF[dS]).
\]
Recall that $dS=-B[V]$ where $B:\operatorname{ker}\alpha_0\rightarrow (\operatorname{ker}\alpha_0)^*$ is $C^0$-close to the isomorphism induced by $d\alpha_0$ and $V$ is a section of $\operatorname{ker}\alpha_0$ which is $C^1$-small. Let $B_t:\operatorname{ker}\alpha_0\rightarrow (\operatorname{ker}\alpha_0)^*$ denote the isomorphism induced by $d\gamma_t$, which is also $C^0$-close to the isomorphism induced by $d\alpha_0$. We have
\[
Y_t = a_t B_t^{-1}(B[V]-tF[B[V]]) = a_t Q_t[V]
\]
where $Q_t$ is a suitable isomorphism of $\operatorname{ker}\alpha_0$ which is $C^0$-close to the identity. We may therefore write $R_{\gamma_t} = a_t (R_{\alpha_0} + Q_t[V])$. The argument in the proof of Proposition \ref{proposition:normal_form_abbondandolo_benedetti_short_orbits} (ii) shows that the short closed Reeb orbits of $\gamma_t$ are precisely the zero set of $V$ or equivalently the critical set of $S$. The action of such an orbit $\gamma$ is simply given by $S(\gamma)\operatorname{sys}(\alpha_0)$ and in particular independent of $t$.
\end{proof}

The above claim shows that the short action spectrum of $\gamma_t$ is independent of $t$. The function $t\mapsto c_k(\gamma_t)$ for $k\in \left\{ 0,1 \right\}$ is continuous and takes values in this short action spectrum. Since the action spectrum has measure zero, this implies that $t\mapsto c_k(\gamma_t)$ is constant for $k\in \left\{ 0,1 \right\}$. In particular, we obtain
\[
c_0(\alpha) = c_0(S\alpha_0) \quad \text{and} \quad c_1(\alpha) = c_1(S\alpha_0).
\]
Next, we define $\beta_t \coloneqq (1-t)S\alpha_0 + t\alpha_0$. Since $S$ is $C^1$-close to $1$, this is a contact form for all $t\in [0,1]$.

\medskip
\noindent {\sc Claim.} The short closed Reeb orbits of $\beta_t$ are independent of $t\in (0,1]$. If $\gamma$ is such a short closed Reeb orbit, then its action with respect to $\beta_t$ is given by $((1-t) S(\gamma)+t) \operatorname{sys}(\alpha_0)$.

\begin{proof}
The proof of this claim is similar to the proof of the claim above.
\end{proof}

It follows from continuity of the spectral invariants, the above claim and the fact that action spectra have measure zero that
\[
c_k(\beta_t) = (1-t)c_k(S\alpha_0) + t\operatorname{sys}(\alpha_0) \quad \text{for $k\in \left\{ 0,1 \right\}$}.
\]
If $t>0$ is sufficiently small, then $(1-t)S+t$ is $C^2$-small. Therefore, we may apply Proposition \ref{c_kqi} to the contact form $\beta_t$ for $t>0$ small. We conclude that
\[
c_0(\beta_t) = \min_M ((1-t)S + t) \operatorname{sys}(\alpha_0) \quad\text{and}\quad c_1(\beta_t) = \max_M ((1-t)S + t) \operatorname{sys}(\alpha_0).
\]
In combination with the above expression for $c_k(\beta_t)$, this implies
\[
c_0(S\alpha_0) = \min_M S\operatorname{sys}(\alpha_0) = T_{\min}(\alpha) \quad\text{and}\quad c_1(S\alpha_0) = \max_M S\operatorname{sys}(\alpha_0)= T_{\max}(\alpha).
\]
This concludes the proof of the claim.

\appendix

\section{Appendix: Construction of exotic ellipsoids}

\paragraph{The statements.} Denote by $\alpha_0=\lambda_0|_{S^3}$ the standard contact form on the 3-sphere, see \eqref{lambda_0}, by $\xi_{\mathrm{st}}=\ker \alpha_0$ the induced co-oriented contact structure, and by $\mathcal{F}(S^3,\xi_{\mathrm{st}})$ the space of contact forms on $S^3$ which define $\xi_{\mathrm{st}}$. This space can be identified with the space of positive smooth functions on $S^3$ and is endowed with the $C^{\infty}$-topology. If $(a,b)$ is a pair of positive real numbers, we denote by $E(a,b)$ the open ellipsoid
\[
E(a,b)\coloneqq \Bigr\{ (z_1,z_2) \in \C^2 \mid \frac{\pi}{a} |z_1|^2 + \frac{\pi}{b} |z_2|^2 < 1 \Bigr\},
\]
and by $\varepsilon_{a,b} \in \mathcal{F}(S^3,\xi_{\mathrm{st}})$ the contact form which is obtained by pulling the restriction of $\lambda_0$ to $\partial E(a,b)$ back to $S^3$ by the radial projection, i.e.\
\[
\epsilon_{a,b} (z) = \left( \frac{\pi}{a}|z_1|^2 + \frac{\pi}{b} |z_2|^2 \right)^{-1} \alpha_0(z), \qquad \forall z=(z_1,z_2)\in S^3.
\]
The Reeb flow of $\varepsilon_{a,b}$ has the form
\[
\phi_{\varepsilon_{a,b}}^t (z_1,z_2) = \bigl(e^{\frac{2\pi}{a} t i} z_1, e^{\frac{2\pi}{b} t i} z_2 \bigr).
\]
The circles $\Gamma_1\coloneqq \{z_2=0\}\subset S^3$ and $\Gamma_2\coloneqq \{z_1=0\}\subset S^3$ are closed Reeb orbits of $\varepsilon_{a,b}$ of period $a$ and $b$, respectively. The Reeb flow of $\varepsilon_{a,b}$ preserves the 2-tori
\[
\{|z_1|=\cos \beta, \; |z_2|=\sin \beta\} \subset S^3, \qquad \forall \beta\in \left( 0, {\textstyle \frac{\pi}{2}} \right),
\]
on which the flow is periodic if and only if the number $\frac{a}{b}$ is rational. 

The first aim of this Appendix is to present a self-contained proof of the following result.

\begin{thm}
\label{katok}
Let $(a_0,b_0)$ be a pair of positive numbers. For every neighborhood $U$ of $(a_0,b_0)$ in $(0,+\infty)^2$ and every neighborhood $\mathcal{U}$ of $\varepsilon_{a_0,b_0}$ in $\mathcal{F}(S^3,\xi_{\mathrm{st}})$ there exist $(a,b)\in U$, a contact form $\alpha$ in $\mathcal{U}$ and a sequence $(\varphi_j)$ in $\mathrm{Cont}_0(S^3,\xi_{\mathrm{st}})$, each having support away from $\Gamma_1 \cup \Gamma_2$, such that:
\begin{enumerate}[(i)]
\item $(\varphi_j^* \varepsilon_{a,b})$ converges to $\alpha$ in the $C^\infty$-topology;
\item the only closed Reeb orbits of $\alpha$ are $\Gamma_1$ and $\Gamma_2$, with periods $a$ and $b$, respectively;
\item the Reeb flow of $\alpha$ has a dense orbit.
\end{enumerate}
\end{thm} 

By (iii), the Reeb flow of $\alpha$ is not conjugate to the one of $\varepsilon_{a,b}$. Together with (i), this implies that the orbit of $\varepsilon_{a,b}$ by the action of $\mathrm{Cont}_0(S^3,\xi_{\mathrm{st}})$ is not $C^{\infty}$-closed in $\mathcal{F}(S^3,\xi_{\mathrm{st}})$, see Example \ref{irrational} in the Introduction. The above result can be deduced  from Theorem A in \cite{kat73}, where Katok used the conjugacy method which he had introduced together with Anosov in \cite{ak70b} to construct ergodic Hamiltonian flows arbitrarily close to suitably integrable flows. Therefore, (iii) can be replaced by the stronger condition:
\begin{enumerate}[(i')]
\setcounter{enumi}{2}
\item the Reeb flow of $\alpha$ is ergodic with respect to the invariant measure induced by the volume form $\alpha\wedge d\alpha$.
\end{enumerate}
Specializing the argument to the perturbation of $\varepsilon_{a_0,b_0}$ and showing  just the existence of a dense orbit, instead of ergodicity, allows us to give here a short self-contained proof, still based on the Anosov-Katok conjugacy method. See also \cite{cs16} for a reformulation of Katok's result in the Reeb setting.

\begin{rem}
\label{more}
{\rm In the proof of Theorem \ref{katok}, the pair $(a,b)$ is obtained as a limit of pairs of positive numbers $(a_j,b_j)$ with rational ratio which is required to converge extremely fast and we do not control which kind of limiting pairs $(a,b)$ we can produce. It is actually possible to strengthen the above result by prescribing the pair $(a,b)$: For any pair of positive numbers $(a,b)$ whose ratio is irrational and not Diophantine, there exist a contact form $\alpha\in \mathcal{F}(S^3,\xi_{\mathrm{st}})$  and a sequence $(\varphi_j)$ in $\mathrm{Cont}_0(S^3,\xi_{\mathrm{st}})$ which satisfy (i), (ii) and (iii) (or even (iii')) with respect to $(a,b)$. This follows from a corresponding result for area-preserving diffeomorphisms of the 2-disk which is proven in \cite{fs05b} and by a construction from \cite{agz22}, which allows to lift such diffeomorphisms to Reeb flows on $(S^3,\xi_{\mathrm{st}})$. 

The condition on $(a,b)$ is optimal. Indeed, if the ratio $\frac{a}{b}$ is rational, then any contact form $\alpha$ which satisfies (i) (even just with $C^0$-convergence) is smoothly conjugate to $\varepsilon_{a,b}$, as discussed in Example \ref{rational} in the Introduction, so cannot satisfy (iii). If the ratio $\frac{a}{b}$ is irrrational and Diophantine and the contact form $\alpha$ satisfies (i), then the first-return map to a local Poincar\'e section transverse to $\Gamma_1$ of the Reeb flow of $\alpha$ is an area-preserving diffeomorphism with an elliptic and Diophantine fixed point. In his ``last geometric theorem'', Herman proved that arbitrarily near to such a fixed point there are smooth invariant circles, see \cite[Theorem 4]{fk09}, and hence the Reeb flow of $\alpha$ has invariant tori. Since these tori disconnect $S^3$, $\alpha$ cannot satisfy (iii).}
\end{rem}

Writing the contact form $\alpha$ of Theorem \ref{katok} as $\alpha=f^2 \alpha_0$ for a suitable smooth positive function on $S^3$, we consider the starshaped domain $A\subset \C^2$ which is given by 
\[
A\coloneqq \{ r z \mid z\in S^3, \; 0\leq r < f(z) \}.
\]
The pullback by the radial projection $S^3\rightarrow \partial A$ of $\lambda_0|_{\partial A}$ to $S^3$ is then $\alpha$. We can complement Theorem \ref{katok} with the following result. 

\begin{thm}
\label{eliashberg-hofer}
The domains $A$ and $E(a,b)$ are symplectomorphic but the contact forms which are given by the restrictions of $\lambda_0$ to their boundaries are not conjugate.
\end{thm}

The first examples of pairs of starshaped domains in $\C^n$, $n>1$, having the above properties were constructed by Eliashberg and Hofer in \cite{eh96}. The construction we present here is somehow simpler, but builds on some results which are known to hold only in dimension four. 

\paragraph{Proof of theorem \ref{katok}.} By slightly perturbing $(a_0,b_0)$, we may assume that the ratio $\frac{a_0}{b_0}$ is rational.
Starting from the pair $(a_0,b_0)$ and from $\varphi_0=\mathrm{id}$, we shall inductively construct sequences $(a_j,b_j)$ in $U$ with $\frac{a_j}{b_j}\in \Q$ and $(\varphi_j)$ in $\mathrm{Cont}_0(S^3,\xi_{\mathrm{st}})$ such that $(a_j,b_j)$ converges to some $(a,b)$ in $U$ and the sequence of contact forms
\[
\alpha_j \coloneqq \varphi_j^* \varepsilon_{a_j,b_j}
\]
converges to some contact form $\alpha$ in $\mathcal{U}$. Assuming that $(a_j,b_j)$ and $\varphi_j$ have been chosen, we choose $\varphi_{j+1}$ of the form
\[
\varphi_{j+1} \coloneqq \psi_j \circ \varphi_j,
\]
for some $\psi_j$ in $\mathrm{Cont}_0(S^3,\xi_{\mathrm{st}})$ satisfying $\psi_j^* \varepsilon_{a_j,b_j}= \varepsilon_{a_j,b_j}$ whose properties are discussed below. With these choices, we have for every $(a',b')\in (0,+\infty)^2$
\begin{equation}
\label{general}
\begin{split}
& \varphi_{j+1}^* \varepsilon_{a',b'} - \varphi_j^* \varepsilon_{a_j,b_j} =  \varphi_{j+1}^* \varepsilon_{a',b'} - (\psi_j^{-1} \circ \varphi_{j+1})^*  \varepsilon_{a_j,b_j} \\ &= \varphi_{j+1}^* \varepsilon_{a',b'} - \varphi_{j+1}^* ( (\psi_j^{-1})^*   \varepsilon_{a_j,b_j} ) = \varphi_{j+1}^* (  \varepsilon_{a',b'} -\varepsilon_{a_j,b_j} ).
\end{split}
\end{equation}
Once $\psi_j$, and hence $\varphi_{j+1}$, has be chosen, the above identity shows that by choosing $(a_{j+1},b_{j+1})$ close to $(a_j,b_j)$ we can make the $C^k$-norm of
\[
\alpha_{j+1} - \alpha_j = \varphi_{j+1}^* (  \varepsilon_{a_{j+1},b_{j+1}} -\varepsilon_{a_j,b_j} )
\]
arbitrarily small. Therefore, we can ensure the convergence of $(a_j,b_j)$ to some $(a,b)\in U$ and of $(\alpha_j)$ to some $\alpha \in \mathcal{U}$ in the $C^{\infty}$-topology. Moreover, \eqref{general} shows that by making the convergence of $(a_j,b_j)$ fast enough and by a diagonal argument we can also ensure that the sequence $(\varphi_j^* \varepsilon_{a,b})$ converges to $\alpha$, as claimed in (i). 

Next, we remark that any open condition which is satisfied by $\alpha_j$ can be transmitted to $\alpha$ by making the convergence of $(a_j,b_j)$ sufficiently fast. More precisely: if $\mathcal{V}_j$ is some open neighborhood of $\alpha_j$ in $\mathcal{F}(S^3,\xi_{\mathrm{st}})$, we can fix a closed neighborhood $\mathcal{W}_j\subset \mathcal{V}_j$ of $\alpha_j$ and make all the subsequent choices of $(a_i,b_i)$ in such a way that for every $i>j$ the contact form $\alpha_i$ belongs to $\mathcal{W}_j$. By doing this, we can ensure that also the limiting contact form $\alpha$ belongs to $\mathcal{V}_j$.

We now discuss the choice of $\psi_j$. By choosing each $\psi_j$ to be supported away from the link $\Gamma\coloneqq \Gamma_1 \cup \Gamma_2$, we can ensure that $\alpha$ agrees up to every order with $\varepsilon_{a,b}$ on $\Gamma$, and hence $\Gamma_1$ and $\Gamma_2$ are closed Reeb orbits of $\alpha$ of period $a$ and $b$, respectively.

Write $\frac{a_j}{b_j}=\frac{p_j}{q_j}$, where $p_j$ and $q_j$ are coprime natural numbers. By choosing $(a_{j+1},b_{j+1})$ close to $(a_j,b_j)$, but not equal to it, we have that the sequences $(p_j)$ and $q_j)$ diverge. 

We claim that if the convergence of $(a_j,b_j)$  is sufficiently fast then $\Gamma_1$ and $\Gamma_2$ are the only closed Reeb orbits of $\alpha$. Let $V_j$ be a sequence of open neighborhoods of $\Gamma$ whose intersection is $\Gamma$. Since the Reeb orbits of $\alpha_j$ other than $\Gamma_1$ and $\Gamma_2$ are all closed with period $p_j b_j =  q_j a_j$, the Reeb flow of $\alpha_j$ has no closed orbits of period $\tau\leq \tau_j\coloneqq p_j b_j - 1$ which pass through $S^3 \setminus V_j$. The set of contact forms with the latter property is $C^1$-open, so by the remark above we can guarantee that the Reeb flow of $\alpha$ has no closed orbits of period $\tau\leq \tau_j$ which pass through $S^3 \setminus V_j$, for every $j\in \N$. Since the sequence $(V_j)$ shrinks to $\Gamma$ and  $(\tau_j)$ diverges, $\Gamma_1$ and $\Gamma_2$ are the only two closed Reeb orbits of $\alpha$, as claimed in (ii).	

Fix some infinitesimal sequence $(\epsilon_j)$ of positive numbers. We claim that by a suitable choice of the defining data we can ensure that for every $j \geq 1$ the Reeb flow of $\alpha_j$ has an orbit which is $\epsilon_j$-dense. Here, a subset of $S^3$ is said to be $\epsilon$-dense if it meets any open ball of radius $\epsilon$, with respect to some fixed metric $d$ on $S^3$.

Indeed, once this is achieved we can argue as follows. Having an $\epsilon_j$-dense orbit is a $C^0$-open condition in the space of smooth vector fields, and hence a $C^1$-open condition in $\mathcal{F}(S^3,\xi_{\mathrm{st}})$. Therefore, the above property of each $\alpha_j$ can be transmitted to $\alpha$, whose Reeb flow then has an $\epsilon_j$-dense orbit for every $j$. When non-empty, the set $O_j$ of $\epsilon_j$-dense orbits of a flow is open and $\epsilon_j$-dense. Using that $\epsilon_j\to 0$, a standard argument, which is analogous to the proof of Baire's theorem, shows that $\cap_{k\in\N}O_{j_k}$ is non-empty for a suitable subsequence $(j_k)$. Every element of this intersection has a dense orbit and, therefore, $\alpha$ satisfies (iii).

There remains to prove the following claim, in which $\epsilon$ is an arbitrary positive number: Once $\varphi_j$ and $(a_j,b_j)$ have been chosen, by a suitable choice of the diffeomorphism $\psi_j$, which we recall is required to preserve the contact form $\varepsilon_{a_j,b_j}$ and to be supported away from $\Gamma$, and by choosing $(a_{j+1},b_{j+1})$ sufficiently close to $(a_j,b_j)$, we can make sure that the Reeb flow of 
\[
\alpha_{j+1}= \varphi_j^*( \psi_j^* \varepsilon_{a_{j+1},b_{j+1}}) 
\]
has an orbit which is $\epsilon$-dense. The proof of this fact uses the following lemma.

\begin{lem}
\label{transitive}
Let $(a,b)$ be a pair of positive numbers with rational ratio. Let $(\gamma_1,\dots,\gamma_k)$ and $(\gamma_1',\dots,\gamma_k')$ be $k$-tuples of distinct Reeb orbits of $\varepsilon_{a,b}$ in $S^3\setminus \Gamma$. Then there exists $\psi\in \mathrm{Cont}_0(S^3,\xi_{\mathrm{st}})$ which preserves the contact form $\varepsilon_{a,b}$, is  supported away from $\Gamma$ and satisfies $\psi(\gamma_j)=\gamma_j'$ for every $j=1,\dots,k$.
\end{lem}

Postponing the proof of this lemma, we show how it implies the above claim. Note that the Reeb flow of $\alpha_{j+1}$ has the form
\begin{equation}
\label{composizione}
\phi_{\alpha_{j+1}}^t = \varphi_j^{-1} \circ \psi_j^{-1} \circ \phi^t_{\varepsilon_{a_{j+1},b_{j+1}}} \circ \psi_j \circ \varphi_j,
\end{equation}
where $\phi^t_{\varepsilon_{a_{j+1},b_{j+1}}}$ denotes the Reeb flow of $\varepsilon_{a_{j+1},b_{j+1}}$. Let $\delta>0$ be such that
\begin{equation}
\label{uc1}
d(z,z') < \delta \qquad \Rightarrow \qquad d \bigl( \varphi_j^{-1}(z),\varphi_j^{-1}(z')) < \epsilon,
\end{equation}
and let $A\subset S^3\setminus \Gamma$ be a finite set which is $\frac{\delta}{2}$-dense in $S^3$. Fix some invariant torus $T$ in $S^3\setminus \Gamma$, such as for instance
\[
T \coloneqq \bigl\{ (z_1,z_2)\in S^3 \mid |z_1|^2 = |z_2|^2 = \textstyle{\frac{1}{2}} \bigr\}.
\]
Since this torus consists of infinitely many Reeb orbits of $\varepsilon_{a_j,b_j}$, by the above lemma there exists $\psi_j \in \mathrm{Cont}_0(S^3,\xi_{\mathrm{st}})$ which  preserves  $\varepsilon_{a_j,b_j}$ and is supported away from $\Gamma$ such that $\psi_j^{-1}(T)$ contains the finite set $A$, and is therefore $\frac{\delta}{2}$-dense in $S^3$.

Now let $\sigma>0$ be such that
\begin{equation}
\label{uc2}
d(z,z') < \sigma \qquad \Rightarrow \qquad d \bigl( \psi_j^{-1}(z),\psi_j^{-1}(z')) < \textstyle{\frac{\delta}{2}}.
\end{equation}
Let $(a',b')$ be a pair of positive numbers with rational ratio $\frac{a'}{b'} = \frac{p}{q}$, where $p$ and $q$ are coprime natural numbers. Each Reeb orbit of $\varepsilon_{a',b'}$ in $T$ is $\sigma$-dense, provided that $p$ and $q$ are large enough. Therefore, by choosing $(a_{j+1},b_{j+1})$ with the already required properties and sufficiently close to $(a_j,b_j)$, we obtain that each Reeb orbit of $\varepsilon_{a_{j+1},b_{j+1}}$ is $\sigma$-dense in $T$.

Let $z$ be a point in $\varphi_j^{-1}\circ \psi_j^{-1}(T)$ and set $w\coloneqq \psi_j\circ \varphi_j(z)\in T$. Then the orbit of $w$ by the Reeb flow of $\varepsilon_{a_{j+1},b_{j+1}}$ is $\sigma$-dense in $T$.  By \eqref{uc2}, the image of this orbit by $\psi_j^{-1}$ is $\frac{\delta}{2}$-dense in $\psi_j^{-1}(T)$ and hence $\delta$-dense in $S^3$. Together with \eqref{composizione} and \eqref{uc1}, we conclude that the orbit of $z$ by the Reeb flow of $\alpha_{j+1}$ is $\epsilon$-dense. This concludes the proof of Theorem \ref{katok}.

\begin{proof}[Proof of Lemma \ref{transitive}]
Since $\frac{a}{b}$ is rational, the Reeb flow of $\varepsilon_{a,b}$ defines a free $S^1$-action on $S^3\setminus \Gamma$. The orbit space of this action is a smooth surface $\Sigma$ which is endowed with a symplectic form $\omega$ satisfying $\pi^* \omega = d \varepsilon_{a,b}$, where 
\[
\pi: S^3 \setminus \Gamma \rightarrow \Sigma
\]
denotes the quotient projection. Let $(p_1,\dots,p_k)$ and $(p_1',\dots,p_k')$ be the $k$-tuples of distinct points in $\Sigma$ such that 
\[
\gamma_j = \pi^{-1}(p_j), \qquad \gamma_j' = \pi^{-1}(p_j') \qquad \forall j=1,\dots,k.
\]
It is easy to show that the group of compactly supported Hamiltonian diffeomorphisms of $(\Sigma,\omega)$ acts $k$-transitively: There exists a smooth function $H: [0,1]\times \Sigma \rightarrow \R$ with compact support such that the time-1 map of the induced Hamiltonian flow maps $p_j$ to $p_j'$, for every $j=1,\dots,k$ (see e.g.\ \cite[p.\ 109]{ban97}). Lift $H$ to a smooth function $K:[0,1]\times S^3\rightarrow \R$ supported away from $\Gamma$ and let $X$ be the contact vector field on $S^3$ which is induced by the Hamiltonian $K$ and the contact form $\varepsilon_{a,b}$, i.e., the unique vector field satisfying
\[
\imath_X d\varepsilon_{a,b} = \bigl( \imath_{R_{\varepsilon_{a,b}}} dK) \varepsilon_{a,b} - dK, \qquad \imath_X \varepsilon_{a,b}= K,
\]
where $R_{\varepsilon_{a,b}}$ denote the Reeb vector field of $\varepsilon_{a,b}$.
The flow of $X$ consists of contactomorphisms of $(S^3,\xi_{\mathrm{st}})$ preserving $\varepsilon_{a,b}$ which are supported away from $\Gamma$ and whose time-1 map $\psi$ maps $\gamma_j$ to $\gamma_j'$, for every $j=1,\dots,k$.
\end{proof}

\paragraph{Proof of theorem \ref{eliashberg-hofer}.} By composition with the radial projections, a diffeomorphism from $\partial E(a,b)$ to $\partial A$ which intertwines the restrictions of $\lambda_0$ would induce a smooth conjugacy between $\varepsilon_{a,b}$ and $\alpha$ on $S^3$, which cannot exist because the Reeb flow of $\alpha$ has a dense orbit while the one of $\varepsilon_{a,b}$ does not. So we only have to prove that $E(a,b)$ and $A$ are symplectomorphic.

Set for simplicity $E\coloneqq E(a,b)$. Write $\varphi_j^* \varepsilon_{a,b} = f_j^2 \, \alpha_0$ and set
\[
A_j \coloneqq \{rz \mid z\in S^3, \; 0\leq r < f_j(z) \}.
\]
By composing the contactomorphism $\varphi_j^{-1}: S^3 \rightarrow S^3$, which isisotopic to the identity, with the radial projections $\partial E \rightarrow S^3$ and $S^3 \rightarrow \partial A_j$, we obtain a diffeomorphism from $\partial E$ to $\partial A_j$ which intertwines the restrictions of $\lambda_0$. The positively 1-homogeneous extension of this map is a symplectomorphism of $\C^2\setminus \{0\}$ isotopic to the identity via such maps, and mapping $r E\setminus \{0\}$ to $r A_j\setminus \{0\}$, for every $r>0$. Repeating the argument at the end of the proof of Theorem \ref{1stcap} (ii) in Section
\ref{proof-of-thm-1}, this map can be smoothed near the origin producing a global symplectomorphism $\phi_j : \C^2 \rightarrow \C^2$ such that
\[
\phi_j( r E) = r A_j \qquad \forall r\geq \textstyle{\frac{1}{2}}.
\]
The $C^0$-convergence of $(f_j)$ to $f$ implies that, up to passing to a subsequence, we can find a strictly increasing sequence $(r_j)$ in the interval $( \frac{1}{2},1)$ which converges to 1 and satisfies
\begin{equation}
\label{mp2}
r_j \overline{A}_j \subset r_{j+1} A_{j+1}, \qquad \bigcup_{j\geq 0} r_j A_j = A.
\end{equation}
Starting with $\psi_0\coloneqq \phi_0$ and arguing inductively, we wish to construct symplectomorphisms $\psi_j :\C^2 \rightarrow \C^2$ such that for every $j\geq 0$:
\begin{eqnarray}
\label{wish1}
&\psi_{j+1} = \psi_j  \quad  & \mbox{on } r_j E,\\
\label{wish2}
&\psi_{j+1} = \phi_{j+1} \quad & \mbox{on } \C^2 \setminus r_{j+1} E.
\end{eqnarray}
Once $\psi_0,\dots,\psi_j$ have been constructed, the construction of $\psi_{j+1}$ goes as follows. The symplectomorphism $\theta\coloneqq \phi_{j+1}^{-1}\circ \psi_j$ maps $r_j E$ to the domain $\phi_{j+1}^{-1}(r_j A_j)$, whose closure satisfies
\[
\overline{\phi_{j+1}^{-1}(r_j A_j)} = \phi_{j+1}^{-1}(r_j \overline{A_j}) \subset \phi_{j+1}^{-1}(r_{j+1} A_{j+1}) = r_{j+1} E.
\]
Since the space of symplectic embeddings between ellipsoids is connected, see \cite{mcd09}, we can find a symplectic isotopy $\theta_t:\C^2 \rightarrow \C^2$ such that $\theta_0=\mathrm{id}$, $\theta_1= \theta$ and
\[
\theta_t(r_j \overline{E}) \subset r_{j+1} E \qquad \forall t\in [0,1].
\]
Let $H_t: \C^2 \rightarrow \R$ be a smooth path of Hamiltonians generating $\theta_t$. Let $\chi:\C^2 \rightarrow \R$ be a smooth function with support in $r_{j+1} E$ and taking the value 1 on $\bigcup_{t\in [0,1]} \theta_t(r_j \overline{E})$. Denote by $\eta:\C^2 \rightarrow \C^2$ the time-1 map of the Hamiltonian isotopy induced by the non-autonomous Hamiltonian $\chi H$. Then
\[
\eta = \theta_1 =   \phi_{j+1}^{-1}\circ \psi_j \quad \mbox{on } r_j E, \qquad \eta = \mathrm{id} \quad \mbox{on } \C^2 \setminus r_{j+1} E,
\]
and hence $\psi_{j+1}\coloneqq \phi_{j+1}\circ \eta$ satisfies \eqref{wish1} and \eqref{wish2}.

Note that by \eqref{wish2} we have
\begin{equation}
\label{inpart}
\psi_{j}(rE) = r A_{j} \quad \forall r\geq r_{j}, \qquad \forall j\geq 0.
\end{equation}
By \eqref{wish1}, the sequence $\psi_j(z)$ stabilizes for every $z\in E$, and hence the map
\[
\psi:E \rightarrow \C^2, \qquad \psi(z) \coloneqq \lim_{j\rightarrow \infty} \psi_j(z)
\]
is a well defined symplectic embedding, which by \eqref{mp2} and \eqref{inpart} has image $A$. The maps $\psi$ is then the required symplectomorphism from $E(a,b)$ to $A$.


\newcommand{\etalchar}[1]{$^{#1}$}
\providecommand{\bysame}{\leavevmode\hbox to3em{\hrulefill}\thinspace}
\providecommand{\MR}{\relax\ifhmode\unskip\space\fi MR }
\providecommand{\MRhref}[2]{%
  \href{http://www.ams.org/mathscinet-getitem?mr=#1}{#2}
}
\providecommand{\href}[2]{#2}

\end{document}